\documentclass[a4paper,10pt]{article}

\usepackage{amsmath,amsfonts,amsthm,amssymb,graphicx,bm}
\usepackage{enumerate}
\usepackage{pdfsync}
\usepackage[active]{srcltx}
\usepackage{geometry}
\usepackage{url}

\usepackage[applemac]{inputenc}

\usepackage{bbm}

\newtheorem{theorem}{Theorem}[section]
\newtheorem{definition}[theorem]{Definition}
\newtheorem{proposition}[theorem]{Proposition}
\newtheorem{corollary}[theorem]{Corollary}
\newtheorem{lemma}[theorem]{Lemma}
\newtheorem{remark}[theorem]{Remark}

\numberwithin{equation}{section}

\newcommand{\RR}{\mathbb{R}}

\newcommand{\eps}{\epsilon}

\newcommand{\spa}{\mathrm{span}\,}

\newcommand{\sign}{\mathrm{sign}\,}

\newcommand{\al}{\alpha}
\newcommand{\lb}{\lambda}

\newcommand{\R}{{\mathbb {R}}}

\newcommand{\la}{\left\langle}
\newcommand{\ra}{\right\rangle}

\newcommand{\m}{m}
\newcommand{\mS}{\mathcal{S}}
\newcommand{\mZ}{\mathcal{Z}}
\newcommand{\mT}{\mathcal{T}}
\newcommand{\mB}{\mathcal{B}}

\newcommand{\beq}{\begin{equation}}
\newcommand{\eeq}{\end{equation}}
\newcommand{\f}{\frac}
\newcommand{\dis}{\displaystyle}

\title{Optimal growth for linear processes with affine control}

\author{\textsc{Vincent Calvez}\thanks{Ecole Normale Sup\'erieure de Lyon, UMR CNRS 5669 'UMPA', and Inria Rh\^one-Alpes, projet NUMED, 46 all\'ee d'Italie, F-69364~Lyon~cedex~07, France.
Email: \texttt{vincent.calvez@ens-lyon.fr}} \and
\textsc{Pierre Gabriel}\thanks{Inria Rh\^one-Alpes, \'equipe-projet BEAGLE, BP~52132, 66~Boulevard Niels Bohr, F-69603~Villeurbanne~cedex, France. Email: \texttt{pierre.gabriel@inria.fr}}\;\thanks{Corresponding author}}

\begin{document}

\maketitle

\begin{abstract}
We analyse an optimal control with the following features: the dynamical system is linear, and the dependence upon the control parameter is affine.
More precisely we consider $\dot x_\alpha(t) = (G + \alpha(t) F)x_\alpha(t)$, where $G$ and $F$ are $3\times 3$ matrices with some prescribed structure.
In the case of constant control $\alpha(t)\equiv \alpha$, we show the existence of an optimal Perron eigenvalue with respect to varying $\alpha$ under some assumptions.
Next we investigate the Floquet eigenvalue problem associated to  time-periodic controls $\alpha(t)$.
Finally we prove the existence of an eigenvalue (in the generalized sense) for the  optimal control problem. The proof is based on the results by [Arisawa 1998, Ann. Institut Henri Poincar\'e] concerning the ergodic problem for Hamilton-Jacobi equations. 
We discuss the relations between the three eigenvalues. Surprisingly enough, the three eigenvalues appear to be numerically the same.\bigskip

\end{abstract}

\section{Introduction}

We aim at optimally controlling the following 3-dimensional system in the long-time horizon,
\begin{equation}\label{eq:dynsyst}\left\{\begin{array}{l}
\dot x_\alpha (t) = (G+\al(t) F)\,x_\alpha(t)\, , 
\vspace{1mm}\\
x_\alpha(0) = x \in \left(\RR_+\right)^3\setminus\{0\}\,,
\end{array}\right.\end{equation}
where $G,F\in \mathcal M_{3}(\R)$  are matrices with nonnegative off-diagonal entries, and $\alpha:\RR_+\to\RR_+ $ is a nonnegative control parameter.
There is no running reward. Let $T>0$ be the final time.
The final reward is the linear function $\la m, x_\alpha(T)\ra$, where $m\in \left(\RR_+\right)^3\setminus\{0\}$ is in the kernel of $F^T$: $m^T F = 0$. 

We investigate the asymptotic behaviour of the optimal reward when $T\to+\infty$ in the three different regimes: $\alpha$ is constant (constant control), $\alpha$ is periodic (periodic control), and $\alpha$ is any measurable function from $(0,T)$ to $[a,A]$, where $a,A$ are given bounds (optimal control).
 
Because of the linear structure of system \eqref{eq:dynsyst}, we expect an exponential growth of the final reward. This motivates the introduction of the following renormalized reward:
\[ r_\al(T,x) =\frac1T\log\langle m,x_\alpha (T)\rangle\, .\]
When $\al$ is a constant (resp. periodic) control, it is well-known that  
$r_\al(T,x)$ converges to the Perron (resp. Floquet) eigenvalue $\lambda_P(\alpha)$ (resp. $\lambda_F(\alpha)$) of system~\eqref{eq:dynsyst} as $T\to +\infty$. Hence, the best control is obtained by optimizing the  Perron (resp. Floquet) eigenvalue.
In the case of optimal control, we define the best possible reward as follows,
\begin{equation*}
V(T,x) =  \sup_\alpha \left\{ \la m, x_\alpha(T)\ra \;:\;\dot x_\alpha(t) = (G + \alpha(t) F) x_\alpha(t)\;,\; x_\alpha(0) = x  \right\}\, ,
\end{equation*}
where the supremum is taken over all measurable control functions $\alpha: (0,T)\to [a,A]$. We can resolve this optimal control problem using the dynamic programming principle. This yields a Hamilton-Jacobi equation for the value function \[v(s,x) = \sup \left\{ \la m, x_\alpha(T)\ra \;:\;\dot x_\alpha(t) = (G + \alpha(t) F) x_\alpha(t)\;,\; x_\alpha(s) = x  \right\}\,,\] 
defined for intermediate times $0< s < T$. Namely it is solution to the following Hamilton-Jacobi-Bellman equation in the sense of viscosity solutions \cite{Bardi-Capuzzo},
\[ \frac{\partial}{\partial s} v (s,x) + \widetilde H(x, D_x v(s,x)) = 0 \, ,
 \]
where the Hamiltonian is given by $ \widetilde H(x,p) = \max_{\alpha\in [a,A]} \left\{ \la(G + \alpha F)x , p \ra  \right\}$ (see Section~\ref{sec:HJB} for details). 

Under some assumptions we prove the following ergodic result: there exists a constant $\lambda_{HJ}$ such that,
\begin{equation} \forall x \in \left(\RR_+\right)^3\setminus\{0\}\, \quad  \lim_{T\to +\infty}\dfrac{1}{T}\log V(T,x) = \lambda_{HJ}\, . \label{eq:ergodic}\end{equation} 
This constant $\lambda_{HJ}$ is the analog of the Perron and Floquet eigenvalues in the case of optimal control. We have obviously,
\begin{equation*} 
\sup_{\alpha} \lambda_{P}(\alpha) \leq \sup_{\alpha} \lambda_F(\alpha) \leq \lambda_{HJ}\, . 
\end{equation*} 
Interestingly, numerical simulations show that these three eigenvalues may coincide (see Section~\ref{sec:num} for a discussion).

\subsubsection*{Motivations and running example}

From a theoretical viewpoint, the convergence result \eqref{eq:ergodic} is known as the ergodic problem for Hamilton-Jacobi equations. It appears in homogenization problems \cite{LPV}. In this case the constant $\lambda$ is called the effective hamiltonian  \cite{Barles-Evans-Souganidis,Evans-Gomes}. It can be interpreted as a nonlinear Perron eigenproblem associated with an eigenvector $\overline{u}$ which solves the following stationary Hamilton-Jacobi equation in the viscosity sense,
\begin{equation}\label{eq:eigenproblem} - \lambda + H(y,D_y \overline{u}) = 0\, , \quad y \in Y \end{equation}
where $H: Y\times \RR^n \to \RR_+ $ denotes the hamiltonian.
It also appears in weak KAM theory for lagrangian dynamical systems \cite{Fathi,Fathi-book}. A natural way to attack this issue is to consider the following stationary Hamilton-Jacobi equation with small parameter $\eps>0$,
\[ - \eps u_\eps + H(x,D_x u_\eps) = 0\, . \]
The asymptotic behaviour of $u_\eps$, as $\eps\to 0$, has been investigated in several works \cite{LPV,CDL,Fathi,Namah-Roquejoffre,Barles-Souganidis,Barles-Roquejoffre,Imbert-Monneau,Barles-2008,Alvarez-Bardi,Cardaliaguet}.
In many cases the set $Y$ is assumed to be compact, and the Hamiltonian $H(y,p)$ is assumed to be coercive: $H(\cdot,p) \to +\infty$ as $|p|\to +\infty$. Under these assumptions it can be proven that the function $\eps u_\eps$ converges to a constant $\lambda$ (uniquely determined) \cite{LPV,CDL,Fathi}. Moreover the function $u_\eps$ converges uniformly, up to extraction, to some lipschitz function $\overline{u}$, solution of \eqref{eq:eigenproblem}. However the function $\overline{u}$ is generally not unique. The question of convergence of $u_\eps$ towards $\overline{u}$ (modulo a large constant) has been investigated in \cite{Namah-Roquejoffre,Barles-Souganidis,Barles-Roquejoffre}.

In the context of optimal control, the coercivity of the hamiltonian $H(y,p) = \max_{\alpha \in \mathcal A} \{ \la b(y,\alpha), p\ra \}$ is guaranteed under the hypothesis of uniform controllability \cite{CDL,Bardi-Capuzzo}: there exists a constant $\mu>0$ such that 
\begin{equation*} \forall y \in \overline{Y}\quad  B(0,\mu) \subset \overline{\mathrm{convex\; hull}}\{ b(y,\alpha)\; |\; \alpha \in \mathcal A \}\,.  
\end{equation*}
This hypothesis ensures that any two points $y,y'\in Y$ can be connected with some control $\alpha(t)$ within a time $T = O(|y - y'|)$. This yields equicontinuity of the family $(u_\eps)_{\eps>0}$, and thus compactness.

The criterion of uniform controllability is not verified in our case, hence the hamiltonian is not coercive (see Remark \ref{rem:not coercive} below). Different approaches have been developped to circumvent the lack of coercivity. Several works rely on some partial coercivity \cite{Imbert-Monneau,Barles-2008,Alvarez-Bardi}. In \cite{Arisawa-Lions} the authors introduce a non-resonance condition which yields ergodicity. This condition is restricted to hamiltonian with separated variables $H(y,p) = H(p) + V(y)$. In \cite{Cardaliaguet} the author extends this result to the case of a non-convex hamiltonian $H(p)$, in two dimensions of space. In the context of optimal control, Arisawa \cite{Arisawa1,Arisawa2} has shown the equivalence between ergodicity  and the existence of a stable subset $Z\subset Y$ which attracts the trajectories. It is required in addition that the restriction of the system to $Z$ is controllable. The present work follows the latter approach. 
\medskip


From a modeling viewpoint, this work is motivated by the optimization of some experimental protocol for polymer amplification, called Protein Misfolding Cyclic Amplification (PMCA) \cite{Soto}. In the system \eqref{eq:dynsyst} the matrix $G$ represents the growth of polymers in size, whereas $F$ is the fragmentation of polymers into smaller pieces. The vector $m$ encodes the size of the polymers. We restrict to dimension 3 for the sake of simplicity, {\em i.e.} three possible sizes for the polymers (small, intermediate, large).  The orthogonality relation $m^TF = 0$ accounts for the conservation of the total size of polymers by fragmentation. 

We now give a class of matrices which will serve as an example all along the paper. It is a simplification of the discrete growth-fragmentation process introduced in \cite{Masel} to model Prion proliferation. We make the following choice:
\beq\label{eq:example}G = \left(\begin{array}{ccc}
-\tau_{1}& 0 & 0 \\
\tau_{1} &-\tau_{2}& 0 \\
0 & \tau_{2} & 0 
\end{array}\right) \qquad\text{and}\qquad
F = \left(\begin{array}{ccc}
0 & 2\beta_2 & \beta_3 \\
0 & -\beta_2 & \beta_3 \\
0 & 0 & -\beta_3
\end{array}\right).\eeq
Here $\tau_1>0$ denotes the rate of increase from small to intermediate polymers, and $\tau_2>0$ from intermediate to large polymers. The $\beta_i>0$ denote the fragmentation rates which distribute larger polymers into smaller compartments, keeping the total size constant. The size vector is $m = (1\; 2 \; 3)^T$. Finally, the rate $\alpha(t)$ accounts for \emph{sonication}, {\em i.e.} externally driven intensity of fragmentation.

The continuous version of the  baby model \eqref{eq:dynsyst}-\eqref{eq:example} consists in the following size-structured PDE with variable coefficients~\cite{Greer,CL1,DG}
\beq
\label{eq:pde}
\partial_tf(t,\xi) + \partial_\xi(\tau(\xi)f(t,\xi)) = \alpha(t) \left( 2 \int_\xi^\infty\beta(\zeta)\kappa(\xi,\zeta)f(t,\zeta)\,d\zeta - \beta(\xi)f(t,\xi)\right)\, .
\eeq
Here, $f(t,\xi)$ represents the density of protein polymers of size $\xi>0$ at time $t$.
The transport term accounts for the growth of  polymers in size, whereas the r.h.s is the fragmentation operator.
The final reward is the total mass of polymers, namely $\int_0^\infty\xi u(T,\xi)\,d\xi$ (see~\cite{GabrielPhD} for more details).
The Perron eigenvalue problem for \eqref{eq:pde} has been investigated in \cite{DG,CDG}.

\subsubsection*{Main assumptions}

\begin{description}
\item[(H1)] We assume that the matrices $G$ and $F$ are both reducible. We assume that the matrix $G+\alpha F$ is irreducible for all $\alpha>0$. 
\end{description}
From the Perron-Frobenius Theorem there exists a simple dominant eigenvalue $\lambda_P(\alpha)$ and a positive left- (resp. right-) eigenvector $e_\alpha$ (resp. $\phi_\alpha$) such that:
\beq\label{def:eigenelements}
\left\{\begin{array}{l}
(G+\al F)e_\alpha=\lb_P(\alpha) e_\alpha,\quad e_\alpha>0\,,
\medskip\\
\phi_\alpha^T(G+\al F)=\lb_P(\alpha)\phi_\alpha^T,\quad\phi_\alpha>0\,.               
\end{array}\right.\eeq
We choose the following normalizations:
\[ \la m , e_\alpha \ra=1\,,\quad \la \phi_\alpha, e_\alpha\ra =1\,.\]
We denote by $(\lambda_1(\alpha), \lambda_2(\alpha), \lambda_3(\alpha))$ the eigenvalues of the matrix $G + \alpha F$, where $\lambda_1(\alpha) \in \R$ is the dominant eigenvalue $\lambda_P(\alpha)$.
We shall use repeatedly the spectral gap property of $G+\alpha F$. We denote $\mu_\alpha = \min_{i = 2,3}(\lambda_1 - \Re(\lambda_i))$, where $(\lambda_i)_{i = 2,3}$ are the two other (possibly complex) eigenvalues of $G+ \alpha F$. By continuity and compactness, $\mu_\alpha$ is uniformly strictly positive on compact intervals of $(0,+\infty)$. 
\begin{description}
\item[(H2)] We assume that the Perron eigenvalue $\lambda_P(\alpha)$ is bounded admits a global maximum as $\alpha\in (0,+\infty)$. This maximum is attained at $\alpha = \alpha^*$. We assume that the bounds $a,A$ are such that $a< \alpha^*< A$.
\end{description}
We will show in the next Section that this assumption is satisfied for the example \eqref{eq:example} provided that $\tau_2>2\tau_1$. This condition can be justified heuristically \cite{CDG}. In fact, when growth of intermediate polymers is fast, it is interesting to have a significant fraction of intermediate polymers in the population in order to optimally increase the total size of the population. On the contrary, when $\alpha\to +\infty$ then the  eigenvector $e_\alpha$ converges towards $(1\; 0\; 0)^T$ because fragmentation is very large. 

The matrix $G$ possesses a nonnegative eigenvector $e_0 = \lim_{\alpha\to 0} e_\alpha$, associated to a dominant eigenvalue $\lambda_P(0)$. 
On the other hand, from Hypothesis {\bf (H2)} and  \eqref{def:eigenelements} we deduce  that $e_\alpha$ converges to a nonnegative eigenvector $e_\infty$ as $\alpha\to +\infty$, with $e_\infty\in \ker F$.

\begin{description}
\item[(H3)] We assume that both the eigenvectors $e_0$ and $e_\infty$ have at least one zero coordinate. 
\end{description}
This Hypothesis is compatible with the reducibility of both $G$ and $F$. However since $G+\alpha F$ is irreducible for $\alpha>0$, we have $e_0\neq e_\infty$. In fact $e_\alpha>0$ is the unique nonnegative eigenvector up to a multiplicative constant. Furthermore, $e_\infty$ is not an eigenvector for $G$. In particular we have $Ge_\infty\neq 0$.
In the case of the running example \eqref{eq:example}, we have $e_0 = (0\; 0\; 1/3)^T$ and $e_\infty = (1\; 0\; 0)^T$.

\begin{description}
\item[(H4)-(H5)] We assume two technical conditions, related to the dynamics of the trajectories of \eqref{eq:dynsyst} projected on the simplex $\mS = \{ y\geq 0\;:\; \la m,y \ra = 1 \}$. We refer to Section \ref{ssec:result} for the statement of these conditions.
\end{description}
These two technical conditions are satisfied for the running example \eqref{eq:example}. 
It is not clear whether these two conditions are necessary or not for the validity of our result. We refer to Section \ref{sec:num} for a discussion.


\bigskip 

In Section~\ref{sec:PerronFloquet}, we give conditions on the running example to ensure the existence of an optimal constant control which maximizes the Perron eigenvalue.
Then we consider periodic controls and we investigate the variations of the Floquet eigenvalue around this best constant control.
In Section~\ref{sec:HJB}, we turn to the full optimal control problem for which the control $\al(t)$ is any measurable function taking values in  $[a,A].$
We prove the main result of this paper which is the convergence of the best reward when $T\to\infty$ \eqref{eq:ergodic}.
The technique consists in solving an ergodic problem for the Hamilton-Jacobi-Bellman equation.
Finally we give numerical evidences that the ergodicity constant $\lambda_{HJ}$ coincides with the best Perron eigenvalue in the case of the running example \eqref{eq:example}. Finally we discuss some  possible ways to remove each of the assumptions {\bf (H1-5)}.

\section{Optimizing Perron and Floquet eigenvalues}\label{sec:PerronFloquet}

We begin with a constant control $\alpha>0$. 
Since the eigenvalue $\lambda_P(\alpha)$ is simple, we have the following asymptotic behaviour:
%
\begin{equation}
\lim_{ t\to +\infty}x_\alpha(t)e^{-\lb_P(\alpha)t} = \langle \phi_\alpha,x \rangle e_\alpha\,. 
\label{eq:perron taylor}
\end{equation}
Plugging this convergence property into our optimization problem, we obtain the expansion
\[r_\al(T,x)=\lb_P(\al)+\frac{1}{T} \log\langle  \phi_\alpha,x \rangle +o\left(\frac1T\right)\,. \]
Thus in the class of constant controls, our optimization problem reduces to maximizing the Perron eigenvalue.
The following Proposition gives an answer to this problem in the case of the running example~\eqref{eq:example} (see also Figure~\ref{fig:Perron}).

\begin{figure}
\begin{center}
\includegraphics[width = 0.6\linewidth]{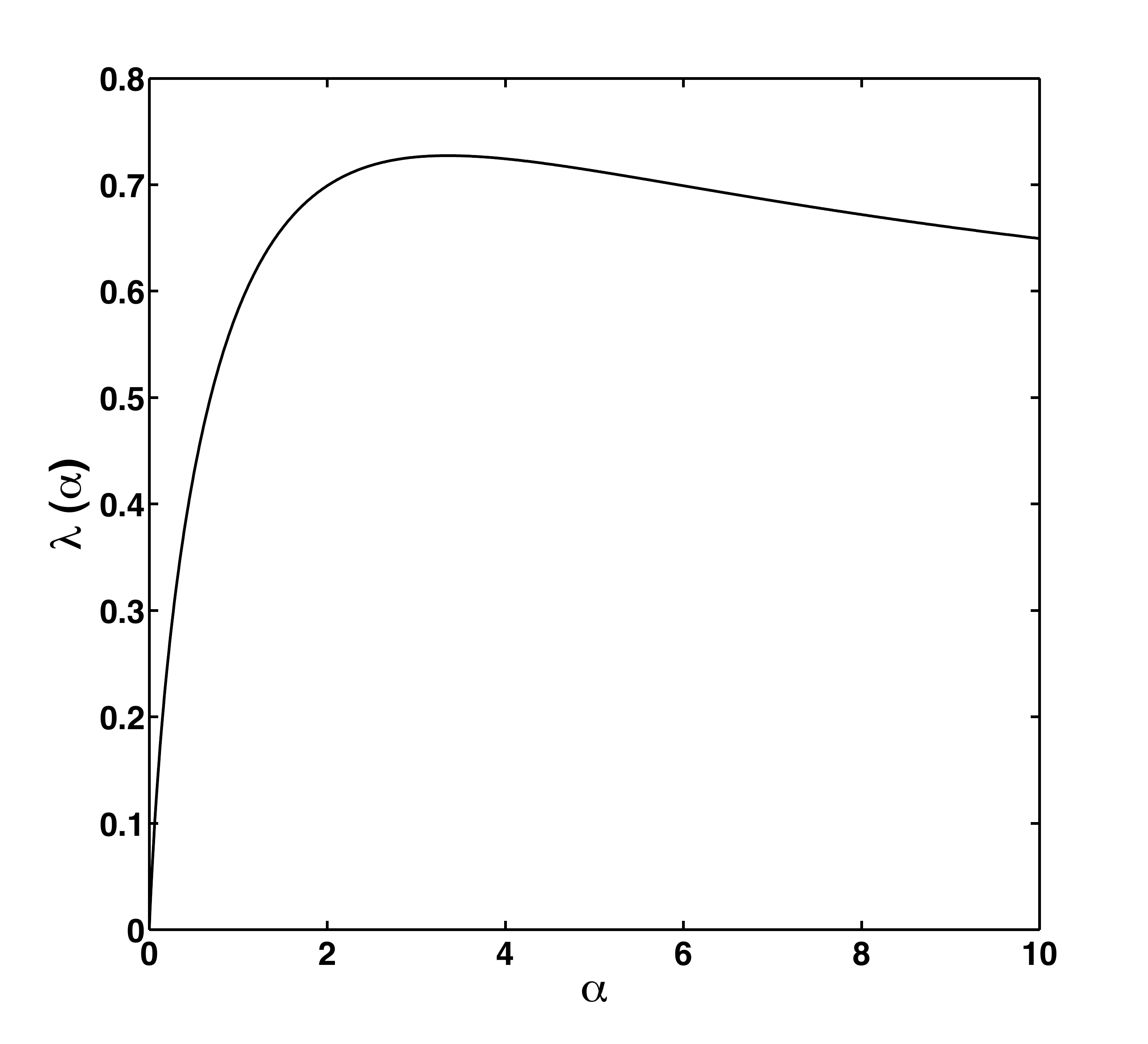}
\caption{Dominant eigenvalue of the matrix $G+\alpha F$ as a function of the control parameter $\alpha$. The function achieves a maximum for $\alpha = \alpha^* \in (0,+\infty) $.}\label{fig:Perron}
\end{center}
\end{figure}

\begin{proposition}\label{prop:n=3}
There exists a maximal eigenvalue $\lambda_P(\alpha^*)$ for some $\alpha^* \in (0,+\infty)$ if and only if $\tau_2>2\tau_1$. 
Furthermore we have 
the following alternative:
\begin{itemize}
\item either $\tau_2\leq 2\tau_1$ and $\lb_P(\al)$ increases from $0$ to $\tau_1,$
\item or $\tau_2>2\tau_1$ and $\lb_P(\al)$ first increases from $0$ to $\lb_P(\al^*)$ and then decreases to $\tau_1.$
\end{itemize}
\end{proposition}

\begin{proof}
The characteristic polynomial of $G+\al F$ is
\[P(X)=X^3+\left(\tau_1+\tau_2+\al(\beta_2+\beta_3)\right)X^2+\left(\tau_1\tau_2+\al\tau_1(\beta_3-\beta_2)+\al^2\beta_2\beta_3\right)X-\al\tau_1\tau_2\beta_3-\al^2\tau_1\beta_2\beta_3\,. \]
The Perron eigenvalue $\lb_P(\alpha)$ is the largest zero of this polynomial, so $P(\lb_P(\al))=0$ and, dividing by $\al^2$ we get that that $\lim_{\al\to+\infty}\lb_P(\al)=\tau_1.$ We easily check that $P'(\lambda)>0$ for $\lambda>\tau_1$, and we evaluate
\[
P(\tau_1)=2\tau_1^2(\tau_1+\tau_2)+\al\tau_1\beta_3(2\tau_1-\tau_2)\, .
\]
We have the following alternative: either $\tau_2<2\tau_1$ and $P(\tau_1)>0$ for all $\alpha>0$, or $\tau_2>2\tau_1$ and $P(\tau_1)<0$ for large $\alpha$.
In the former case, we have $\lambda_P(\alpha)<\tau_1$ for all $\alpha>0$. In the latter case, we have $\lambda_P(\alpha)>\tau_1$ for large $\alpha$.
On the other hand we have $\lambda_P(0) = 0$. Therefore, the condition $\tau_2>2\tau_1$ is necessary and sufficient for $\lambda_P$ to reach a global maximum.

Differentiating the relation $P(\lb_P(\al))=0,$ we get that any critical point $\alpha$ such that $\lambda_P'(\al)=0$ satisfies
\begin{equation} \label{eq:lpa} (\beta_2 + \beta_3)\lb_P(\al)^2 + \tau_1(\beta_3 - \beta_2) \lb_P(\al) + 2\alpha \beta_2 \beta_3 \lb_P(\al) = \tau_1 \tau_2 \beta_2 \beta_3 + 2 \alpha \tau_1 \beta_2 \beta_3\, . \end{equation}
Differentiating twice the relation $P(\lb_P(\al))=0,$ we get that any such critical point $\alpha$ satisfies
$$\lb_P''(\al)=\f{2\beta_2\beta_3(\tau_1-\lb_P(\alpha))}{3\lb_P(\alpha)^2+2(\tau_1+\tau_2+\al(\beta_2+\beta_3))\lb_P(\alpha)+\tau_1\tau_2+\al\tau_1(\beta_3-\beta_2)+\al^2\beta_2\beta_3}.$$
We claim that the denominator is always positive, despite the possible negative terme $\al\tau_1(\beta_3-\beta_2)$. In fact we can factorize by $\alpha\beta_2$: $\alpha\beta_2 \left(  \lb_P(\alpha) - \tau_1\right)\geq 0$ if $\lb_P(\al)\geq\tau_1$. In the other case $\lb_P(\al)\leq\tau_1$, we use the relation \eqref{eq:lpa} to get
\begin{align*}
2\al(\beta_2+\beta_3)\lb_P(\alpha)+\al\tau_1(\beta_3-\beta_2)+\al^2\beta_2\beta_3 
&\geq \frac\al{\lb_P(\alpha)}\left( \tau_1 \tau_2 \beta_2 \beta_3 + 2 \alpha \tau_1 \beta_2 \beta_3 \right) - \alpha^2 \beta_2 \beta_3 \\
& \geq  \frac{\al^2}{\lb_P(\alpha)}\beta_2 \beta_3\left( 2 \tau_1 - \lb_P(\alpha)\right) \geq 0\, . 
\end{align*}
In conclusion, $\lb_P''(\al)$ has the same sign as $\tau_1-\lb_P(\al)$ so $\lb_P(\al)$ can be a local minimum only if $\lb_P(\al)\leq\tau_1$
and a local maximum only if $\lb_P(\al)\geq\tau_1.$
The alternative announced in the proposition follows.
\end{proof}

Next we consider a periodic control $\al(t)$ with period $\theta>0$.
There exists a Floquet eigenvalue $\lb_F(\al)$ and periodic eigenvectors $e_\alpha(t)$, $\phi_\alpha(t)$ such that
\[
\left\{\begin{array}{l}
\dis\f{d}{dt}e_\alpha(t) + \lb_F(\alpha)e_\alpha(t)=(G+\al(t)F)e_\alpha(t)\,,
\medskip\\
\dis\f{d}{dt}\phi_\alpha(t) + \lb_F(\alpha)\phi_\alpha(t)=\phi_\alpha(t)(G+\al(t)F)\,.      
\end{array}\right.
\]
These eigenfunctions are unique after normalization,
\[ \dfrac1\theta \int_0^\theta \la m, e_\alpha(t) \ra \,dt=1\, ,\quad \dfrac1\theta\int_0^\theta\la \phi_\alpha(t)e_\alpha(t)\ra\,dt=1\, .\]
Again we obtain the following expansion for the payoff,
\[r_\al(T,x)=\lb_F(\alpha)+\frac{1}{T}\Big(  \log\langle \phi_\alpha(0), x \rangle+\log\langle m,e_\alpha(T)\rangle\Big)+o\left(\frac1T\right)\,. \]

A natural problem is to find periodic controls $\alpha(t)$ such that the Floquet eigenvalue is better than the optimal Perron eigenvalue $\lambda_P(\alpha^*)$. The following Proposition gives a partial answer to this question.
We consider small periodic perturbations of the best constant control: $\alpha(t) = \al^* + \epsilon \gamma(t)$, where $\gamma$ is a given $\theta$-periodic function.
%
For the sake of clarity we introduce the following notation for the time average over a period,
\[\langle f\rangle_\theta=\f1\theta\int_0^\theta f(t)\,dt\, .\]
We assume that the matrix $G + \alpha^* F$ is diagonalizable (in $\mathbb{R}$). This is the case for the running example (see Appendix~\ref{ap:diag}). We  denote by $(e_1^*,e_2^*,e_3^*)$ and $(\phi_1^*,\phi_2^*,\phi_3^*)$ the bases of right- and left- eigenvectors associated to the eigenvalues $\lambda_1^*,\lb_2^*,\lb_3^*$ for the the best constant control $\al^*$.
Notice that according to previous notations we have $e_1^*=e_{\alpha^*}$ and $\lambda_1^*=\lambda_P(\al^*)$.

\begin{proposition}\label{prop:per:firstderiv}
The directional derivative of the dominant eigenvalue vanishes at $\epsilon = 0$:
\begin{equation} \left.\dfrac{d \lb_F(\alpha^* + \epsilon \gamma)}{d\epsilon}\right|_{\epsilon = 0}  = 0\, . \label{eq:1st order}\end{equation}
Hence, $\al^*$ is also a critical point in the class of periodic controls. The second directional derivative of the dominant eigenvalue writes at $\epsilon = 0$: 
\beq\label{eq:secondFloquet}
\left.\dfrac{d^2 \lb_F(\alpha^* + \epsilon \gamma)}{d\epsilon^2}\right|_{\epsilon = 0} = 2\sum_{i=2}^{3}  \langle\gamma_i^2\rangle_\theta\frac{(\phi_1^* F e_i^*)(\phi_i^* F e_1^*)}{ \lambda_1^*-\lambda_i^* }\, ,\eeq
where $\gamma_i(t)$ is the unique $\theta$-periodic function which is solution to the relaxation ODE
\[
\dfrac{\dot{\gamma_i}(t)}{ \lambda_1^* - \lambda_i^* }  +  \gamma_i(t) =  \gamma(t) \, .
\]
\end{proposition}

The idea of computing directional derivatives has been used in a similar context in~\cite{M2} for optimizing the Perron eigenvalue in a continuous model for cell division. 


Taking $\gamma\equiv1$ in Equation~\eqref{eq:secondFloquet}, we get the second derivative of the Perron eigenvalue at $\alpha^*$,
\beq\label{eq:secondPerron}\dfrac{d^2\lambda_P}{d\al^2}(\al^*)=2\sum_{i=2}^{3}\dfrac{(\phi_1^* F e_i^*)(\phi_i^* F e_1^*)}{\lambda_1^*-\lambda_i^*} \, ,\eeq
which is nonpositive since $\al^*$ is a maximum point. However we cannot conclude directly in the general case  that the quantity \eqref{eq:secondFloquet} is nonpositive 
since we do not know the relative signs of the coefficients $(\phi_1^* F e_i^*)(\phi_i^* F e_1^*)$ in \eqref{eq:secondPerron}.



We study in~\cite{CCGS} a variant of \eqref{eq:dynsyst}, where the control is not affine. We prove that there exist directions $\gamma$ for which $\alpha^*$ is a minimum point. Alternatively speaking, periodic controls can beat the best constant control.
For this, we perturb the constant control with high-frequency modes, and we compute the limit of $\frac{d^2\lb_F}{d\eps^2}$ when the frequency tends to infinity.
Unfortunately, this procedure gives no additional information in the case of an affine control.

\begin{proof}[Proof of Proposition~\ref{prop:per:firstderiv}]
First we derive a formula for the first derivative of the Perron eigenvalue:
By definition we have
\[(G+\al F)e_\alpha=\lambda_P(\alpha)e_\alpha\,.\]
Deriving with respect to $\al$ we get
\[
\frac {d\lambda_P}{d\alpha}(\alpha) e_\alpha+\lambda_P(\alpha)\frac {d e_\alpha}{d\alpha} =Fe_\alpha+(G+\al F)\frac {d e_\alpha}{d\alpha}\, .
\]
Testing against the left- eigenvector $\phi_\alpha$ we obtain
\[\frac {d\lambda_P}{d\alpha}(\alpha)    =  \phi_\alpha Fe_\alpha \, .\]
Second, we write the Floquet eigenvalue problem corresponding to the periodic control $\alpha = \alpha^* + \epsilon\gamma$:
\[\frac{\partial}{\partial t} e_\alpha(t) + \lb_F(\alpha^* + \epsilon\gamma ) e_\alpha(t) = (G+(\al^*+ \epsilon \gamma(t))F) e_\alpha(t)\, .\]
Deriving this ODE with respect to $\epsilon$, we get
\beq\label{eq:firstderiv_periodic} 
\frac{\partial}{\partial t} \dfrac{\partial e_\alpha}{\partial\epsilon}(t) + \dfrac{d \lb_F(\alpha^* + \epsilon\gamma )}{d\epsilon} e_\alpha(t) + \lb_F(\alpha^* + \epsilon\gamma ) \dfrac{\partial e_\alpha}{\partial\epsilon}(t)  =  \gamma(t)F e_\alpha(t) + (G+(\al^*+ \epsilon \gamma(t))F)\dfrac{\partial e_\alpha}{\partial\epsilon}(t)  \, .
\eeq
Testing this equation against $\phi_1^*$ and evaluating at $\epsilon = 0$, we obtain
\[ \frac{\partial}{\partial t} \left( \phi_1^* \left.\dfrac{\partial e_\alpha}{\partial\epsilon}\right|_{\epsilon = 0}(t) \right) + \left.\dfrac{d \lb_F(\alpha^* + \epsilon \gamma)}{d\epsilon}\right|_{\epsilon = 0} = \gamma(t) \phi_1^*F e_1^*\, .\]
After integration over one period, we get
\[ \left.\dfrac{d \lb_F(\alpha^* + \epsilon \gamma)}{d\epsilon}\right|_{\epsilon = 0} = \left(\frac1\theta\int_0^\theta \gamma(t)\,dt\right)\phi_1^*F e_1^*= \la \gamma \ra_\theta \dfrac{d\lambda_P}{d\alpha} (\al^*) = 0\, ,\]
which is the first order condition \eqref{eq:1st order}.

Next, we test~\eqref{eq:firstderiv_periodic} against another left- eigenvector $\phi_i^*$ and we evaluate at $\epsilon = 0$.
We obtain the following equation satisfied by $\gamma_i(t) = (\lambda_1^* - \lambda_i^*)\phi_i^* \dfrac{\partial e_\alpha}{\partial\epsilon}(t)(\phi_i^* Fe_1^*)^{-1}$:
\begin{equation} 
\frac1{\lambda_1^* - \lambda_i ^*} \frac \partial {\partial t}{\gamma_i}(t)  +  \gamma_i(t) =  \gamma(t)  \, .
\label{eq:gamma_i}
\end{equation}
We differentiate~\eqref{eq:firstderiv_periodic} with respect to $\epsilon$. This yields
\begin{multline*}
\frac{\partial}{\partial t} \dfrac{\partial^2 e_\alpha}{\partial\epsilon^2}(t) + \dfrac{d^2 \lb_F(\alpha^* + \epsilon\gamma )}{d\epsilon^2} e_\alpha(t) + 2 \dfrac{d \lb_F(\alpha^* + \epsilon\gamma )}{d\epsilon} \dfrac{\partial  e_\alpha}{\partial\epsilon }(t)   +  \lb_F(\alpha^* + \epsilon\gamma ) \dfrac{\partial^2 e_\alpha}{\partial\epsilon^2}(t) 
\\ = 2 \gamma(t)F \dfrac{\partial  e_\alpha}{\partial\epsilon }(t) + (G+(\al^*+ \epsilon \gamma(t))F)\dfrac{\partial^2 e_\alpha}{\partial\epsilon^2}(t)  \, .
\end{multline*}
Testing this equation against $\phi_1^*$ and evaluating at $\epsilon = 0$, we find
\begin{equation}
\frac{\partial}{\partial t} \left( \phi_1^* \left.\dfrac{\partial^2 e_\alpha}{\partial\epsilon^2}\right|_{\epsilon = 0}(t) \right) +  \left.\dfrac{d^2 \lb_F(\alpha^* + \epsilon \gamma)}{d\epsilon^2}\right|_{\epsilon = 0}    = 2\gamma(t) \phi_1^* F   \left.\dfrac{\partial e_\alpha}{\partial\epsilon}\right|_{\epsilon = 0}(t)  \, . 
\label{eq:extremal Floquet}
\end{equation}
We decompose the unknown $\frac{\partial e_\alpha}{\partial\epsilon}(t)$ along the basis $(e_1^*,e_2^*,e_3^*)$
\[  \frac{\partial e_\alpha}{\partial\epsilon} (t) = \sum_{i = 1}^3 \gamma_i(t)\frac{(\phi_i^* Fe_1^*)}{\lambda_1^* - \lambda_i^*}  e_i^*  \, .\]
In particular, we have
\[ \phi_1^* F    \dfrac{\partial e_\alpha}{\partial\epsilon} (t) = \sum_{i=2}^{3}\gamma_i(t)\frac{(\phi_i^* F e_1^*)}{\lambda_1^* - \lambda_i^*}(\phi_1^* F e_i^*)\, ,   \]
since $\phi_1^* Fe_1^* = 0$ by optimality.
To conclude, we integrate \eqref{eq:extremal Floquet} over one period,
\begin{equation*}
\left.\dfrac{d^2 \lb_F(\alpha^* + \epsilon \gamma)}{d\epsilon^2}\right|_{\epsilon = 0}  =
 2 \sum_{i=2}^{3}\langle\gamma\gamma_i\rangle_\theta \dfrac{(\phi_i^* F e_1^*)( \phi_1^* F e_i^*)}{\lambda_1^* - \lambda_i^*} \,.
\end{equation*}
We conclude thanks to the following identity derived from \eqref{eq:gamma_i}: $\la \gamma_i ^2\ra_\theta = \langle\gamma\gamma_i\rangle_\theta$.
\end{proof}

\section{Optimal control and the ergodic problem}\label{sec:HJB}

Let $a,A$ be some given bounds on the control parameter $\alpha$, with $a<\alpha^*< A$. 
Let $T>0$ be a (large) time. The optimal control problem associated to \eqref{eq:dynsyst} reads as follows
\begin{equation*}
V(T,x) = \sup_\alpha \left\{ \la m, x_\alpha(T)\ra \;:\;\dot x_\alpha(t) = (G + \alpha(t) F) x_\alpha(t)\;,\; x_\alpha(0) = x  \right\}\, ,
\end{equation*}
where the supremum is taken over all measurable control functions $\alpha: (0,T)\to [a,A]$. The principle of dynamic programming enables to solve this optimal control problem by introducing the value function $v(s,x)=\sup \left\{ \la m, x_\alpha(T)\ra \;:\;\dot x_\alpha(t) = (G + \alpha(t) F) x_\alpha(t)\;,\; x_\alpha(s) = x  \right\}$. We have $v(T,x) = \la m, x \ra$, $v(0,x) = V(T,x)$, and $v$ satisfies the following Hamilton-Jacobi-Bellman equation in the viscosity sense,
\[ \frac{\partial}{\partial s} v (s,x) + \widetilde H(x, D_x v(s,x)) = 0 \, ,
 \]
where the Hamiltonian is given by $ \widetilde H(x,p) = \max_{\alpha\in [a,A]} \left\{ \la(G + \alpha F)x , p \ra  \right\}$.
This equation is backward in time: the initial data is prescribed at the final time $s=T$. 

Intuitively we expect an exponential growth of the reward $\la m, x_\alpha(T)\ra$ as $T\to +\infty$. This motivates the following reduction to a compact space/linear growth.

\subsection{Reduction to a compact space/linear growth}


We perform a logarithmic change of variable: $w(s,x) = \log v(s,x) - \log \la  \m , x\ra$. The reward function satisfies:
\[ \dfrac d{dt} \log \la  \m , x_\alpha(t)\ra = \dfrac{ \la  \m , \dot x_\alpha(t)\ra}{  \la  \m , x_\alpha(t)\ra} = \la  \m  , (G + \alpha(t) F) y_\alpha(t)\ra = \la  \m  , G y_\alpha(t)\ra\,,  \]
where $y = \frac{x}{\la  \m , x\ra}$ denotes the projection on the simplex $\mathcal S = \{ y\geq 0\;:\; \la m,y \ra = 1 \}$. We can write a close equation for the projected trajectory $y_\alpha(t)$ due to the linearity of the system:
\begin{align*}
\dot y_\alpha(t) & = \frac{\dot x_\alpha(t)}{\la  \m , x_\alpha(t)\ra} - \frac{x_\alpha(t)}{\la  \m , x_\alpha(t)\ra}\frac{\la  \m , \dot x_\alpha(t)\ra}{\la  \m , x_\alpha(t)\ra} \nonumber\\
&= (G + \alpha(t) F) y_\alpha(t) - \la  \m , (G + \alpha(t) F) y_\alpha(t)\ra y_\alpha(t) \nonumber\\
&= (G + \alpha(t) F) y_\alpha(t) - \la  \m ,  G y_\alpha(t)\ra y_\alpha(t)\, ,
\end{align*}
(recall $m^T F = 0$). We introduce the following notation for the vector field on the simplex:
\[ b(y,\alpha) = (G + \alpha F) y - \la  \m ,  G y\ra y\, . \]
It is worth mentioning that the only stationary point of the vector field $b(y,\alpha)$ on the simplex $\mS$ is the dominant eigenvector $e_\alpha$: $b(e_\alpha,\alpha) = 0$. Furthermore, for any $y \in \mS$ the trajectory starting from $y$ with constant control $\alpha(t)\equiv \alpha$ converges to $e_\alpha$ as $t \to +\infty$, and leaves the simplex $\mS$ as $t \to -\infty$.
Moreover the vector field $b(y,\alpha)$ can be easily computed on each eigenvector $y = e_\beta$ due to the affine structure of the problem: 
\begin{align*}
b(e_\beta,\alpha) & = (G + \alpha F) e_\beta - \la  \m ,  G e_\beta\ra e_\beta \\
& = ( \alpha - \beta ) Fe_\beta + b(e_\beta, \beta) \\
& = ( \alpha - \beta ) Fe_\beta\, .
\end{align*}

The logarithmic value function $w$ satisfies the following optimization problem for $x\in \mS$:
\[ w(s,x) = \sup_\alpha\left\{ \log \la  \m , x_\alpha(T)\ra  - \log \la  \m , x\ra \right\} = \sup_\alpha \left\{ \int_{s}^T \la  \m  , G y_\alpha(t)\ra\, dt\;:\; \dot y_\alpha(t) = b(y_\alpha(t),\alpha(t))\; , \; y_\alpha(s) = x \right\}  \]
Finally for notational convenience, we perform the time reversal $t = T-s$, and introduce $u(t,y) = w(T-t,y)$ for $y\in \mS$. The optimal problem is now reduced to a compact space (the simplex $\mathcal S$), with a running reward: 
\[ u(t,y) = \sup_\alpha \left\{ \int_{0}^t L(y_\alpha(s))\, ds  \;:\; \dot y_\alpha(s) = b(y_\alpha(s),\alpha(s))\; , \; y_\alpha(0) = y\right\}\, . \]
where the reward function is linear: $L(y) = \la  \m  , G y\ra$.
The function $-u$ is a viscosity solution of the following Hamilton-Jacobi equation:
\begin{equation} \label{eq:HJB} 
\left\{\begin{array}{l}
\dfrac{\partial}{\partial t} (-  u(t,y)) + H(y,D_y u(t,y)) = 0\, , \medskip \\
u(0,y) = 0\, ,
\end{array}\right.
\end{equation}
where the Hamiltonian is given by 
\begin{equation} \label{eq:Hamiltonian} H(y,p) =  \max_{\alpha\in [a,A]} \left\{ \la b(y,\alpha), p \ra + \la  \m , G y\ra \right\}\, .\end{equation}

\begin{remark}\label{rem:not coercive}
An important observation is that this hamiltonian does not satisfy the classical coercivity assumption $H(\cdot,p) \to +\infty$ as $|p|\to +\infty$. Indeed for all $y\in \mS$ there exists a cone $C(y)$ such that $H(y,p)\to -\infty$ if $p\in C(y)$ and $|p|\to +\infty$. In addition we have $C(y) = \emptyset$ if and only if $y = e_\beta$ for some $\beta\geq 0$. In the latter case we have 
\[ H(e_\beta,p) = \max_{\alpha\in [a,A]}\{ \la(\alpha - \beta) Fe_\beta , p \ra \} + L(e_\beta) = (A - \beta) \la  Fe_\beta , p \ra_+ + (\beta- a) \la  Fe_\beta , p \ra_- + \lambda_P(\beta)\, . \]
It is not coercive either.
\end{remark}

Due to the lack of coercivity the ergodic problem is difficult to handle with. We follow the procedure described in \cite{Arisawa1,Arisawa2} to exhibit an ergodic set in the simplex. This set has to fulfill two important features: controllability, and attractivity. We construct below such an ergodic set, and we prove the two required properties.

\subsection{Notations and statement of the result}

\label{ssec:result}

Following \cite{Arisawa2} we introduce the auxiliary problem with infinite horizon:
\begin{equation} \label{eq:infinite horizon}
u_\eps(y) = \sup_\alpha \left\{ \int_{0}^{+\infty} e^{-\eps t} L(y_\alpha(t))\, dt  \;:\; \dot y_\alpha(t) = b(y_\alpha(t),\alpha(t))\; , \; y_\alpha(0) = y\right\}\, .
\end{equation}
The limit of the quantity $\eps u_\eps(y)$ as $\eps\to 0$ measures the convergence of the reward $L(y(t))$ as $t\to +\infty$ in average (convergence {\em \`a la C\'esaro}). The function $-u$ is a viscosity solution of the following stationary Hamilton-Jacobi equation: 
\begin{equation} \label{eq:HJ stat}
- \eps u_\eps(y) + H(y,D_y u_\eps(y)) = 0\, .
\end{equation}
The next Theorem claims that the function $\eps u_\eps(y)$ converges to a constant.
We interpret this constant as the (generalized) eigenvalue associated to the optimal control problem.

We introduce some notations necessary for the statement of the two technical hypotheses {\bf (H4-5)}:
\begin{itemize}
\item $\mS$: the simplex $\{ y\geq 0\;:\; \la m,y \ra = 1 \}$,
\item $\Theta$: the (direct) orthogonal rotation with angle $\pi/2$ on the tangent space $T\mS$,
\item $\Phi_0 = \{ e_\alpha\;|\; \alpha\in \RR_+ \}\subset\mS$: the set of Perron eigenvectors,
\end{itemize}
\begin{description}
\item[(H4)] We assume a first technical condition:
\begin{equation}\label{as:tech}
\la \frac{d e_\al}{d\al}  ,  \Theta F e_\alpha  \ra\quad\text{has a constant sign for $\alpha \geq 0$.}
\end{equation}
\end{description}
This hypothesis enables to determine the direction of the vector field $b(y,\alpha)$ accross $\Phi_0$ at $y = e_\beta$: 
\begin{equation} 
\la \Theta \left.\frac{d e_\al}{d\al}\right|_{\alpha = \beta}  ,  b(e_\beta,\alpha) \ra = -( \alpha - \beta ) \la \left.\frac{d e_\al}{d\al}\right|_{\alpha = \beta}  ,  \Theta F e_\beta \ra \, . 
\label{eq:sign accross}
\end{equation}
Due to Hypothesis {\bf (H4)}, only the sign of $\alpha - \beta$ determines the sign of the above quantity. 

We assume without loss of generality that the quantity \eqref{as:tech} is negative for all $\alpha\geq 0$. This has been checked numerically in the case of the running example \eqref{eq:example} (see Appendix~\ref{ap:tech}). 
\begin{description}
\item[(H5)] We assume a second technical condition: there exists $\delta_0>0$ such that for any constant control $\beta\in (A-\delta_0, A + \delta_0)$, the trajectory starting from $e_0$ with constant control $\beta$, which connects $e_0$ and $e_\beta$, does not cross the curve $\Phi_0$ for $t>0$. Similarly, for any constant control $\beta\in (a-\delta_0, a + \delta_0)$, the trajectory starting from $e_\infty$ with constant control $\beta$, which connects $e_\infty$ and $e_\beta$, does not cross the curve $\Phi_0$.
\end{description}
A proper construction of some remarkable sets in the proof of the following Theorem relies on these two technical assumptions, {\em e.g.} $\mZ_0,\mT_\pm$ (Figures \ref{fig:eigenvector} and \ref{fig:tunneling}).

\begin{theorem}[Ergodicity] \label{th:eigHJ}
Assume that hypotheses {\bf (H1-2-3-4-5)} are satisfied. 
There exists a constant $\lambda_{HJ}$ such that the following uniform convergence holds true:
\[ \lim_{\eps\to 0}\eps u_\eps(y) = \lambda_{HJ}\, , \quad \text{uniformly for $y\in \mS$}\, . \] 
\end{theorem}

The constant $\lambda_{HJ}$ is interpreted as the eigenvalue associated to the Hamiltonian $H(y,p)$. We are not able to prove the existence of an eigenvector here. Indeed we lack an equicontinuity estimate on  the family $(u_\eps)_\eps$. We only obtain equicontinuity of the family $(\eps u_\eps)_\eps$.
However we postulate such an eigenvector does exist, based on numerical evidence (see Section \ref{sec:num}). 

\begin{corollary}\label{cor:ergodic}
Assume that hypotheses {\bf (H1-2-3-4-5)} are satisfied. Then the solution of the Hamilton-Jacobi-Bellman equation satisfies the following ergodic property:
\begin{equation*}
\lim_{T\to +\infty} \dfrac{u(T,y)}{T} = \lambda_{HJ}\, , \quad \text{uniformly for $y\in \mS$}\,.
\end{equation*}
\end{corollary}

Back to the original problem, Corollary \ref{cor:ergodic} translates into \eqref{eq:ergodic}, where convergence is uniform on compact subsets of $\left(\RR_+\right)^3\setminus\{0\}$. 


\subsection{Proof of Theorem \ref{th:eigHJ}}

We mainly follow~\cite{Arisawa2} with some refinements specific to our context. We proceed in several steps, with illustrative figures to facilitate the argumentation.

The program is the following: we identify the so-called ergodic set $\mZ_0\subset\mS$~\cite{Arisawa1,Arisawa2}, and we show some of its interesting properties. We begin with the controllability problem. Controllability enables to prove convergence to the constant $\lambda_{HJ}$ when $y$ belongs to the ergodic set $\mZ_0$. Next we demonstrate the attractiveness of the ergodic set $\mZ_0$. This enables to extend the convergence to every $y\in \mathcal S$. 


There is a technical subtlety concerning the time required for controllability/attractiveness. This time could degenerate as we get close to the boundary of the ergodic set. We circumvent this issue by proving that close-to-optimal trajectories do not stay in the vicinity of the boundary for long time.

\paragraph{Step 1- Identification of some remarkable sets: the set of eigenvectors, and the ergodic set.} 
The first remarkable subset of the simplex $\mathcal S$ is the set of eigenvectors $\Phi_0 = \{ e_\alpha\;|\; \alpha\in \RR_+ \}$. It is alternatively defined as the zero level set of the cubic function 
\[\varphi(y) = \la b(y,\alpha), \Theta Fy\ra = \la Gy - \la  \m , Gy\ra y , \Theta Fy  \ra\, .\] 
Hence we have,
\[ \Phi_0 = \{ e_\alpha\;|\; \alpha\in \RR_+ \} = \{  y\in \mS \;:\;  \varphi(y) = 0\}\, . \]
Indeed, $y\in \mS$ is an eigenvector if and only if there exists $\alpha$ such that $(G + \alpha F) y - \la  \m ,  G y\ra y = 0$, {\em i.e.} $  Gy - \la  \m ,  G y\ra y \in \spa(F y)$, or equivalently $  Gy - \la  \m ,  G y\ra y \perp \Theta F y $.
From Hypothesis {\bf (H3)} we deduce that $\Phi_0$ is a curve connecting two boundary points of the simplex $\mS$. In Figure \ref{fig:eigenvector} we have plotted the set of eigenvectors in the case of the running example.

The set of eigenvector splits the simplex $\mS$ into two subsets, denoted by $\Phi_+$ and $\Phi_-$, respectively:
\begin{equation*}
\Phi_+ = \{y\in \mS \;:\; \varphi(y) \geq 0 \}\quad, \quad \Phi_- = \{ y\in \mS \;:\; \varphi(y) \leq 0 \}\, .
\end{equation*}
Notice that the sign of $\varphi$ depends on the orientation of the rotation $\Theta$. Obviously the following discussion does not depend on this convention.

\begin{figure}
\begin{center}
\includegraphics[width = 0.46\linewidth]{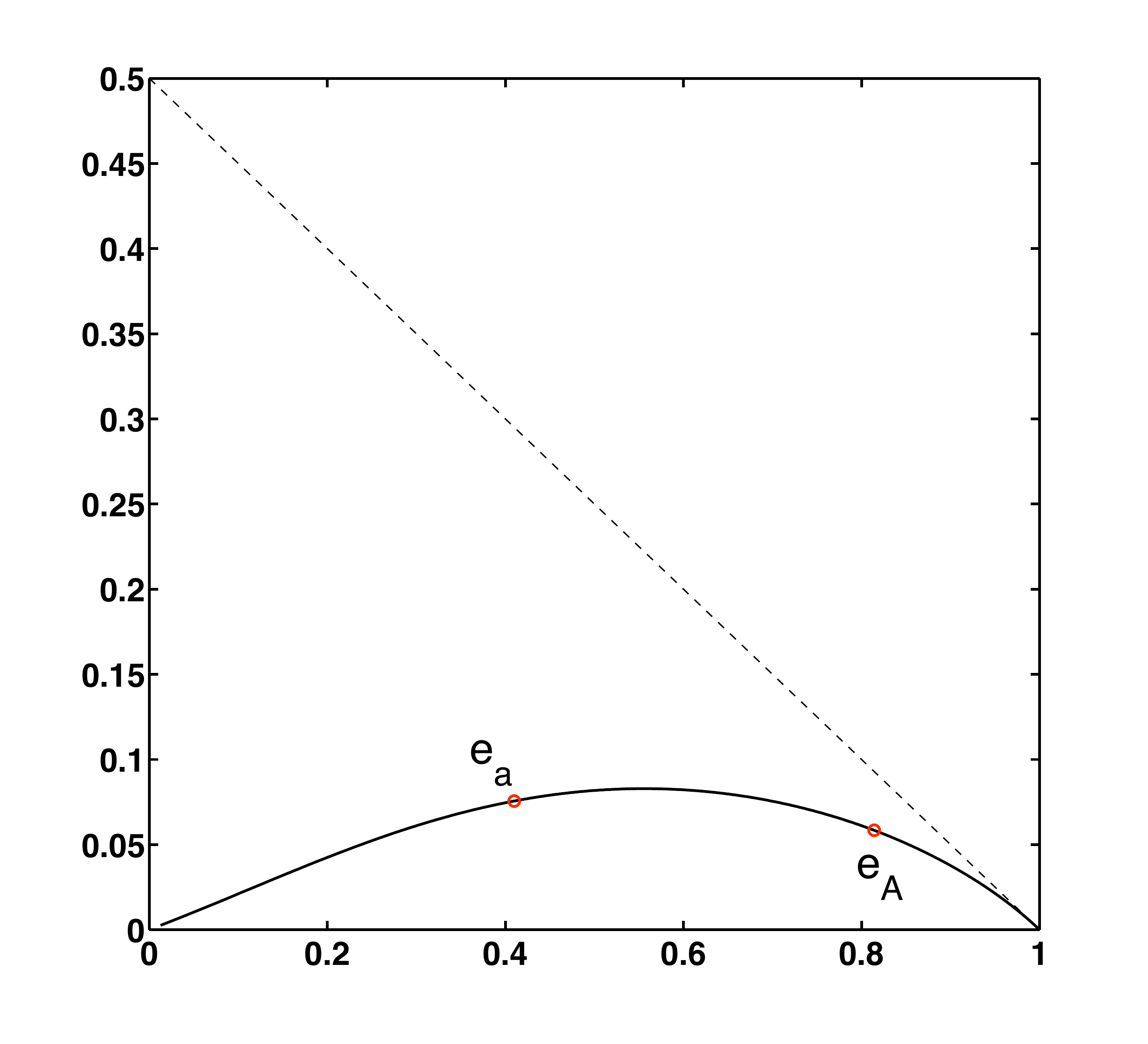}(a)\,\includegraphics[width = 0.46\linewidth]{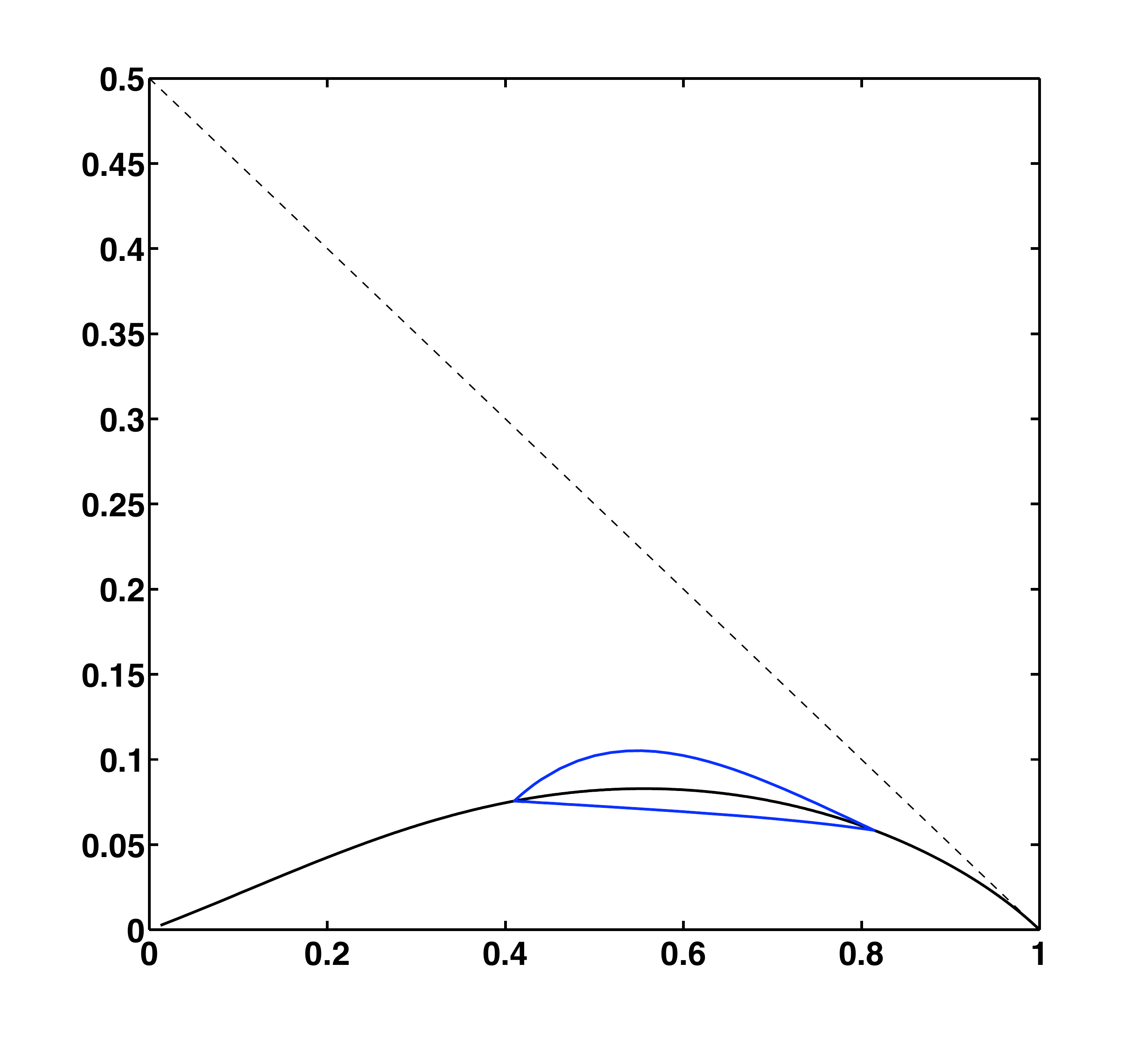}(b)
\caption{The set of eigenvectors $\Phi_0$ (a) and the boundary of the  ergodic set $\mZ_0$ (b). The dashed line represent the boundary of the simplex.}\label{fig:eigenvector}
\end{center}
\end{figure}

We seek the ergodic set which is stable, attractant and controllable. The natural candidate is defined through its boundary as follows. The boundary consists of two curves which are joining on the set $\Phi_0$.
\begin{definition}
The ergodic set $\mZ_0$ is the compact set enclosed by the two following curves:
\begin{equation} \label{eq:def Z0}
\gamma_a^A \begin{cases}
& \dot \gamma (s) = b(\gamma (s),a) \\
& \gamma (0) = e_A
\end{cases}\, ,  \quad
\gamma_A^a\begin{cases}
& \dot \gamma (s) = b(\gamma (s),A) \\
& \gamma (0) = e_a
\end{cases}   \, .
\end{equation}
The set $\mZ_0$ is well defined. In fact we have $\lim_{s\to+\infty} \gamma_a^A(s) = e_a$ and $\lim_{s\to+\infty} \gamma_A^a(s) = e_A$.
\end{definition}
In Figure \ref{fig:eigenvector} we have plotted an example of the ergodic set $\mZ_0$ for the running example.
We list in the following Proposition some useful properties of the set $\mZ_0$ which are derived from the very definition.
\begin{lemma} \label{lem:stability} 
\begin{enumerate}[(i)]
\item The curves $\gamma^A_a,\gamma^a_A$ lie on the opposite sides of $\Phi_0$. We assume without loss of generality that the curve $\gamma^A_a$ belongs to the subset $\Phi_+$, whereas the curve $\gamma^a_A$ belongs to the subset $\Phi_-$. 
\item The set $\mZ_0$ is stable: the vector fields $\{b(y,\alpha)\;,\; \alpha\in [a,A]\}$ are all pointing inwards on the boundary of $\mZ_0$.
\end{enumerate}
\end{lemma}

\begin{proof}
\noindent\emph{(i)} We denote by $\gamma_A^0$ the trajectory connecting $e_0$ to $e_A$ with constant control $A$. Similarly we denote by $\gamma_a^\infty$ the trajectory connecting $e_\infty$ to $e_a$ with constant control $a$. From Hypothesis {\bf (H4)} these two trajectories initially start on the two opposite sides of the line of eigenvectors $\Phi_0$. Notice that $\gamma_a^\infty$ may start tangentially to $\Phi_0$ (it is actually the case for the running example). However a continuity argument for $\gamma_a^\beta$ with $\beta\to +\infty$ yields the statement. Hypothesis  {\bf (H5)} guarantees that they do not cross $\Phi_0$, so that they stay on opposite sides of $\Phi_0$ forever. We assume without loss of generality that $\gamma_A^0$ lies in $\Phi_-$ and $\gamma_a^\infty$ lies in $\Phi_+$. 



Second from the bounds on the control $a\leq \alpha(t)\leq A$, together with Hypothesis {\bf (H4)} \eqref{eq:sign accross}, the trajectory $\gamma_A^a$ starts on the same side as $\gamma_A^0$ ({\em i.e.} $\Phi_-$ with the above convention). Moreover it cannot cross the  line of eigenvectors $\Phi_0$ at $e_\beta$ for $\beta\in (0,A)$. From the uniqueness theorem for ODE, it cannot cross the trajectory $\gamma_A^0$ either. As a conclusion, the trajectory $\gamma_A^a$ is sandwiched between the portion of $\Phi_0$ between $e_0$ and $e_A$, and $\gamma_A^0$. Therefore it belongs to $\Phi_-$. 

The same holds true for $\gamma_a^A\subset \Phi_+$.


\noindent\emph{(ii)} We denote by $n(z)$ the unit vector normal and exterior to the boundary of $\mZ_0$. We have more precisely:
\begin{equation*} n_a(z) = \dfrac{\Theta b(z,a)}{|\Theta b(z,a)|}  \quad \mbox{on $\gamma^A_a$}\quad, \quad 
n_A(z) = \dfrac{\Theta b(z,A)}{|\Theta b(z,A)|}  \quad \mbox{on $\gamma^a_A$}\,.
\end{equation*} 
We have on the one side $\gamma^A_a$: $\forall\alpha\in [a,A]\; \la b(z,\alpha) , n_a(z) \ra \leq \la b(z,a) , n_a(z) \ra = 0$. Indeed we have
\begin{align*}
\la b(z,\alpha)- b(z,a) , n_a(z) \ra & = (\alpha - a) \la   Fz , n_a(z)\ra  \\
&  = (\alpha - a) \la \Theta  Fz  , \Theta n_a(z) \ra \\
& = (\alpha - a) \la  \Theta Fz  , - \dfrac{b(z,a)}{|b(z,a)|}\ra \\
& = \dfrac{(a - \alpha )}{|b(z,a)|} \la \Theta Fz  , Gz - \la  \m ,  G z\ra z \ra \leq 0\,. 
\end{align*}
We have on the other side  $\gamma^a_A$:
\begin{equation*}
\la b(z,\alpha)- b(z,A) , n_A(z) \ra = (\alpha - A) \la   Fz , n_A(z)\ra = \dfrac{(A - \alpha )}{|b(z,A)|} \varphi(z) \leq 0\,. 
\end{equation*}
\end{proof}

\paragraph{Step 2- Exact controllability inside the ergodic set.}
The purpose of this step is to prove the following Lemma \ref{lem:controllability}, which asserts exact controllability in the ergodic set, except  a narrow band.
First we introduce some notations:
\begin{itemize}
\item $\mZ_{-\delta}$ denotes the ergodic set, from which we have substracted a narrow band close to the boundary $\gamma^A_a\cup\gamma^a_A$ (Figure \ref{fig:controllability}). It is defined similarly as $\mZ_0$ by the delimitation of the two following curves,
\begin{equation*}
\gamma^{A-\delta}_{a+\delta} \begin{cases}
& \dot \gamma (s) = b (\gamma (s),a+\delta ) \\
& \gamma (0) = e_{A-\delta}
\end{cases}\, ,  \quad
\gamma^{a+\delta}_{A-\delta}\begin{cases}
& \dot \gamma (s) = b(\gamma (s),A-\delta) \\
& \gamma (0) = e_{a+\delta}
\end{cases}  \, .
\end{equation*}
\item $\mZ_{+2\delta}$ denotes the ergodic set, to which we have added a narrow band close to the boundary $\gamma^A_a\cup\gamma^a_A$ (Figure \ref{fig:tunneling}). It is defined similarly as $\mZ_0$ by the delimitation of the two following curves,
\begin{equation*}
\gamma^{A+2\delta}_{a-2\delta} \begin{cases}
& \dot \gamma (s) = b(\gamma (s),a
-2\delta) \\
& \gamma (0) = e_{A+2\delta}
\end{cases}\, ,  \quad
\gamma^{a-2\delta}_{A+2\delta}\begin{cases}
& \dot \gamma(s) = b(\gamma (s), A+2\delta) \\
& \gamma (0) = e_{a-2\delta}
\end{cases}   \, . 
\end{equation*}
\end{itemize}

\begin{lemma}\label{lem:controllability}
The system $\dot y(t) = b(y(t),\alpha)$ is exactly controllable in the subset $\mZ_{-\delta}$. For any $z$, $z'$ in $\mZ_{-\delta}$ there exists a time $T = T(z,z')$ and  $\alpha(t):[0,T]\to [a,A]$ such that $y(0) = z$ and $y(T) = z'$. Moreover it is possible to construct the control $\alpha(t)$ following a bang-bang procedure: $\alpha(t):[0,T]\to \{a,A\}$. Last, the minimal time $T$ needed to connect $z$ and $z'$ is uniformly bounded for $z,z'\in \mZ_{-\delta}$.
\end{lemma}

\begin{figure}
\begin{center}
\includegraphics[width = 0.46\linewidth]{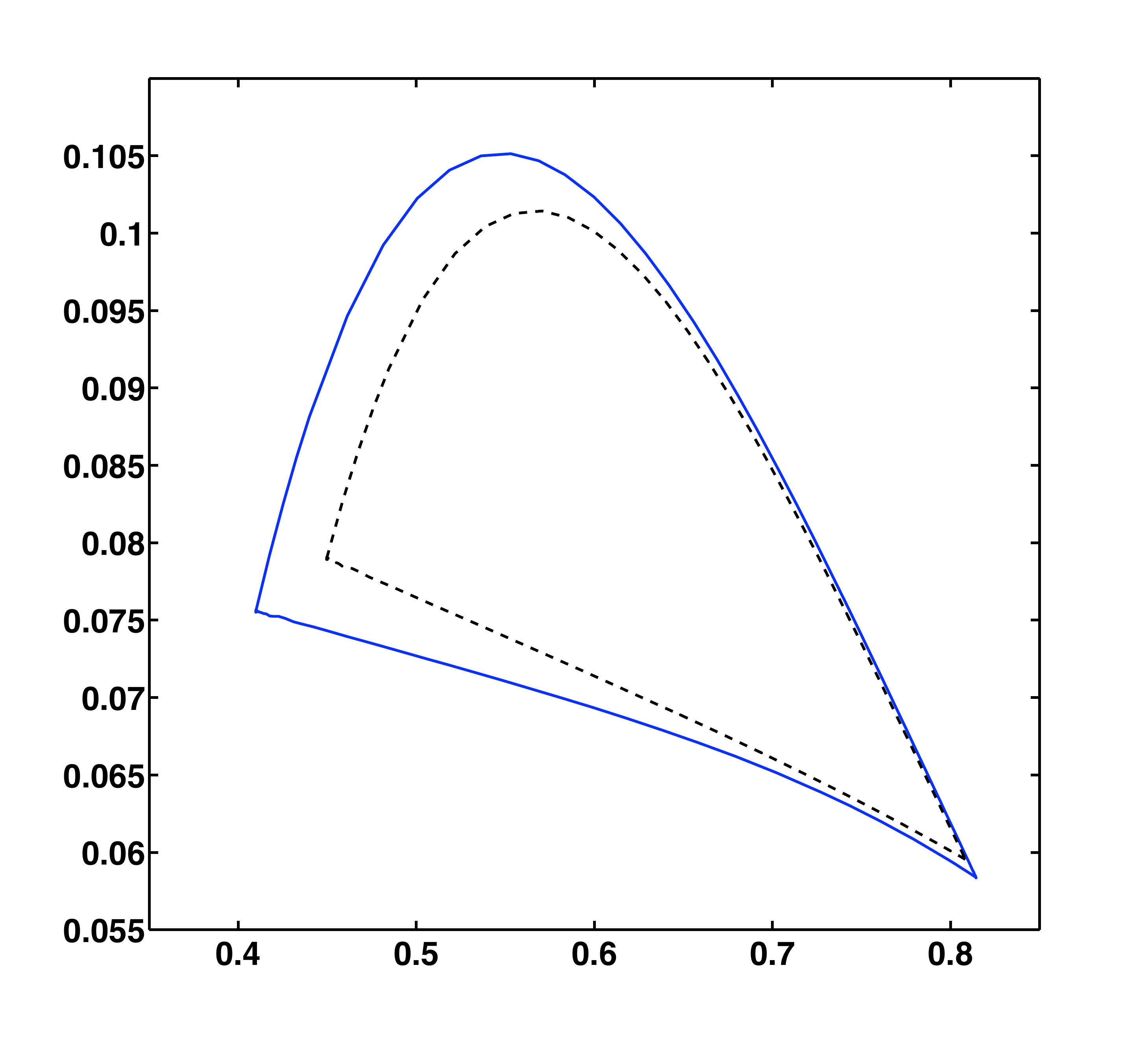}(a)\,\includegraphics[width = 0.46\linewidth]{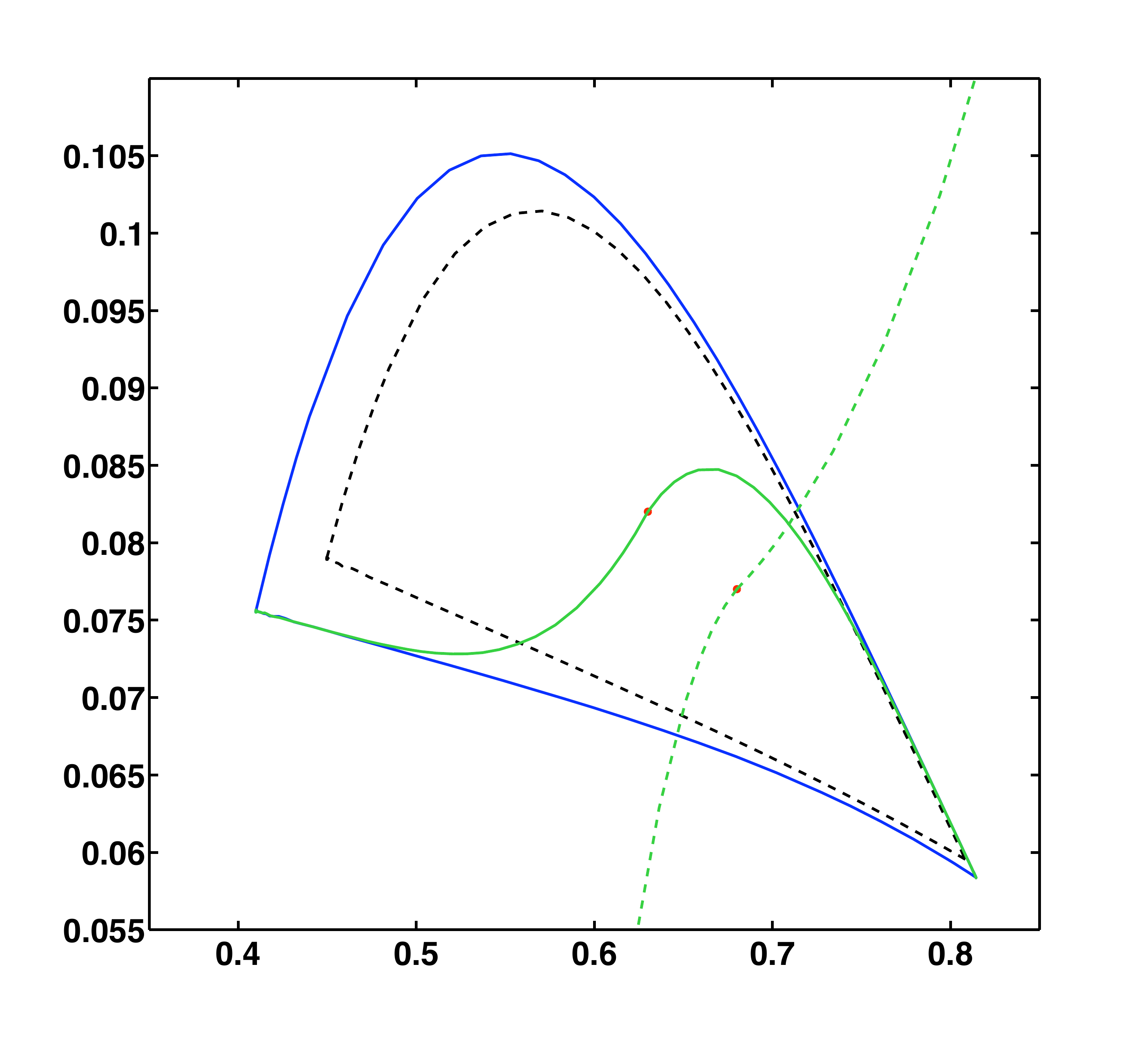}(b)
\caption{(a) The ergodic set $\mZ_0$ and the set $\mZ_{-\delta}$ (dashed line). (b) Illustration for the proof of the exact controllability in the set $\mZ_{-\delta}$: the two curves $Y$ (plain line) and $Y'$ (dashed line) necessarily intersect in $\mZ_0$.}
\label{fig:controllability}
\end{center}
\end{figure}

\begin{proof}
We aim at sending $z$ onto $z'$. We define the trajectories $y_{a}$ and $y_A$ starting from $y$ as follows,
\begin{equation*}
y_{a} \begin{cases}
& \dot y_{a}(t) = b(y_a(t),a) \\
& y_a(0) = z
\end{cases}\, ,  \quad
y_{A} \begin{cases}
& \dot y_{A}(t) = b(y_A(t),A) \\
& y_A(0) = z
\end{cases}\, .\end{equation*}
When we concatenate the two trajectories $Y = y_a\cup y_A$, we obtain a curve which connects $a$ and $A$.
Reversing time, we define two other trajectories starting from $z'$,
\begin{equation*}
y'_{a} \begin{cases}
& \dot y'_{a}(t) = - b(y'_a(t),a) \\
& y'_a(0) = z'
\end{cases}\, ,  \quad
y'_{A} \begin{cases}
& \dot y'_{A}(t) = - b(y'_A(t),A) \\
& y'_A(0) = z'
\end{cases}\, .\end{equation*}
Concatening the two trajectories we obtain a curve $Y' = y_a'\cup y_A'$ which leaves the simplex $\mS$. Moreover it leaves the  ergodic set $\mZ_0$ by the two opposite sides $\gamma^a_A$ and $\gamma^A_a$. Indeed $y_a'$ cannot intersect $\gamma^A_a$, and $y_A'$ cannot intersect $\gamma^a_A$ by the uniqueness theorem for ODE. It necessarily intersects the curve $Y$ since the latter  connects $a$ and $A$. 

Let $z''\in Y\cap Y'$. We assume without loss of generality that $z''= y_a(T)$ and also $z''= y'_A(T')$. Observe that we do not necessarily have $z''\in \mZ_{-\delta}$. However we do have $z''\in \mZ_0$ since the set $\mZ_0$ is stable. Finally we construct the control $\alpha$ as follows: (i) from time $t=0$ to $t=T$ we set $\alpha(t)\equiv a$, (ii) from $t=T$ to $t= T+T'$ we set $\alpha(t)\equiv A$. This control sends $z$ onto $z'$ within time $T+T'$.

The time $T'$ is clearly uniformly bounded. To prove that the time $T$ is also uniformly bounded for $z,z'\in \mZ_{-\delta}$, it is sufficient to avoid small neighbourhoods of the stationary points $e_a$ and $e_A$. The image of $\mZ_{-\delta}$ by the backward trajectories $y_a'$ and $y_A'$ inside $\mZ_0$ is a compact set which does not contain $e_a$ nor $e_A$. Therefore the time $T$ is uniformly bounded.  
\end{proof}

\paragraph{Step 3- Proof of the convergence towards $\lambda_{HJ}$ in the set $\mZ_{-\delta}$.}

We are now ready to prove an important lemma, which is a weaker version of Theorem \ref{th:eigHJ}.
\begin{lemma}
Under the same assumptions as above, there exists a constant $\lambda_{HJ}$ such that the following uniform convergence holds true:
\[   \lim_{\eps\to 0}\eps u_\eps(z) = \lambda_{HJ}\, , \quad \text{uniformly for $z\in \mathcal Z_{-\delta}$}\, . \] 
\end{lemma}

\begin{proof}
First we observe that the function $\eps u_\eps(y)$ is uniformly bounded on the simplex $\mS$:
\begin{equation} \label{eq:bound L}
\left| \int_0^{\infty} \eps e^{-\eps t}L(y_\alpha(t))\, dt \right| \leq \|L\|_\infty \int_0^{\infty} \eps e^{-\eps t}\, dt = \|L\|_\infty\, . 
\end{equation}
Taking the supremum over all possible controlled trajectories, we end up with $ \|\eps u_\eps \|_\infty \leq \|L\|_\infty$.

Second, we show that the family $(\eps u\eps)_\eps$ is equicontinuous. Let $z,z'\in \mZ_{-\delta}$. From Lemma~\ref{lem:controllability} there exists a time $T(z,z')$ and a controlled trajectory $z_\alpha(t)$ sending $z$ onto $z'$ within time $T(z,z')$.  By the dynamic programming principle, we have
\begin{align*} \eps u_\eps (z') -  \eps u_\eps (z) 
&\leq \eps u_\eps (z') - \eps \int_0^{T(z,z')} e^{-\eps t} L(z_\alpha(t))\, dt   - \eps e^{-\eps T(z,z')} u_\eps(z') \\
& \leq \eps u_\eps(z') \left( 1 - e^{-\eps T(z,z')}\right) + \|L\|_\infty \left( 1 - e^{-\eps T(z,z')} \right)
\, .    \end{align*}
Exchanging the roles of $z$ and $z'$, we obtain the following uniform bound:
\begin{equation} \label{eq:almost lipschitz} |\eps u_\eps (z') - \eps u_\eps (z)| \leq 2\|L\|_\infty \eps T(z,z')\, . \end{equation}
This simple estimate proves that (up to extraction), $\eps u_\eps (z)$ converges uniformly towards some constant. 
We learn from \cite{Arisawa2} that it converges in fact to a unique constant, not depending on the choice of the extraction (see the conclusion of the proof in Step~7). \end{proof}

\begin{remark} Any refinement in the estimate of the minimal control time $T(z,z')$ would bring additional information about the convergence of $u_\eps$. For instance, if we were able to prove that $T = \mathcal O(|z-z'|)$, then we would be able to conclude that $u_\eps$ is uniformly Lipschitz \eqref{eq:almost lipschitz}. Hence the Ascoli-Arzela theorem would yield convergence of $u_\eps -\eps^{-1} \lambda_{HJ}$ towards some eigenvector $\overline{u}$ up to extraction \cite{CDL}.
\end{remark}

\paragraph{Step 4- Flow of trajectories and local charts.}
In order to prove attractiveness of the ergodic set, we shall use two couples of charts to cover parts of the simplex $\mS$. The first couple consists in trajectories starting from the boundary of the simplex $\partial \mS$, and driven by constant controls $a-\delta$ or $A+\delta$. The second couple consists in trajectories starting from the line of eigenvectors $\Phi_0$, and driven by constant controls $a-\delta$ or $A+\delta$ too. 

\begin{figure}
\begin{center}
\includegraphics[width = 0.46\linewidth]{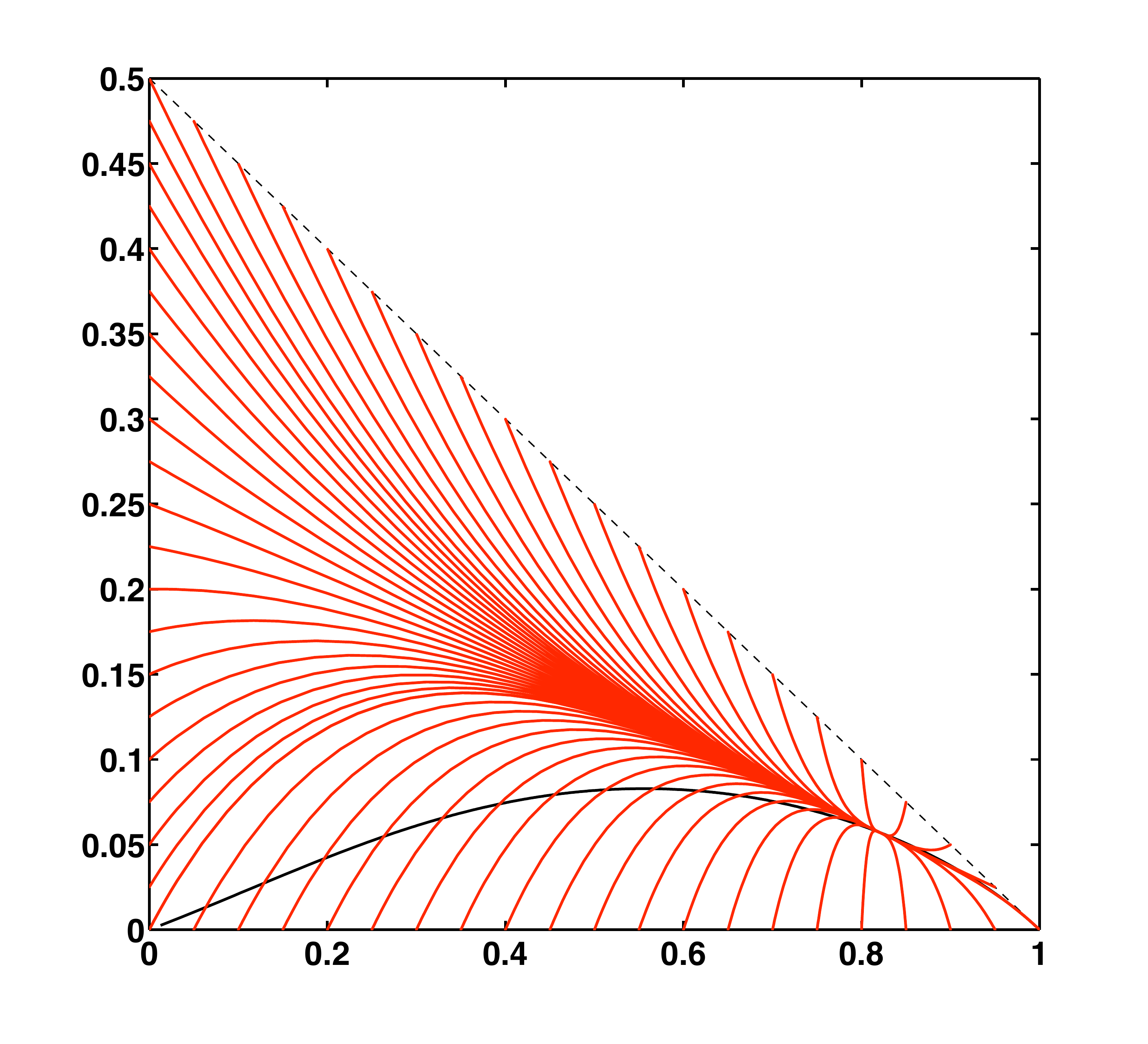}(a)\,\includegraphics[width = 0.46\linewidth]{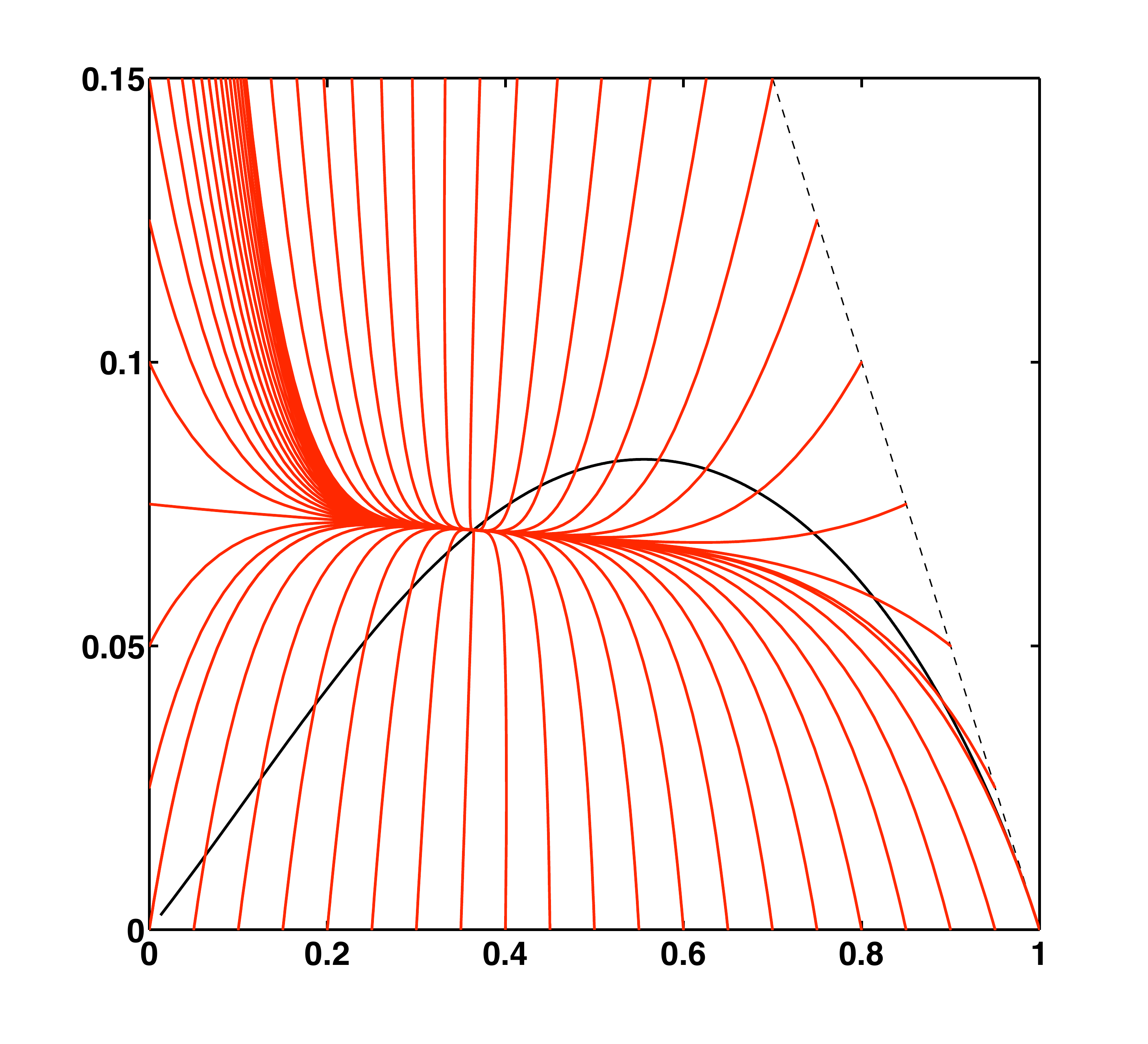}(b) 
\caption{(a) Grid of trajectories starting from the boundary of the simplex $\mS$, with constant control $\alpha(t)\equiv A+\delta$. (b) Similar figure with constant control $\alpha(t)\equiv a-\delta$.} \label{fig:local charts}
\end{center}
\end{figure}


We shall perform all the computation for one given set of charts (starting from the boundary $\partial\mS$ with constant control $A+\delta$). The other cases are similar.
We introduce the following family of curves, 
\begin{equation} \label{eq:chartA} 
 \Gamma_{A+\delta} \begin{cases} \frac{\partial}{\partial s} \Gamma(\theta,s) = b(\Gamma(\theta,s), A+\delta) \\ \Gamma(\theta,0) = Y_\theta  \in \partial\mS \end{cases}\, ,
\end{equation}
where the scalar $\theta$ is a parametrization of the boundary $\partial \mS$. More precisely we adopt a piecewise linear parametrization with the direct orientation. This set of charts covers the whole simplex $\mS$, see Figure \ref{fig:local charts}a. 

It is worth recalling that all these trajectories converge towards the eigenvector $e_{A+\delta}$ which lies on $\Phi_0$.
We denote $n_{A+\delta}(y) = \frac{\Theta b(y,A+\delta)}{|\Theta b(y,A+\delta)|}$ the unitary vector normal to the curve $\Gamma_{A+\delta}(\theta,\cdot)$ passing through $y = \Gamma_{A+\delta}(\theta,s)\in \mS$. We denote $\Gamma_{A+\delta} = \Gamma$ and $n_{A+\delta}= n$ throughout this step.

The key observation is that any trajectory $\dot y_\alpha(t) = b(y_\alpha(t),\alpha(t))$ is in fact moving monotically along the family of curves $\{\Gamma(\theta,\cdot)\}_\theta$ on the sets $\Phi_-$ and $\Phi_+$. We are able to quantify this phenomenon. 
Let $\alpha(t)$ be any measurable control which takes values in $[a,A]$. We have
\begin{align*} 
\la \dot y_\alpha(t), n(y_\alpha(t))  \ra 
&= \la (\alpha(t) - (A+\delta)) F y_\alpha(t) + b(y_\alpha(t),A+\delta), n(y_\alpha(t)) \ra \\
&= \dfrac{((A+\delta) - \alpha(t))}{|\Theta b(y_\alpha(t),A+\delta)|}\la \Theta F y_\alpha(t)  , b(y_\alpha(t),A+\delta) \ra \\
&=  \dfrac{((A+\delta) - \alpha(t))}{| b(y_\alpha(t),A+\delta)|}  \varphi(y_\alpha(t))\, . 
\end{align*}
We rewrite $y_\alpha(t)$ using the local chart $y_\alpha(t) = \Gamma(\theta(t),s(t))$. We have accordingly,
\[ \dot y_\alpha(t) =  \dot \theta(t) \dfrac{\partial}{\partial \theta}\Gamma(\theta(t),s(t)) + \dot s(t) \frac{\partial}{\partial s}\Gamma(\theta(t),s(t)) \, . \]
By construction, we have $\frac{\partial}{\partial s}\Gamma(\theta,s)=  b(\Gamma(\theta,s),A+\delta)$. Consequently,
\[ \la \dot y_\alpha(t), n(y_\alpha(t) \ra =  \dot \theta(t) \la \dfrac{\partial}{\partial \theta}\Gamma(\theta(t),s(t)), n(\Gamma(\theta(t),s(t))) \ra + 0\, . \] 
We compute the evolution of the quantity $\langle \frac{\partial\Gamma }{\partial \theta}, n(\Gamma)\rangle$ along the curve $\Gamma(\theta,\cdot)$. We aim at showing that this quantity has constant sign. This is intuitively clear, since trajectories cannot cross. We make it quantitative in the following calculation. For the sake of clarity, we write temporarily $b(\Gamma,A+\delta) = b(\Gamma)$.
\begin{align*}
\frac{\partial}{\partial s} \la \frac{\partial\Gamma }{\partial \theta}, n(\Gamma)\ra 
&= \la \frac{\partial^2\Gamma }{\partial \theta\partial s}  ,n(\Gamma)  \ra + \la \frac{\partial \Gamma }{\partial \theta }  ,\Theta Db(\Gamma) \frac{\partial \Gamma}{\partial s}    \ra - \la \frac{\partial\Gamma }{\partial \theta}, n(\Gamma)\ra \la n(\Gamma)  , \dfrac1{|\Theta b|} \Theta Db(\Gamma) \frac{\partial \Gamma}{\partial s}\ra \\
&= \la Db(\Gamma) \frac{\partial \Gamma }{\partial \theta }   ,n(\Gamma)  \ra + \la \frac{\partial \Gamma }{\partial \theta } ,\Theta Db(\Gamma)b(\Gamma)       \ra - \la \frac{\partial\Gamma }{\partial \theta}, n(\Gamma)\ra \la n(\Gamma)  ,  \Theta Db(\Gamma) \dfrac{b(\Gamma)}{|\Theta b(\Gamma)|}\ra \\
&= \la  \frac{\partial \Gamma }{\partial \theta }   , \left[ Db(\Gamma)^T \Theta + \Theta Db(\Gamma) \right] \dfrac{b(\Gamma)}{|\Theta b(\Gamma)|}  \ra - \la \frac{\partial\Gamma }{\partial \theta}, n(\Gamma)\ra \la n(\Gamma)  ,  \Theta Db(\Gamma) \Theta^{-1} n(\Gamma) \ra 
\end{align*}
Notice that the matrix $[ Db(\Gamma)^T \Theta + \Theta Db(\Gamma)]$ is everywhere skew-symmetric, so is the matrix $\Theta$:
\begin{equation*} \left[ Db(\Gamma)^T \Theta + \Theta Db(\Gamma) \right]^T =   -\Theta Db(\Gamma)   -  Db(\Gamma)^T\Theta  \, . \label{eq:lambda}\end{equation*}
Therefore it is equal to $\Theta$ up to a scalar factor:
\[ \left[ Db(\Gamma)^T \Theta + \Theta Db(\Gamma) \right]= \widetilde{\omega}(\theta, s) \Theta\, . \]
To conclude, the scalar product $\langle \frac{\partial\Gamma }{\partial \theta}, n(\Gamma)\rangle$ cannot vanish. In addition, we have the semi-explicit formula: 
\[ \la \frac{\partial\Gamma }{\partial \theta}, n(\Gamma)\ra  = \exp\left( \int_0^s \omega(\theta,s')\, ds'\right) \la \dfrac d{d\theta} Y_\theta, n(Y_\theta)\ra \, , \]
where the scalar $\omega$ is given by
\begin{align*} 
\omega &= \widetilde{\omega} - \la n(\Gamma)  ,  \Theta Db(\Gamma) \Theta^{-1} n(\Gamma) \ra \nonumber \\
& = \la n(\Gamma), \widetilde{\omega}n(\Gamma)\ra - \la n(\Gamma)  ,  \Theta Db(\Gamma) \Theta^{-1} n(\Gamma) 
\ra  \nonumber \\
& = \la n(\Gamma), \left[Db(\Gamma)^T   + \Theta Db(\Gamma) \Theta^{-1}\right]n(\Gamma)\ra - \la n(\Gamma)  ,  \Theta Db(\Gamma) \Theta^{-1} n(\Gamma) 
\ra \nonumber \\
& = \la n(\Gamma),  Db(\Gamma)^T  n(\Gamma)\ra \, .\end{align*}
To conclude, we have proven the following Lemma.
\begin{lemma}[A monotonicity formula]\label{lem:dotbeta}
Let $(\theta_{A+\delta},s_{A+\delta})$ be some parametrization of the simplex, given by \eqref{eq:chartA}. The evolution of $\theta_{A+\delta}$ along the trajectories of $\dot y_\alpha(t) = b(y_\alpha(t),\alpha(t))$ is given by the following formula:
\begin{equation*} \dot \theta_{A+\delta}(t) =  ((A+\delta)-\alpha(t) ) \varphi(y_\alpha(t))  \exp\left(- \int_0^{s (t)} \omega(\theta ,s')\, ds' \right) \la \frac{d}{d \theta} Y_\theta, n_{A+\delta}(Y_\theta) \ra^{-1}\, ,  \label{eq:thetaA}\end{equation*}
where the scalar $\omega$ is defined by
\begin{equation*}  \label{eq:muA}
\omega(\theta,s) = \la  D_yb(\Gamma(\theta,s),A+\delta) n(\Gamma(\theta,s)), n(\Gamma(\theta,s))\ra\, .
\end{equation*}
In particular $\omega$ is uniformly bounded: $\|\omega\|_\infty \leq \|D_y b\|_\infty$.
\end{lemma}

We build the same chart $(\theta_{a-\delta},s_{a-\delta})$ with the constant control $\alpha(t)\equiv a - \delta$,
\begin{equation} \label{eq:charta} 
 \Gamma_{a-\delta} \begin{cases} \frac{\partial}{\partial s} \Gamma(\theta,s) = b(\Gamma(\theta,s), a-\delta) \\ \Gamma(\theta,0) = Y_\theta  \in \partial\mS \end{cases}\, .
\end{equation} 
We have similarly
\begin{equation*} \dot \theta_{a-\delta}(t) =  ((a-\delta)-\alpha(t) ) \varphi(y_\alpha(t))  \exp\left(- \int_0^{s(t)} \omega(\theta,s')\, ds' \right) \la \frac{d}{d \theta} Y_\theta, n_{a-\delta}(Y_\theta) \ra^{-1}\, .  
\end{equation*}
%
%
We recall that we adopt a piecewise linear parametrization with the direct orientation for the boundary $\partial \mS$. Since the fields $b(y,\alpha)$ are all pointing inward the simplex $\mS$, this guarantees that we have in both cases ($A+\delta$ and $a-\delta$),
\[ \forall \theta\quad \la \dfrac{d}{d\theta} Y_\theta , n(Y_\theta) \ra \leq 0\, . \]

\begin{figure}
\begin{center}
\includegraphics[width = 0.46\linewidth]{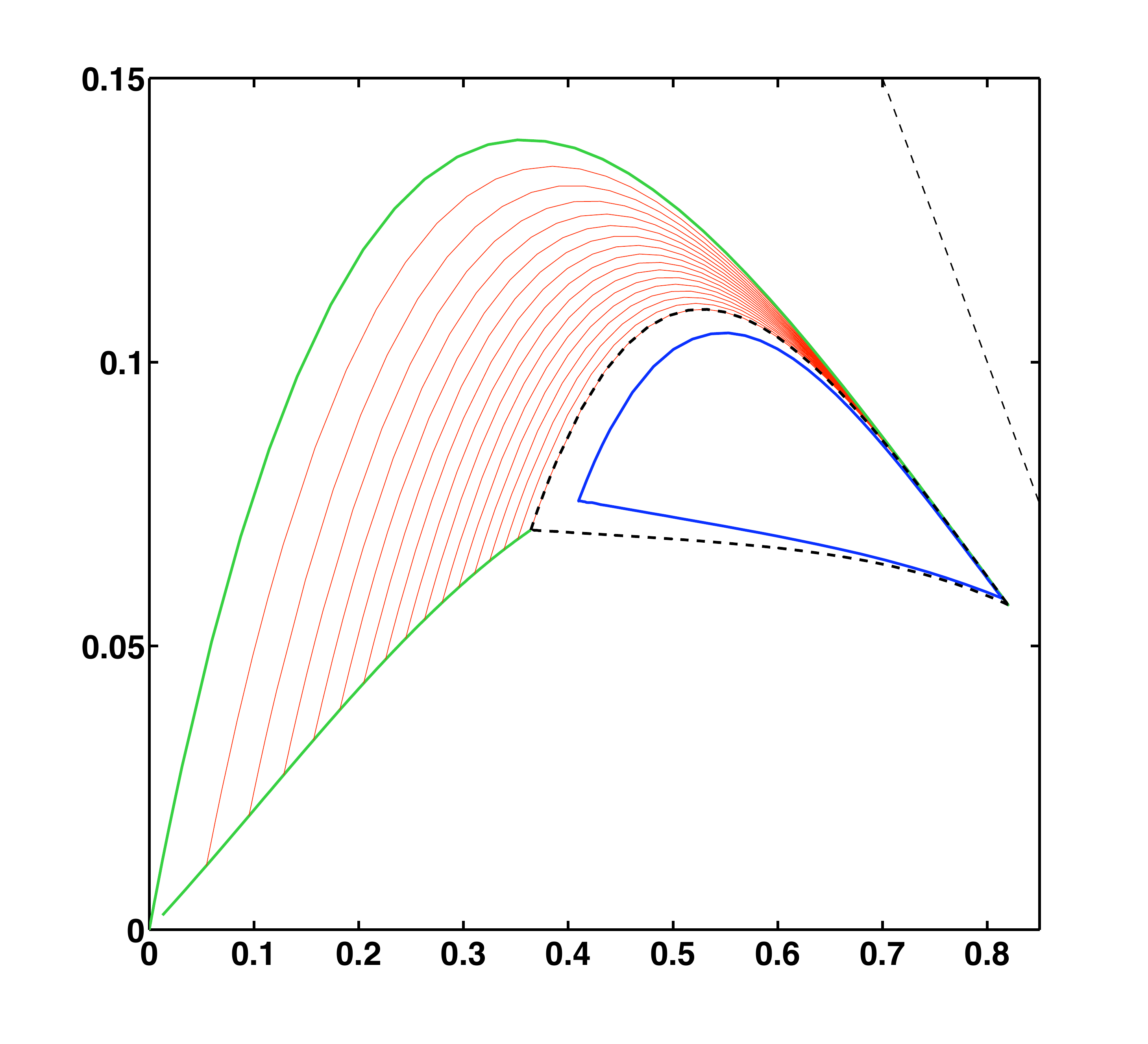}(a)\,\includegraphics[width = 0.46\linewidth]{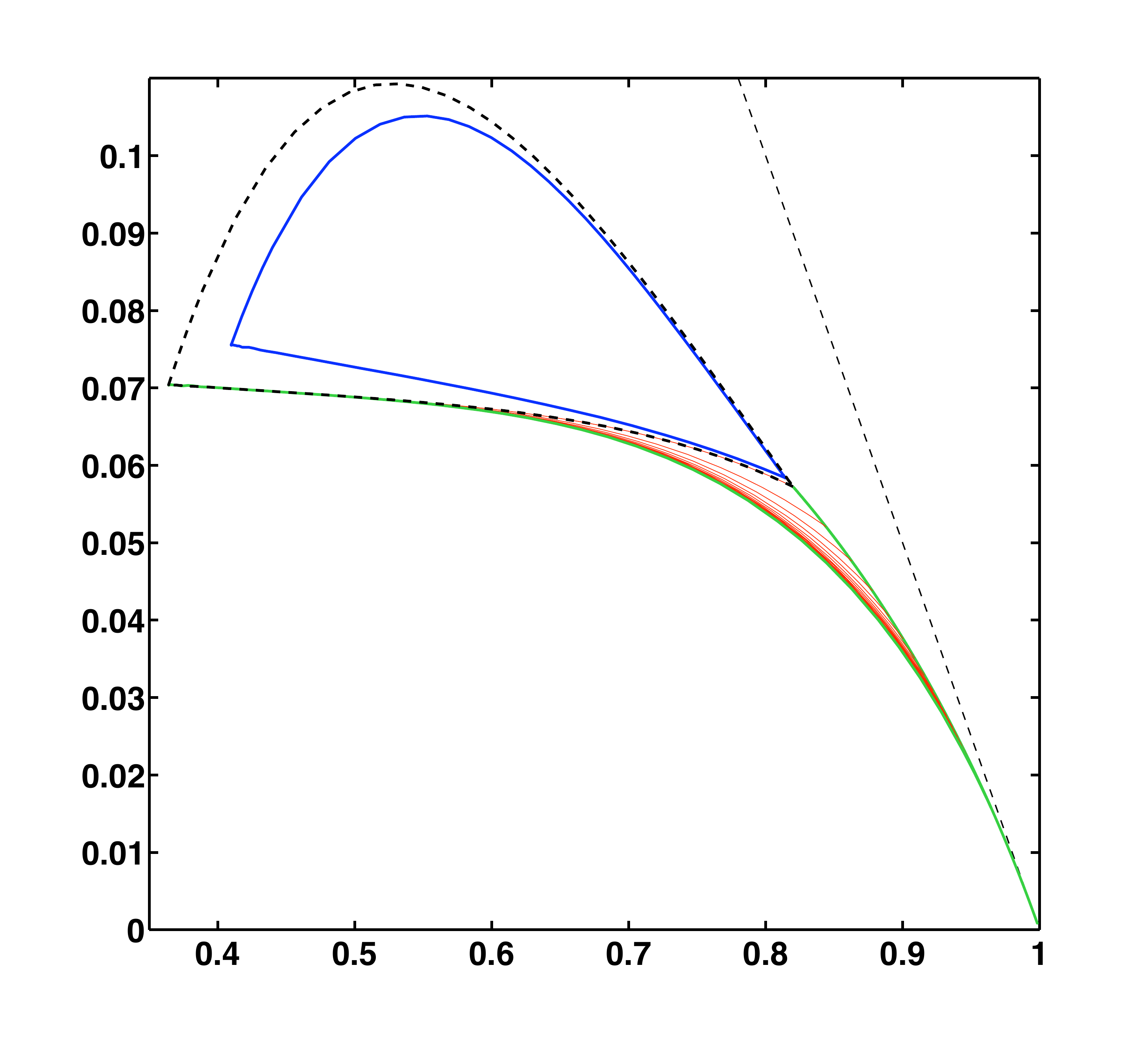}(b) 
\caption{Tunneling effect: the trajectories are trapped in the areas delimited by the green curves. They are forced to enter the approximated ergodic set $\mZ_{+\delta}$ due to the monotonicity formulas \eqref{eq:betaA}-\eqref{eq:betaa}. (a) The set $\mT_-\subset \Phi_-$; (b) The set $\mT_+\subset \Phi_+$.} \label{fig:tunneling}
\end{center}
\end{figure}

Similarly we build the second couple of charts, defined through the following families of curves, parametrized by the starting point $e_\beta \in \Phi_0$, see Figure \ref{fig:tunneling},
\begin{equation} \label{eq:chart-tunnel} 
\Gamma_{A+\delta}\begin{cases} \frac{\partial}{\partial s} \Gamma (\beta,s) = b(\Gamma (\beta,s), A+\delta) \\ \Gamma (\beta,0) = e_\beta  \in \Phi_0 \end{cases}\, , \quad \Gamma_{a-\delta}\begin{cases} \frac{\partial}{\partial s} \Gamma (\beta,s) = b(\Gamma (\beta,s), a-\delta) \\ \Gamma (\beta,0) = e_\beta  \in \Phi_0 \end{cases}\, .
\end{equation}
We obtain as in Lemma~\ref{lem:dotbeta},
\begin{align} 
\dot \beta_{A+\delta}(t) &=  ((A+\delta)-\alpha(t) ) \varphi(y_\alpha(t))  \exp\left(- \int_0^{s (t)} \omega(\beta ,s')\, ds' \right) \la \frac{d}{d \beta} e_\beta, n_{A+\delta}(e_\beta) \ra^{-1}\, ,  \label{eq:betaA} \\
\dot \beta_{a-\delta}(t) &=  ((a-\delta)-\alpha(t) ) \varphi(y_\alpha(t))  \exp\left(- \int_0^{s (t)} \omega(\beta ,s')\, ds' \right) \la \frac{d}{d \beta} e_\beta, n_{a-\delta}(e_\beta) \ra^{-1}\, .  \label{eq:betaa}
\end{align}
In this case, we have $b(e_\beta,A + \delta) = ((A + \delta) - \beta) F e_\beta$. Therefore,
\begin{equation} \forall \beta\quad \la \dfrac{d}{d\beta} e_\beta , n_{A+\delta}(e_\beta) \ra  = \sign((A + \delta) - \beta)\la \dfrac{d e_\beta}{d\beta}   , \dfrac{\Theta F e_\beta}{|\Theta F e_\beta|} \ra \, . \label{eq:tangentialA} \end{equation} 
Assumption {\bf (H4)}~\eqref{as:tech} allows to guarantee that this quantity has a constant sign for $\beta<A+\delta$.
Notice that when the charts are defined with the control $a - \delta$, the formula \eqref{eq:tangentialA} is replaced with the following:
\begin{equation*} \forall \beta\quad \la \dfrac{d}{d\beta} e_\beta , n_{a-\delta}(e_\beta) \ra  = \sign((a - \delta) - \beta)\la \dfrac{d e_\beta}{d\beta}  , \dfrac{\Theta F e_\beta}{|\Theta F e_\beta|} \ra \, . 
\end{equation*} 
Again this has constant sign for $\beta>a-\delta$.
Unfortunately the second set of charts does not cover the whole simplex $\mS$ (Figure \ref{fig:tunneling}). It is the reason why we introduce the transient set of charts parametrized by $Y_\theta \in \partial \mS$. We are particularly interested in the two tunnels $\mT_-,\mT_+$ defined by the following curves:
\begin{itemize}
\item the subset $\mT_-$ is defined by its boundary made of 3 pieces (Figure~\ref{fig:tunneling}a): the trajectory starting from the corner $e_0 = (0\; 0\; 1/3)$ with constant control $A+\delta$; the portion of $\Phi_0$ from $e_0$ to $e_{a-\delta}$; and the trajectory starting from $e_{a-\delta}$ with constant control $A+\delta$. We have $\mT_- \subset \Phi_-$, 
\item the subset $\mT_+$ is defined similarly (Figure~\ref{fig:tunneling}b): it is enclosed by the trajectory starting from the corner $e_\infty = (1\; 0\; 0)$ with constant control $a-\delta$; the portion of $\Phi_0$ from $e_{\infty}$ to $e_{A+\delta}$; and the trajectory starting from $e_{A+\delta}$ with constant control $a-\delta$. We have $\mT_+ \subset \Phi_+$.
\end{itemize}
By comparison arguments we get easily that the second set of charts covers the tunnel $\mT_-$ (resp. $\mT_+$) with the constant control $A+\delta$ (resp. $a-\delta$). 
The main feature of these two tunnels is the following: due to Hypothesis  {\bf (H4)}~\eqref{as:tech} and Lemma \ref{lem:dotbeta}, any trajectory that enters one of these tunnels can exit only by entering $\mZ_{+\delta}$. Moreover it follows monotonically the parametrization of the second chart \eqref{eq:betaA}-\eqref{eq:betaa}. This enables to prove that the trajectory eventually reaches the end of the tunnel: the ergodic set.


\paragraph{Step 5- Attractiveness of the ergodic set.} 
%

\begin{figure}
\begin{center}
\includegraphics[width = 0.46\linewidth]{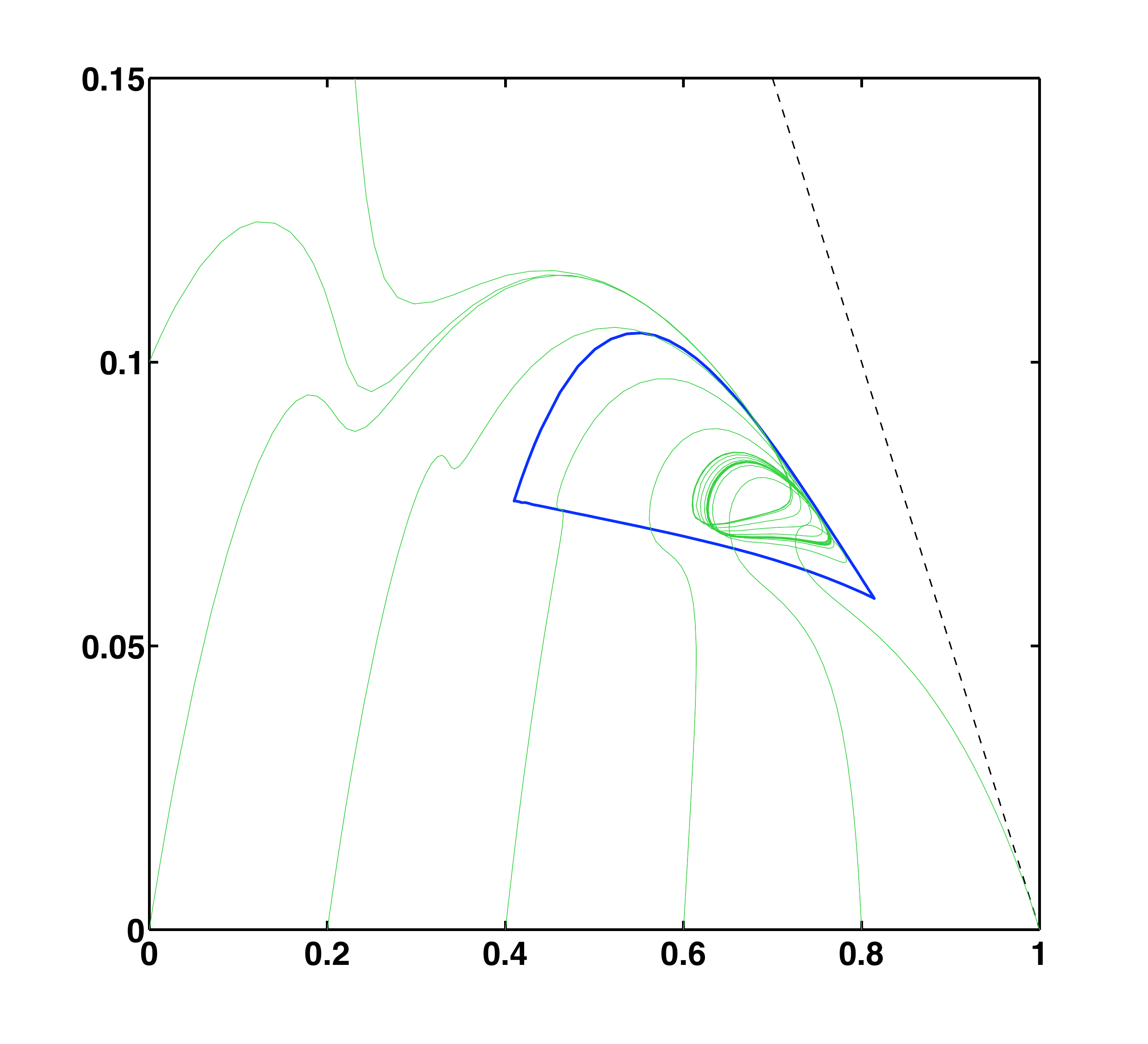}(a)\,\includegraphics[width = 0.46\linewidth]{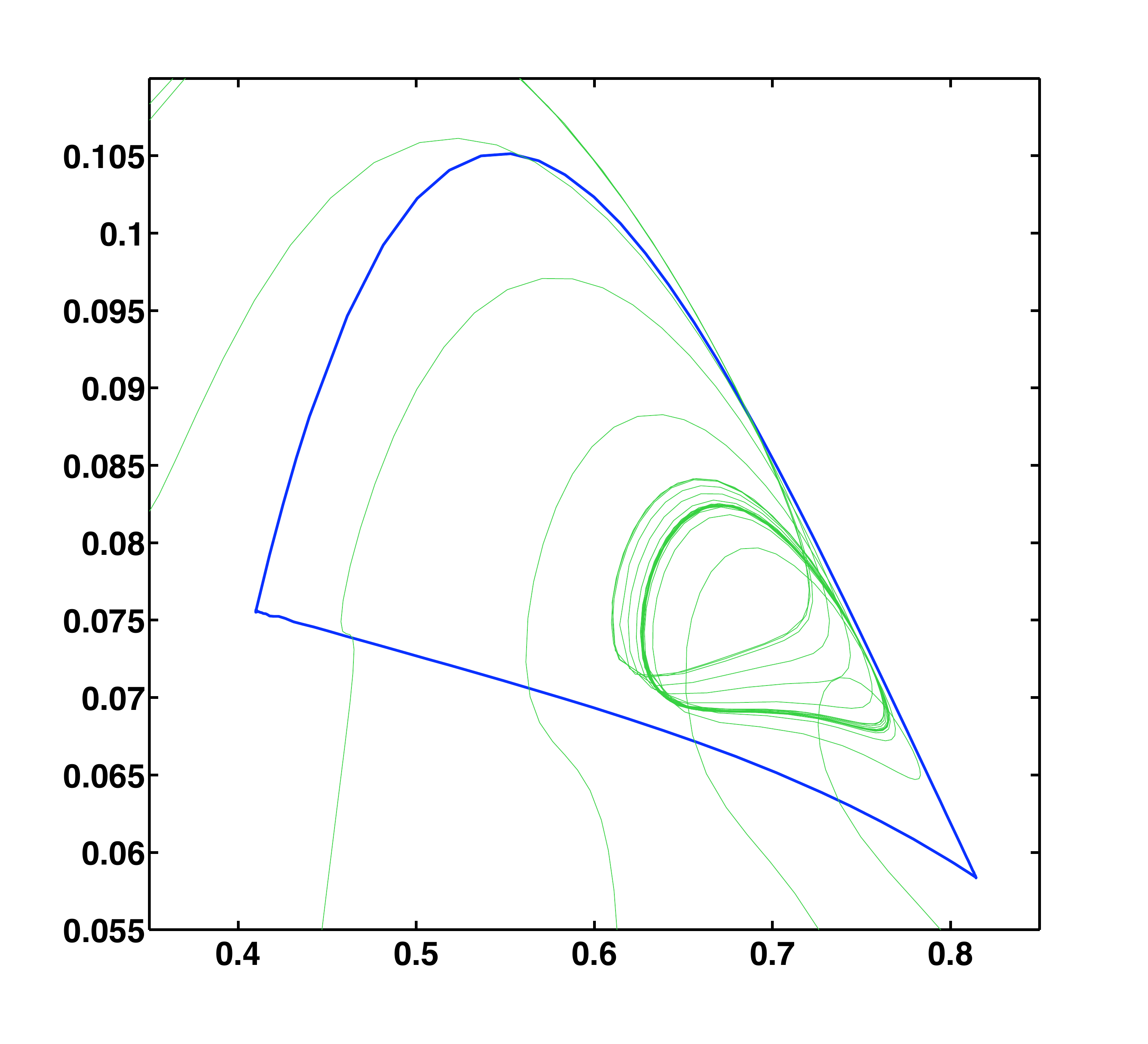}(b) 
\caption{Attractiveness and stability of the ergodic set $\mZ_0$. Several trajectories with arbitrary control $\alpha(t)$ (here, periodic) have been plotted. They all enter the ergodic set and stay there. The right figure is a zoom of the left one.}
\label{fig:attractiveness}
\end{center}
\end{figure}

We  state the main result of this step: attractiveness of the approximated ergodic set $\mZ_{+2\delta}$.

\begin{lemma}
The set $\mZ_{+2\delta}$ is attractive: any trajectory $y_\alpha(t)$ enters the set $\mZ_{+2\delta}$ after finite time, at most $T_0(\delta)$. The maximal time $T_0(\delta)$ does not depend on $\eps$.
\end{lemma}

As this property holds true for any $\delta>0$ we can conclude that the set $\mZ_0$ is approximately attractive, but we get no estimate about the maximal time $T_0(\delta)$ as $\delta\to 0$. In fact we will prove in the next step that close-to-optimal trajectories enter the ergodic set $\mZ_0$ (and even the smaller set $\mZ_{-\delta}$) in finite time: this is a stronger property. 
The attractiveness property is illustrated in Figure \ref{fig:attractiveness}: trajectories necessarily enter the set $\mZ_{+2\delta}$. 

\begin{proof}
Let $y_\alpha(t)$ be any trajectory of the system $\dot y_\alpha(t) = b(y_\alpha(t),\alpha(t))$, where the control $\alpha(t)$ takes values in $[a,A]$.
We assume without loss of generality that $y_\alpha(0)\in \Phi_+\setminus \mZ_{+2\delta}$. We define $t_0$ as the first exit time:
\[t_0 = \inf_t \left\{ t \;|\; y_\alpha(t) \notin \Phi_+\setminus \mZ_{+2\delta}  \right\} \, . \]

We face the following alternative: either  $y_\alpha(0)\notin \mT_+$ then we shall use the first set of charts parametrized by $Y_\theta \in \partial\mS$ \eqref{eq:charta}, or $y_\alpha(0)\in \mT_+$ then we shall use the second set of charts parametrized by $e_\beta \in \Phi_0$ \eqref{eq:chart-tunnel}. We begin with the latter case.

The only way to exit the tunnel $\mT_+$ is to enter the set $\mZ_{+\delta}\subset \mZ_{+2\delta}$. The quantity $\beta_{a-\delta}(t)$ is monotonic. More precisely we have
\begin{equation*} 
\forall t\in [0,t_0) \quad \dot \beta_{a-\delta}(t) \leq - \delta    \varphi(y_\alpha(t))    \exp\left( - s(t) \|D_y b\|_\infty \right) \left|\la \frac{d e_\beta}{d \beta}  , \dfrac{\Theta F e_\beta}{|\Theta F e_\beta|} \ra\right|^{-1}
\, .
\end{equation*}
Recall that the time $s(t)$ is given by the parametrization of the trajectory $y_\alpha(t) = \Gamma_{a-\delta}(\beta_{a-\delta}(t),s(t))$. Hence $s(t)$ is uniformly bounded on $\Phi_+\setminus \mZ_{+2\delta}$. This is the reason why we use the set $\mZ_{+2\delta}$ rather than the set $\mZ_{+\delta}$ for which we could have  $s(t)\to +\infty$ because $e_{a-\delta}\in \partial \mZ_{+\delta}$.
Therefore the r.h.s of \eqref{eq:dissipation beta} can vanish only when $y_\alpha$ gets close to the set $\Phi_0$. 
Indeed the factor $\left|\la \frac{d e_\beta}{d \beta}  , \frac{\Theta F e_\beta}{|\Theta F e_\beta|} \ra\right|^{-1}$ is clearly bounded from below since $\left|\la \frac{d e_\beta}{d \beta}  , \frac{\Theta F e_\beta}{|\Theta F e_\beta|} \ra\right|$ is bounded from above.

We have a precise description of the dynamics close to $\Phi_0$:
\begin{align}
\dfrac d{dt} \varphi(y_\alpha(t)) & = \la \dot y_\alpha(t) , \nabla \varphi(y_\alpha(t)) \ra \nonumber\\
& = \la b(y_\alpha(t),\alpha(t)), \nabla \varphi(y_\alpha(t)) \ra   \nonumber\\
& = \la b(e_\beta,\alpha(t)), \nabla \varphi(e_\beta) \ra + o(|y_\alpha(t) - e_\beta|) \label{eq:linearization}\\ 
& = (\alpha(t) - \beta) \la Fe_\beta , \nabla \varphi(e_\beta) \ra + o(|y_\alpha(t) - e_\beta|)  \nonumber\, . 
\end{align} 
Hypothesis {\bf (H4)} guarantees that the quantity $\la Fe_\beta , \nabla \varphi(e_\beta) \ra$ is  negative for all finite $\beta$:
\[\forall \beta < +\infty \quad \la Fe_\beta , \nabla \varphi(e_\beta) \ra = \la Fe_\beta , -\chi_\beta \Theta \dfrac {d e_\beta}{d\beta} \ra = \chi_\beta\la \Theta Fe_\beta ,   \dfrac {d e_\beta}{d\beta} \ra < 0 \, , \]
where $\chi_\beta$ is a positive scalar (coming from the discrepancy between the two  vectors  pointing in the normal direction to $\Phi_0$: $\nabla \varphi(e_\beta)$ and $-\Theta \frac {d e_\beta}{d\beta}$). Unfortunately 
$\la b(e_\infty,\alpha(t)), \nabla \varphi(e_\infty) \ra$ may vanish (it is actually the case for the running example: the field $b(e_\infty,\alpha)$ is independent of $\alpha$ and it is tangent to $\Phi_0$). Therefore we have to isolate the large parameters $\beta\gg1$ (trajectory starting close to $e_\infty$). To circumvent this issue we investigate the behaviour of $y_\alpha(t)$ in the neighbourhood of $e_\infty$, say a ball $\mB(e_\infty,\eta')$. 
We have the following linearization around $e_\infty$,
\begin{equation}
b(y,\alpha) = b(e_\infty,\alpha) + o(|y- e_\infty|) = Ge_\infty + o(\eta')\, , 
\label{eq:linearization e_infty}
\end{equation}
since $Fe_\infty = 0$ (see Hypothesis {\bf (H2)} and subsequent discussion).
Moreover we have $Ge_\infty\neq 0$ (see Hypothesis {\bf (H3)} and subsequent discussion). The linearization is uniform with respect to $\alpha \in [a,A]$. Therefore the trajectory is pushed away from $e_\infty$ in finite time, uniformly with respect to the initial data $y_\alpha(0)\in \mB(e_\infty,\eta')$.

We choose $\eta'>0$ such that the higher-order terms in the linearization \eqref{eq:linearization e_infty} are negligible in front of $Ge_\infty$.  We choose the largest possible $B$ such that the portion $\{e_\beta: \beta \geq B \}\subset  \mB(e_\infty,\eta')$. There exists  $\nu >0$ depending on $B$ such that the following inequality holds,
\[\forall \beta \leq B \quad - \la Fe_\beta , \nabla \varphi(e_\beta) \ra > \nu  >0 \, . \] 
Accordingly we choose a small $\delta'>0$ such that the error term in the linearization \eqref{eq:linearization} is uniformly smaller than $\delta \nu $ for $|y_\alpha(t) - e_\beta|<\delta'$ and $\beta \leq B$. 
Thus we have
\begin{equation*}
\forall t\in [0,t_0)\quad \dfrac d{dt} \varphi(y_\alpha(t))  \geq   2 \delta \nu - \delta \nu = \delta \nu \, .
\end{equation*} 
This defines a narrow band of size $\delta'$ around the portion $\{e_\beta: A+2\delta\leq \beta \leq B \}\subset\Phi_0$ in which the quantity $\varphi(y_\alpha(t))$ is  increasing, uniformly in time. Therefore the trajectory stays away from $\Phi_0$. Consequently the quantity $ \beta_{a-\delta}(t)$ is decreasing, uniformly in time, and $\beta_{a-\delta}(t)$ gets smaller than $A+2\delta$ in finite time. This means that $y_\alpha(t)$ enters $\mZ_{+2\delta}$ in finite time when $y_\alpha(0)\in \mT_+$.

We now consider the case $y_\alpha(0)\notin \mT_+$.
We choose the parametrization \eqref{eq:charta} given by the charts associated to the constant control $a -\delta$, and the initial point $Y_\theta$ lying on $\partial \mS$. 
From \eqref{eq:betaA} we deduce that $\theta_{a-\delta}(t)$ is nonincreasing up to $t_0$: 
\begin{equation} \label{eq:dissipation beta} 
\forall t\in [0,t_0)\quad \dot \theta_{a-\delta}(t) \geq \delta   \varphi(y_\alpha(t))   \exp\left( - s(t) \|D_y b\|_\infty \right) \left|\la \frac{d}{d \theta} Y_\theta, n_{a-\delta}(Y_\theta) \ra\right|^{-1}
\, .
\end{equation}

The function $\theta_{a-\delta}(t)$ is naturally bounded for $t\in [0,t_0)$. In fact, the point $Y_{\theta}$ cannot make a complete turn around the boundary $\partial \mS$. For instance, when $Y_{\theta} = e_0$, we have necessarily $y_\alpha(t)\in \Phi_-$, since the trajectory $\gamma_{a-\delta}^0$ lies in $\Phi_-$. We denote by $\overline \theta$ a bound from above for $\theta_{a-\delta}(t)$.

Recall that the time $s(t)$ is uniformly bounded on $\Phi_+\setminus \mZ_{+2\delta}$. As previously the factor $\left|\la \frac{d}{d \theta} Y_\theta, n_{a-\delta}(Y_\theta) \ra\right|^{-1}$ is clearly bounded from below since $\left|\la \frac{d}{d \theta} Y_\theta, n_{a-\delta}(Y_\theta) \ra\right|$ is bounded from above. 

The descriptions of the dynamics close to $\Phi_0$ \eqref{eq:linearization} and $e_\infty$ \eqref{eq:linearization e_infty} are still valid. From \eqref{eq:linearization} we deduce that the following estimate holds true for $|y_\alpha(t) - e_\beta|<\delta'$ and $\beta\leq a-2\delta$,
\begin{equation*}
\dfrac d{dt} \varphi(y_\alpha(t))  \leq  - 2 \delta \nu + \delta \nu = - \delta \nu\, .
\end{equation*} 
This defines a narrow band of size $\delta'$ around the portion $\{e_\beta: 0\leq \beta \leq a-2\delta \}\subset\Phi_0$ in which the quantity $\varphi(y_\alpha(t))$ is uniformly  decreasing. Once it enters this narrow band, the trajectory necessarily crosses $\Phi_0$ in finite time. It leaves $\Phi_+$ and it enters the tunnel $\mT_-$.

On the other hand, we build an analogous narrow band around the portion $\{e_\beta: A+2\delta\leq \beta \leq +\infty  \}\subset\Phi_0$. It is the union of the tunnel   $\mT_+ $ and the ball $B(e_\infty,\eta')$. 

If $y_\alpha(0)$ starts within $\mB(e_\infty,\eta')$, the trajectory will leave the ball after a finite time of order $\tau'  = 2 |Ge_\infty| \eta'$. 
We cannot rule out the fact that it comes back inside the ball several times. 
Alternatively, we can evaluate the time the trajectory $y_\alpha$ spends inside the balls $\mB(e_\infty,\eta')$ and $\mB(e_\infty,2\eta')$ each time it travels through $\mB(e_\infty,\eta')$. Clearly the trajectory spends a time at most $\tau'  = 2 |Ge_\infty| \eta'$ inside the first ball, and a time at least $\tau'$ in the annulus $\mB(e_\infty,2\eta')\setminus \mB(e_\infty,\eta')$. Outside the narrow bands, the quantity $\theta_{a-\delta}$ is monotonic with the following estimate: 
\[ \dot \theta_{a-\delta}(t) \geq c_{\eta'}  \, , \]
for some constant $c_{\eta'}>0$. 
%
As a conclusion, if we denote by $N$ the number of times the trajectory $y_\alpha$ travels through $\mB(e_\infty,\eta')$, we have 
\[
\theta_{a-\delta}(t_0) - \theta_{a-\delta}(0)  \geq      \underbrace{ N\left (\tau' . 0 + \tau'  c_{\eta'}  \right)  }_{\mB(e_\infty,2\eta')} + \underbrace{(t_0  -   2N\tau') c_{\eta'}}_{\text{outside the narrow bands}}  \, .
\]
Therefore, the number $N$ is bounded. Moreover we have $N\tau' \leq  \frac{\overline{\theta}}{c_{\eta'} }$. Consequently, the time $t_0$ is also bounded, and we have 
\[ t_0 \leq   2N \tau' + \frac{\overline{\theta}}{c_{\eta'}} = 3 \frac{\overline{\theta}}{c_{\eta'} } \, . \]


We have finally proven that the trajectory $y_\alpha(t)$ leaves the set $\Phi_+\setminus \mZ_{+2\delta}$ after finite time. A precise follow-up of the time spent in this zone shows that it is uniformly bounded with respect to the initial data $y_\alpha(0)$. 
We now face the following alternative: either the trajectory enters the set $\mZ_{+2\delta}$ at time $t_0$, or it enters the set $\Phi_-\setminus \mZ_{+2\delta}$. In the former case there is nothing more to say. In the latter case, the trajectory crosses the line $\Phi_0$. It does so on the portion $\{e_{\beta}\;|\; 0\leq \beta\leq a-2\delta \}$ (otherwise the vector field $b(e_\beta,\alpha(t_0))$ points in the wrong direction, from $\Phi_-$ to $\Phi_+$, and the trajectory lies in the tunnel $\mT_+$). We conclude that the trajectory enters the tunnel $\mT_-$ at $t = t_0$.

From $t = t_0$ and $y_\alpha(t_0)\in \mT_-$ we repeat the same procedure as for $y_\alpha(0)\in \mT_+$, except that we change the control determining the charts: we switch from $\Gamma_{a-\delta}$ to $\Gamma_{A+\delta}$. Again, the quantity $\beta_{A+\delta}(t)$ is monotonic, and we have
\begin{equation*}
\forall t\geq t_0 \quad \dot \beta_{A+\delta}(t) \geq \delta   \left|\varphi(y_\alpha(t))\right|   \exp\left( - s(t) \|D_y b\|_\infty \right) \left|\la \frac{d e_\beta}{d \beta}  , \dfrac{\Theta F e_\beta}{|\Theta F e_\beta|} \ra\right|^{-1}
\, .
\end{equation*}
This quantity is positive, and we can prove as above that the quantity $\varphi(y_\alpha(t))$ is uniformly decreasing around the portion $\{e_\beta: 0\leq \beta \leq a-2\delta \}$.
We deduce as previously that the trajectory necessarily enters the set $\mZ_{+2\delta}$ in finite time $T_0(y)$. The time is uniformly bounded  with respect to the trajectory $y_\alpha$, independently of $\eps$. 
\end{proof}

\paragraph{Step 6- Close-to-optimal trajectories do not stay close to the boundary of the ergodic set.}

Let $T_0(\delta)$ be the maximal time of entry in the set $\mZ_{+2\delta}$ for the trajectory $\dot y_\alpha(t) = b(y_\alpha(t),\alpha(t))$. Let us mention that the set $\mZ_{+2\delta}$ is stable (The proof is similar to the stability of $\mZ_0$, see Lemma \ref{lem:stability}).  

The final step in the proof of Theorem \ref{th:eigHJ} consists in carefully following trajectories running in the layer $\Delta_\delta = \mZ_{+2\delta}\setminus \mZ_{-\delta}$. It is a narrow band surrounding the boundary of the ergodic set $\partial \mZ_0$.  Heuristically, if the trajectory stays in this band forever, then it follows closely the two trajectories defining $\partial \mZ_0$ \eqref{eq:def Z0}. Thus it spends most of the time in the neighbourhood of $e_a$ and $e_A$, where the reward values are respectively $\lambda(a)$ and $\lambda(A)$. From Hypothesis {\bf (H2)}, and particularly $\lambda(\alpha^*)> \max\left(\lambda(a),\lambda(A)\right) $, we deduce that such a trajectory is far from optimality.

In the next lines we develop this argumentation with quantitative bounds. We will pay much attention to proving that the maximal exit time $T_1(\eps,\delta)$ outside $\Delta_\delta$ is such that $\eps T_1(\eps,\delta) \ll 1$, uniformly. 

The width of the narrow band $\Delta_\delta = \mZ_{+2\delta}\setminus \mZ_{-\delta}$ is of order $\delta$. We introduce an intermediate scale $\eta$, with $\delta\ll \eta \ll 1$.
We denote $B_a = B(e_a;2\eta)$ and $B_A = B(e_A;2\eta)$ the balls of radius $\eta$ with respective centers $a,A$.
The scales $\delta$ and $\eta$ are to be determined later. Roughly speaking, the scale $\eta$ is determined such that $\eta \ll  \lambda(\alpha^*) -  \max\left(\lambda(a),\lambda(A)\right) $, and $\delta$ is adjusted accordingly.

Our claim is the following: if $\delta$ is chosen small enough, then close-to-optimal trajectories cannot stay in $\Delta_\delta$ forever. Let us mention that the set $\mZ_{-\delta}$ is not stable. However we require only that the trajectory enters once in $\mZ_{-\delta}$ where we can use the uniform controllability Lemma~\ref{lem:controllability}.

\begin{lemma}\label{lem:close-to-optimal}
Let $y_\alpha(t)$ be a close-to-optimal trajectory in the following sense: 
\[    \int_0^{+\infty} e^{-\eps t} L(y_\alpha(t))\, dt \geq   u_\eps(y) - O(1) \, .   \]
This trajectory must leave the narrow band $\Delta_\delta$ before some maximal time $T_1(\eps,\delta)$. In addition we have $\eps T_1(\eps,\delta) = O(\eta)$ as $\eps\to 0$, uniformly for $y\in \mS$.  
\end{lemma}

\begin{proof}
We begin with some considerations about optimality of trajectories. A first important remark is that for all $y\in \mS$ we have
$\liminf_{\eps\to 0}\eps u_\eps(y) \geq \lambda(\alpha^*)$, uniformly in $\mS$. Indeed we may choose the constant value control $\alpha(t)\equiv \alpha^*$ in the optimization problem \eqref{eq:infinite horizon}. We know that exponential convergence occurs towards the eigenvector $e_{\alpha^*}$: 
\[ | y_{\alpha^*}(t) - e_{\alpha^*}| \leq C |y -  e_{\alpha^*}|e^{-\mu_{\alpha^* }t}\, ,  \]
where $\mu_{\alpha^*}$ is the spectral gap and $C$  denotes some absolute constant. This yields the estimate
\begin{align*}
\eps \int_0^{+\infty} e^{-\eps t} L(y_{\alpha^*}(t))\, dt & = \eps \int_0^{+\infty} e^{-\eps t} L(e_{\alpha^*})\, dt +  \eps \int_0^{+\infty} e^{-\eps t} \left( L(y_{\alpha^*}(t)) - L(e_{\alpha^*})\right)\, dt  \\
& \geq \lambda(\alpha^*) -  \|D L\|_\infty \eps \int_0^{+\infty} e^{-\eps t} | y_{\alpha^*}(t) - e_{\alpha^*}| \, dt \\
& \geq \lambda(\alpha^*) -  C \|D L\|_\infty |y -  e_{\alpha^*}| \eps \int_0^{+\infty} e^{-\eps t}e^{-\mu_{\alpha^* }t} \, dt \\
& \geq \lambda(\alpha^*) -  C   \|D L\|_\infty \dfrac{\eps}{\eps + \mu_{\alpha^*}}\, .
\end{align*}
By letting $\eps \to 0 $ we deduce the estimate 
\begin{equation} \label{eq:lbHJ>Lb*}
\liminf_{\eps\to 0}\eps u_\eps(y) \geq \lambda(\alpha^*) \, , \quad \text{uniformly for $y\in \mS$} \, .
\end{equation}

Now we examine the reward of trajectories during their passage in $\Delta_\delta$. Let $y_\alpha$ be such that 
\[ \forall t\in (T_0(y), T_1(y)) \quad y_\alpha(t)\in \Delta_\delta\, ,\]
(note that we may {\em a priori} have $T_1(y) = +\infty$).
We distinguish between the different parts of $\Delta_\delta$. Either $y_\alpha(t)$ belongs to $\Delta_\delta\cap(B_a\cup B_A)$ or it belongs to the complementary set. We begin with the latter case. We assume that 
\[\forall t\in (t_1,t_2)\quad y_\alpha(t) \in \Delta_\delta\setminus(B_a\cup B_A)\, .\]
We first prove that, outside the balls $B_a$ and $B_A$, $y_\alpha(t)$ is close to either the curve $\gamma_a$ or the curve $\gamma_A$. We deal with the upper part $\Delta_\delta\cap\Phi_-$, where the curve $\gamma_A$ lies, without loss of generality. 
We use Lemma \ref{lem:dotbeta} for this purpose. We recall \eqref{eq:betaA} adapted to the constant value control $A$ (instead of $A + \delta)$:
\begin{equation*} \dot \beta_A(t) \geq   ( A -\alpha(t) ) \left|\varphi(y_\alpha(t))\right|  \exp\left( - s(t) \|D_y b\|_\infty \right) \left|\la \frac{d e_\beta}{d \beta}  , \dfrac{\Theta F e_\beta}{|\Theta F e_\beta|} \ra\right|^{-1}\, .  \end{equation*}
Outside the balls $B_a$ and $B_A$, the time $s(t)$ is uniformly bounded, and $|\varphi|$ is bounded from below. We have 
\[ \left|\varphi(y_\alpha(t))\right|  \exp\left( - s(t)  \|D_y b\|_\infty  \right) \left|\la \frac{d e_\beta}{d \beta}  , \dfrac{\Theta F e_\beta}{|\Theta F e_\beta|} \ra\right|^{-1} \geq c_\eta\, ,  \]
for some positive constant $c_\eta$.
Thus, 
\begin{equation*}
   c_\eta \int_{t_1}^{t_2}(A - \alpha(t)) \, dt \leq \beta(t_2) - \beta(t_1) \leq  3 \delta \, .
\end{equation*}
We deduce that 
\begin{align*}
\dfrac d{dt} | \gamma_A(t+\tau) - y_\alpha(t) |^2 & = 2 \la b(\gamma_A(t+\tau),A) - b(y_\alpha(t),\alpha(t)),  \gamma_A(t+\tau) - y_\alpha(t)  \ra\\
& = 2 \la b(\gamma_A(t+\tau),\alpha(t)) - b(y_\alpha(t),\alpha(t)),  \gamma_A(t+\tau) - y_\alpha(t)  \ra 
\\ & \quad + 2 (A - \alpha(t)) \la F\gamma_A(t+\tau),  \gamma_A(t+\tau) - y_\alpha(t)  \ra  \\
& \leq 2 \|D_y b\|   |\gamma_A(t+\tau)- y_\alpha(t) |^2 + C (A - \alpha(t)) \, ,
\end{align*}
where $C$ denotes some absolute constant.
Therefore, 
\begin{equation*} 
| \gamma_A(t+\tau) - y_\alpha(t) |^2 
\leq  e^{2\|D_y b\|(t-t_1)}| \gamma_A(t_1+\tau) - y_\alpha(t_1)|^2 + C   \int_{ t_1}^t e^{2 \|D_y b\|(t'-t)} (A - \alpha(t'))  \, dt'\, .
\end{equation*}
We choose a suitable time shift $\tau$ such that $\gamma_A(t_1+\tau) - y_\alpha(t_1) = O(\delta)$. This yields the estimate
\begin{equation}\label{eq:tracking} 
\forall t\in (t_1,t_2)\quad | \gamma_A(t+\tau) - y_\alpha(t) |^2 
\leq   C \delta^2 e^{2\|D_y b\|(t-t_1)}  + C \dfrac{\delta}{c_\eta}\, .
\end{equation}
Because it is a converging trajectory, we know that $\gamma_A$ spends at most a time $T_\eta$ in the zone outside the ball $ B(e_A, \eta)$. 
We choose $\delta$ so small that the following inequality holds,
\begin{equation*}
C \delta^2  e^{2\|D_y b\|T_\eta}  +   C \dfrac{\delta}{c_\eta} <  \eta^2  \, .\end{equation*}
From  the tracking estimate \eqref{eq:tracking} we deduce that during the time $t\in (t_1,t_1 + T_\eta)$ we have the following estimate, 
\[  \sup_{t\in (t_1,t_1 + T_\eta)} |y_\alpha(t) - \gamma_A(t+\tau)|^2 \leq C \delta^2  e^{2\|D_y b\|T_\eta}  +   C \dfrac{\delta}{c_\eta} <  \eta^2  \, . \]
Therefore we have $y_\alpha(t_1+T_\eta) \in B_A = B(e_A,2 \eta)$. And so we have $t_2 \leq t_1 + T_\eta$ (but this estimate is inessential). 

Concerning the trajectory $\gamma_A$, exponential convergence  towards the eigenvector $e_A$ holds true:
\begin{equation} \label{eq:spectral gap} \forall t>0\quad| \gamma_A(t+\tau) - e_A|  \leq  C | \gamma_A(\tau) - e_A| e^{-\mu_A t}\, ,  \end{equation}
where $\mu_A>0$ is the spectral gap, and $C$  denotes some absolute constant. In fact we replace $\mu_A$ by $\underline{\mu} = \min(\mu_a,\mu_A)$ in \eqref{eq:spectral gap} for the sake of convenience.
Plugging this estimate into the optimal reward, we get
\begin{align*} 
\eps \int_{t_1}^{t_2} e^{-\eps t'} L(y_\alpha(t'))\, dt' 
& = \eps \int_{t_1}^{t_2} e^{-\eps t'}  L(\gamma_A(t'+\tau)) \, dt' + \eps \int_{t_1}^{t_2} e^{-\eps t'} \left[ L(y_\alpha(t')) - L(\gamma_A(t'+\tau))\right]\, dt' \\
& \leq \eps \int_{t_1}^{t_2} e^{-\eps t'}  L(e_A) \, dt' +  \eps \int_{t_1}^{t_2} e^{-\eps t'} \left[ L(\gamma_A(t'+\tau)) - L(e_A)\right]\, dt' \\
& \quad + \eta \|DL\|_\infty \left( e^{-\eps t_1} -e^{-\eps t_2} \right) \\
& \leq \lambda(A) \left( e^{-\eps t_1} - e^{-\eps t_2} \right) +  C \|DL\|_\infty \dfrac{\eps}{\eps + \underline{\mu}} \left( e^{-(\eps + \underline{\mu}) t_1} - e^{-(\eps + \underline{\mu}) t_2} \right) \\
& \quad + \eta \|DL\|_\infty \left( e^{-\eps t_1} -e^{-\eps t_2} \right) \,.
\end{align*}
We conclude that there exists some absolute constant $C$ such that
\begin{equation}\label{eq:t0-t1}
\eps \int_{t_1}^{t_2} e^{-\eps t'} L(y_\alpha(t'))\, dt' \leq
 \overline{\lambda} \left( e^{-\eps t_1} - e^{-\eps t_2} \right)   + \eps C \left( e^{-(\eps + \underline{\mu}) t_1} - e^{-(\eps + \underline{\mu}) t_2} \right) + \eta C\left( e^{-\eps t_1} -e^{-\eps t_2} \right) \, ,
\end{equation}
where we have introduced the notation: $\overline{\lambda}  = \max(\lambda(a),\lambda(A))$.

We now turn to the second (easier) case:
\[\forall t\in (t_2,t_3)\quad y_\alpha(t) \in \Delta_\delta\cap(B_a\cup B_A)\, .\]
We assume that $y_\alpha(t)$ lies inside the ball $B_A$, without loss of generality. We directly have:
\begin{align*} 
\eps \int_{t_2}^{t_3} e^{-\eps t'} L(y_\alpha(t'))\, dt' 
& = \eps \int_{t_2}^{t_3} e^{-\eps t'}  L(e_A) \, dt' + \eps \int_{t_2}^{t_3} e^{-\eps t'} \left[ L(y_\alpha(t')) - L(e_A)\right]\, dt'   \\
& \leq \lambda(A) \left( e^{-\eps t_2} - e^{-\eps t_3} \right) +  \eta \|DL\|_\infty \left( e^{-\eps t_2} -e^{-\eps t_3} \right) \, .
\end{align*}
We conclude that there exists  some absolute constant $C$ such that we also have
\begin{equation}\label{eq:t1-t2}
\eps \int_{t_2}^{t_3} e^{-\eps t'} L(y_\alpha(t'))\, dt' \leq
 \overline{\lambda} \left( e^{-\eps t_2} - e^{-\eps t_3} \right)
   + \eps C \left( e^{-(\eps + \underline{\mu}) t_2} - e^{-(\eps + \underline{\mu}) t_3} \right) + \eta C\left( e^{-\eps t_2} -e^{-\eps t_3} \right) \, .
\end{equation}

All in all we add successively estimations \eqref{eq:t0-t1} and \eqref{eq:t1-t2}. Using telescopic cancellations, we end up with
\begin{equation}\label{eq:T1-T2}
\eps \int_{T_0}^{T_1} e^{-\eps t'} L(y_\alpha(t'))\, dt' \leq
 \overline{\lambda} \left( e^{-\eps T_0} - e^{-\eps T_1} \right)
   + \eps C \left( e^{-(\eps + \underline{\mu}) T_0} - e^{-(\eps + \underline{\mu}) T_1} \right) + \eta C\left( e^{-\eps T_0} -e^{-\eps T_1} \right) \, .
\end{equation}
This proves that, for $\eps$ and $\eta$ small enough, we necessarily have $T_1(y)<+\infty$. Otherwise the trajectory would not be optimal thanks to the estimate \eqref{eq:lbHJ>Lb*}, and also the trivial bound: 
\[\eps \int_{0}^{T_0} e^{-\eps t'} L(y_\alpha(t'))\, dt' \leq  \eps T_0(\delta) \|L\|_\infty\, .\]

This prove that close-to-optimal trajectories necessarily enter the set $\mZ_{-\delta}$ before some maximal time $T_1(\eps,\delta)$.


In fact we can estimate better the maximal time $T_1(\eps,\delta)$, and prove that $\eps T_1(\eps,\delta)= O(\eta)$ as $\eps\to 0$. Let $y_\alpha$ be a close-to-optimal trajectory, starting from $y\in \mS$. Following \eqref{eq:T1-T2} we have 
\begin{align*}
\eps\left( u_\eps(y) - O(1)\right) & \leq \eps \int_0^{T_1} e^{-\eps t'} L(y_\alpha(t'))\, dt' + \eps e^{-\eps T_1}   u_\eps (y_\alpha(T_1))\\
& \leq \eps\int_0^{T_0} e^{-\eps t'} L(y_\alpha(t'))\, dt' + \eps\int_{T_0}^{T_1} e^{-\eps t'} L(y_\alpha(t'))\, dt' + e^{-\eps T_1(y)}  \eps u_\eps (y_\alpha(T_1)) \\
& \leq \eps T_0(\delta) \|L\|_\infty + \overline{\lambda} \left( 1 - e^{-\eps T_1(y)} \right) + O(\eps) + O(\eta) + e^{-\eps T_1(y)}  \eps u_\eps (y_\alpha(T_1)) \, .
\end{align*}
We deduce
\[ e^{-\eps T_1(y)} \left( \eps u_\eps (y_\alpha(T_1)) -  \overline{\lambda} \right) \geq \eps u_\eps(y) -  \overline{\lambda} + O(\eps T_0(\delta)) + O(\eps) + O(\eta)\, .  \]
Therefore we have
\begin{equation}\label{eq:eps T1} \eps T_1(y) \leq \log\left( \dfrac{\eps u_\eps (y_\alpha(T_1)) -  \overline{\lambda}}{\eps u_\eps(y) -  \overline{\lambda} + O(\eps T_0(\delta)) + O(\eps) + O(\eta)} \right) \, .\end{equation}
We strongly use the fact that $y_\alpha(T_1)\in \mZ_{-\delta}$. There exists a time $T_2(\delta)$ (not depending on $\eps$) such that we can send $y$ onto $y_\alpha(T_1)$ within time $T_2(\delta)$. For this, simply send $y$ onto a point $y'\in \mZ_{-\delta}$ (for example with constant control $\alpha(t)\equiv \alpha^*$), then connect $y'$ to $y_\alpha(T_1)$ thanks to Lemma \ref{lem:controllability}. Notice that the time $T_2(\delta)$ may degenerate as $\eta$ (and thus $\delta$) goes to $0$. 
Using the dynamic programming principle we get,
\begin{equation*}
\eps u_\eps(y) \geq \eps \int_0^{T_2} e^{-\eps t'} L(y_\alpha(t'))\, dt' + \eps e^{-\eps T_2(\delta)}   u_\eps (y_\alpha(T_1))\,.
\end{equation*}
Plugging this estimate in \eqref{eq:eps T1}, we conclude
\begin{align*}
\eps T_1(y) &\leq \log\left( \dfrac{\eps u_\eps (y_\alpha(T_1)) -  \overline{\lambda}}{ e^{-\eps T_2(\delta)} \eps   u_\eps (y_\alpha(T_1)) -  \overline{\lambda} + O(\eps T_2(\delta))+ O(\eps T_0(\delta)) + O(\eps) + O(\eta)} \right) \\
& \leq - \log\left(  1 + \dfrac{ \left(e^{-\eps T_2(\delta)} - 1\right) \eps   u_\eps (y_\alpha(T_1))  + O(\eps T_2(\delta))+ O(\eps T_0(\delta)) + O(\eps) + O(\eta)}{\eps u_\eps (y_\alpha(T_1)) -  \overline{\lambda}} \right)
\, .  
\end{align*}
Finally, we obtain the estimate
\[ \eps T_1(y) = O(\eta) \quad \text{as $\eps \to 0$}\, , \quad \text{uniformly for $y\in \mS$}\, . \]
This concludes the proof of Lemma \ref{lem:close-to-optimal}.
\end{proof}

\paragraph{7- Conclusion.}
We are now ready to prove Theorem \ref{th:eigHJ}.

\begin{proof}[Proof of Theorem \ref{th:eigHJ}]
First, the function $\eps u_\eps(y)$ is uniformly bounded on the simplex $\mS$: $\|\eps u_\eps \|_\infty \leq \|L\|_\infty$ \eqref{eq:bound L}. Second, we show equicontinuity of the family of functions $(\eps u_\eps)_\eps$. Let $y,y'\in \mS$. We assume without loss of generality that $y$ and $y'$ lie outside the ergodic set $\mZ_0$. From Steps 5 and 6 we deduce that the trajectories $y_\alpha(t)$ and $y'_\alpha(t)$ enter the approximated ergodic set $\mZ_{-\delta}$ within  times $T_1(y)$ and $T_1(y')$ respectively, such that $\eps T_1 = O(\eta)$. From Step 2 we can connect $z = y_\alpha(T_1(y))$ and $z' = y'_\alpha(T_1(y'))$ within time $T_2(\delta)$ which is independent of $\eps$. From the  dynamic programming principle, we have
\begin{align*} \eps u_\eps (y') -  \eps u_\eps (y) 
&\leq \eps \int_0^{T_1(y')} e^{-\eps t} L(y'_\alpha(t))\, dt  + \eps e^{-\eps T_1(y')} u_\eps(z') \\
&\hspace{20mm} - \eps \int_0^{T_1(y)} e^{-\eps t} L(y_\alpha(t))\, dt - \eps e^{-\eps T_1(y)} u_\eps(z) + O(\eps) \\
& \leq \eps u_\eps(z') - \eps u_\eps(z) + 2 \eps T_1(y') \|L\|_\infty   + 2 \eps T_1(y) \|L\|_\infty +O(\eps)
\\
& \leq   O( \eps T_2(\delta))  + O(\eps T_1 ) + O(\eps)
\, .    \end{align*}
The last estimate was obtained at Step 3 (controllability argument). We end up with
\[ \eps u_\eps (y') -  \eps u_\eps (y) 
\leq O( \eps T_2(\delta))  + O(\eta ) + O(\eps)\, . \]
Exchanging the roles of $y$ and $y'$ we obtain the uniform bound,
\begin{equation*} 
\forall (y,y')\in \mS\times \mS \quad |\eps u_\eps (y') - \eps u_\eps (y)| \leq O(\eta) \quad \mbox{as $\eps\to 0$}\, . \end{equation*}
Since the parameter $\eta$ can be chosen arbitrarily small, we get that the family $(\eps u_\eps)_\eps$ converges uniformly towards some constant (up to extraction). In fact the limit is unique \cite{Bardi-Capuzzo,Arisawa2}.
we re-do the proof of this statement for the sake of completeness.

Assume there exist two constant values $l_1<l_2$ and to subsequences $(u_{\eps_1}), (u_{\eps_2})$ such that $\lim_{\eps_1\to 0} \eps_1 u_{\eps_1} = l_1$ and  $\lim_{\eps_2\to 0} \eps_2 u_{\eps_2} = l_2$ uniformly in $\mS$. We recall that $-u_\eps$ is the viscosity solution of the stationary Hamilton-Jacobi equation \eqref{eq:HJ stat}. Since the convergence is uniform, we can find $\nu,\eps_1$ and $\eps_2$ small enough such that $-\eps_1 u_{\eps_1}\geq -l_1 - \nu$, and $-\eps_2 u_{\eps_2}\leq -l_2 + \nu < -l_1 - \nu$. Therefore,  $-u_{\eps_1}$ is a viscosity subsolution of 
\begin{equation}\label{eq:eps1} -l_1 - \nu + H(y,D_y u_{\eps_1}) \leq 0\,,\end{equation}
and  $-u_{\eps_2}$ is a viscosity supersolution of the same equation,
\begin{equation}\label{eq:eps2}  -l_1 - \nu + H(y,D_y u_{\eps_2}) \geq  0 \, . \end{equation} 
We deduce from standard comparison theorems \cite{Bardi-Capuzzo,CDL} that $\forall y\in \mS\; - u_{\eps_1}(y) \leq - u_{\eps_2}(y)$. However we can add to $u_{\eps_2}$ a large positive constant such that $u_{\eps_2} > u_{\eps_1}$, and \eqref{eq:eps1}--\eqref{eq:eps2} are still verified. This is a contradiction.

We conclude that there is a unique possible limit for the sequence $(\eps u_{\eps})_\eps$. This completes the proof of Theorem \ref{th:eigHJ}.
\end{proof}

\subsection{Proof of Corollary \ref{cor:ergodic}}

The proof of Corollary \ref{cor:ergodic} following Theorem \ref{th:eigHJ} is contained in \cite[Theorem 5]{Arisawa2}. We repeat the argument for the sake of completeness. We fix $\delta>0$, and we choose $T = \frac \delta \eps$. We split the rewards as follows:
\begin{equation}\label{eq:cor}
\eps u_\eps(y) = \sup_{\alpha} \left\{ \eps\int_0^{T} L(y_\alpha(t))\, dt + \eps \int_0^{T} \left( e^{-\eps t} - 1 \right)L(y_\alpha(t))\, dt  + \eps e^{-\eps T}   u_\eps(y_\alpha(T)) \right\}\, .
\end{equation}
We have for the second contribution,
\begin{equation*}
\left| \eps \int_0^{T} \left( e^{-\eps t} - 1 \right)L(y_\alpha(t))\, dt \right| \leq \|L\|_\infty \left( e^{-\delta } - 1 + \delta \right)\, .
\end{equation*}
Therefore, dividing \eqref{eq:cor} by $\delta$ we get as $\eps\to 0$ (or equivalently $T\to +\infty$),
\begin{align*}
& \liminf_{T\to +\infty} \dfrac{1 }{T} \sup_{\alpha} \left\{ \int_0^{T} L(y_\alpha(t))\, dt \right\} +  \dfrac{\left( e^{-\delta} - 1\right)}\delta \lambda_{HJ} = O(\delta)\, , \\
& \limsup_{T\to +\infty} \dfrac{1 }{T} \sup_{\alpha} \left\{ \int_0^{T} L(y_\alpha(t))\, dt \right\} +  \dfrac{\left( e^{-\delta} - 1\right)}\delta \lambda_{HJ} = O(\delta)\, ,
\end{align*}
where we have used the uniform convergence $\eps u_\eps \to \lambda_{HJ}$ in $\mS$. Since $\delta>0$ can be chosen arbitrarily small, Corollary \ref{cor:ergodic} is proven.

\section{Numerical simulations and perspectives}

\label{sec:num}

\begin{figure}
\begin{center}
\includegraphics[width = 0.6\linewidth]{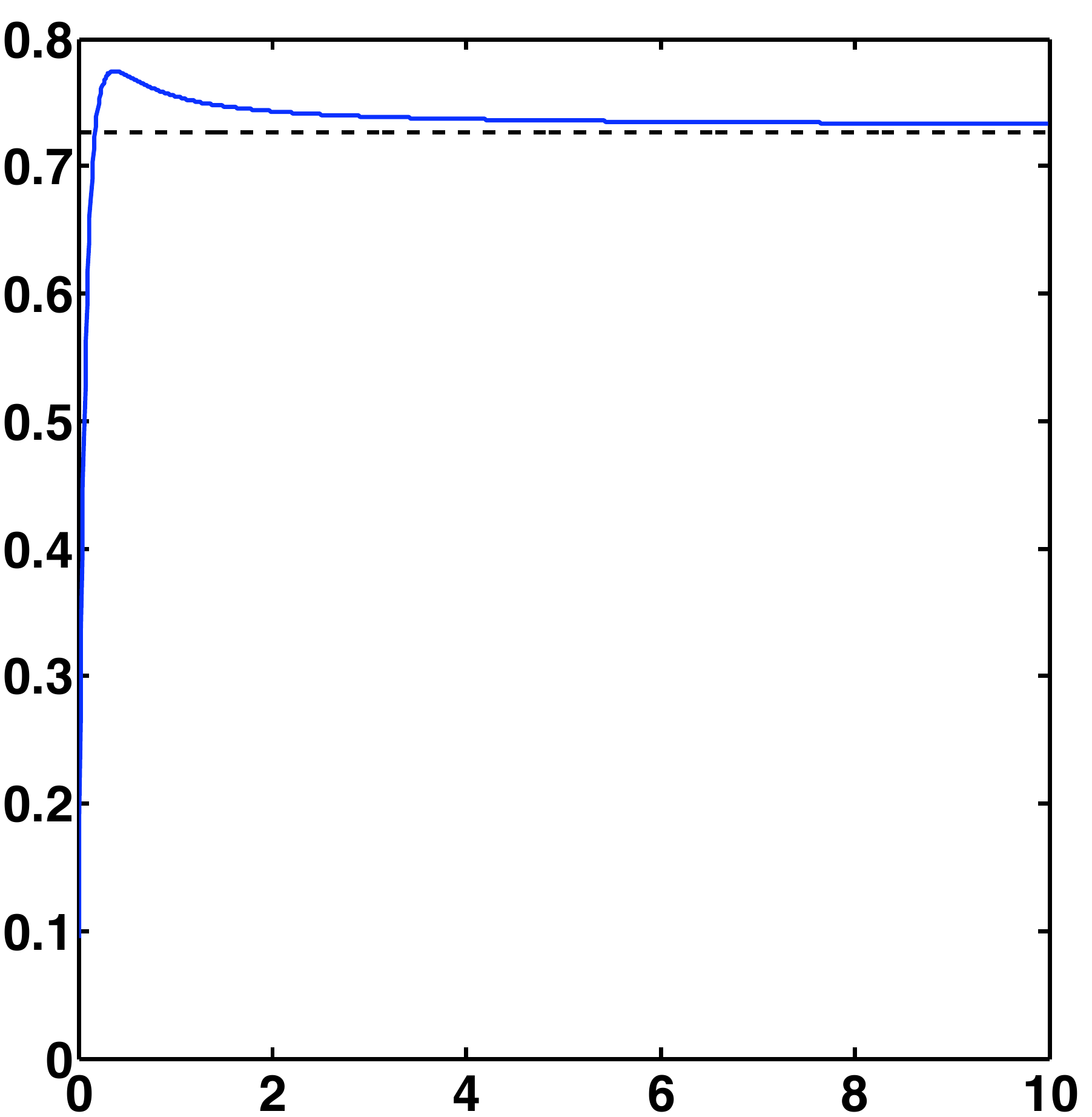}
\caption{Convergence of $\frac{u(T,y)}{T}$ towards a constant value $\lambda_{HJ}$. We have plotted a specific value at an arbitrary point $y_0\in \mS$.
We observe that the limiting value $\lambda_{HJ}$ is close to the maximal Perron eigenvalue $\lambda_P\approx 0.7273$ (dashed line, see also Figure \ref{fig:Perron}).
The discrepancy between the limiting value and $\lambda_P$ (dashed line) falls below the numerical error due to the scheme, which is of order one.} \label{fig:eigenvalue}
\end{center}
\end{figure}

\begin{figure}
\begin{center}
\includegraphics[width = 0.46\linewidth]{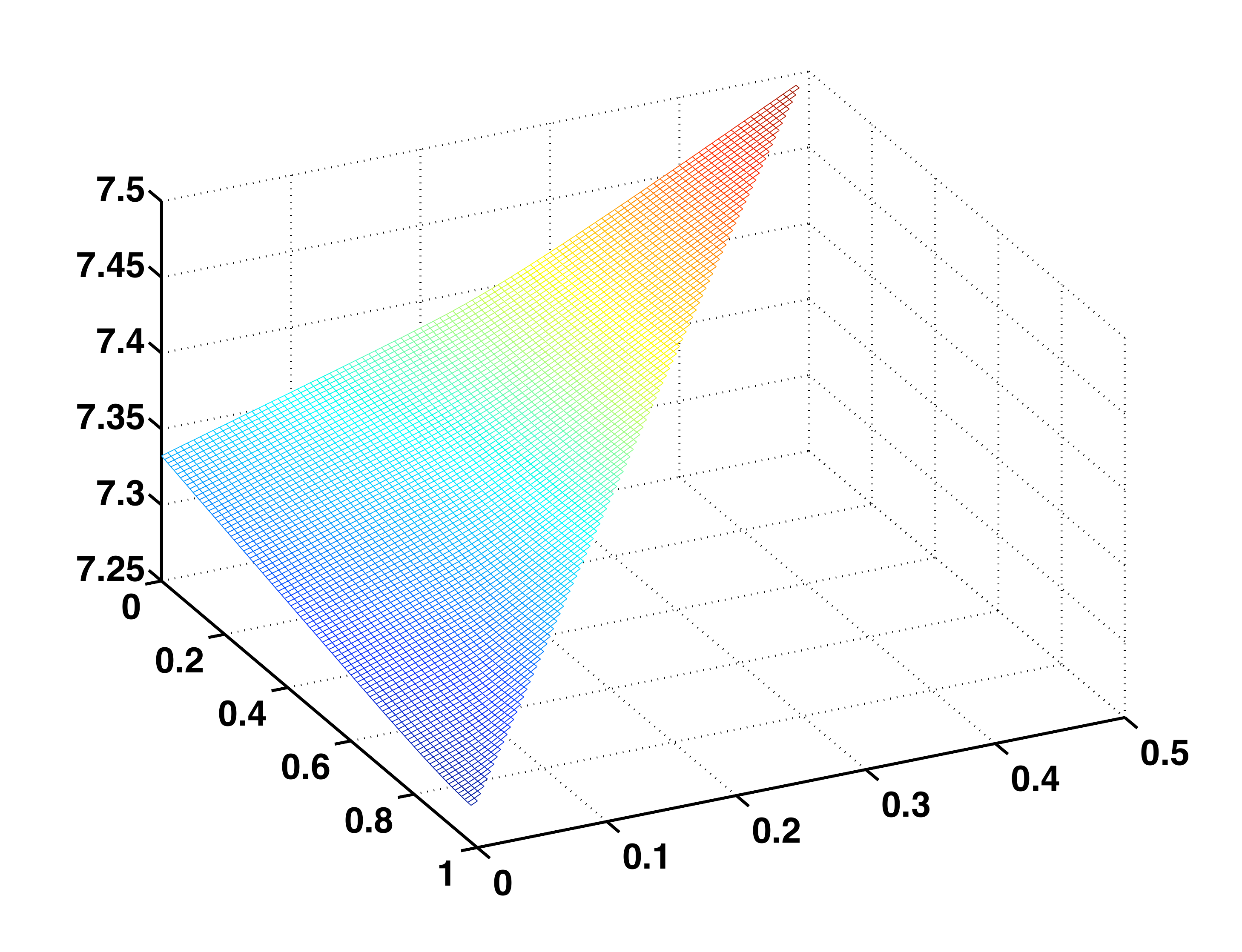}(a)\,\includegraphics[width = 0.46\linewidth]{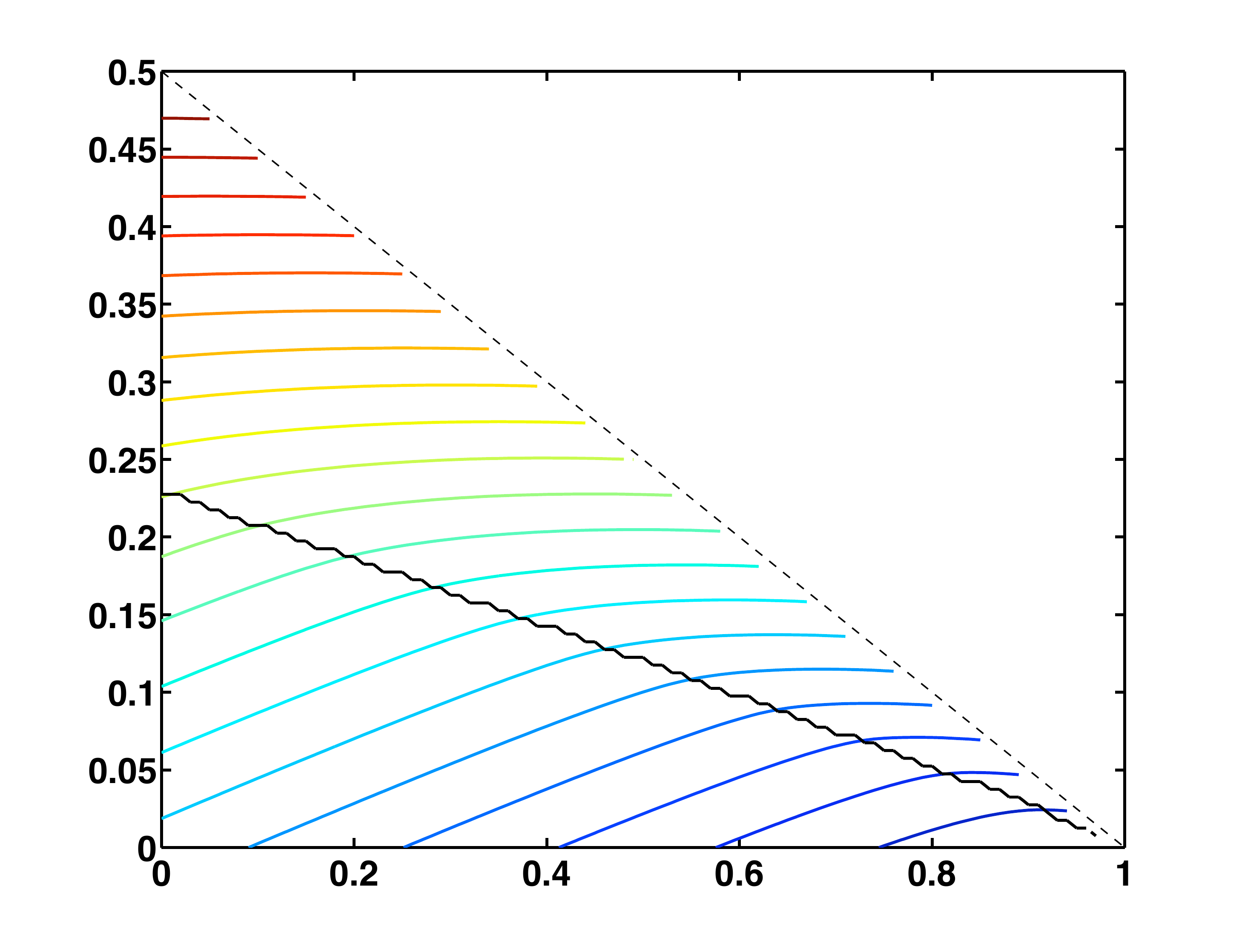}(b) \\
\includegraphics[width = 0.46\linewidth]{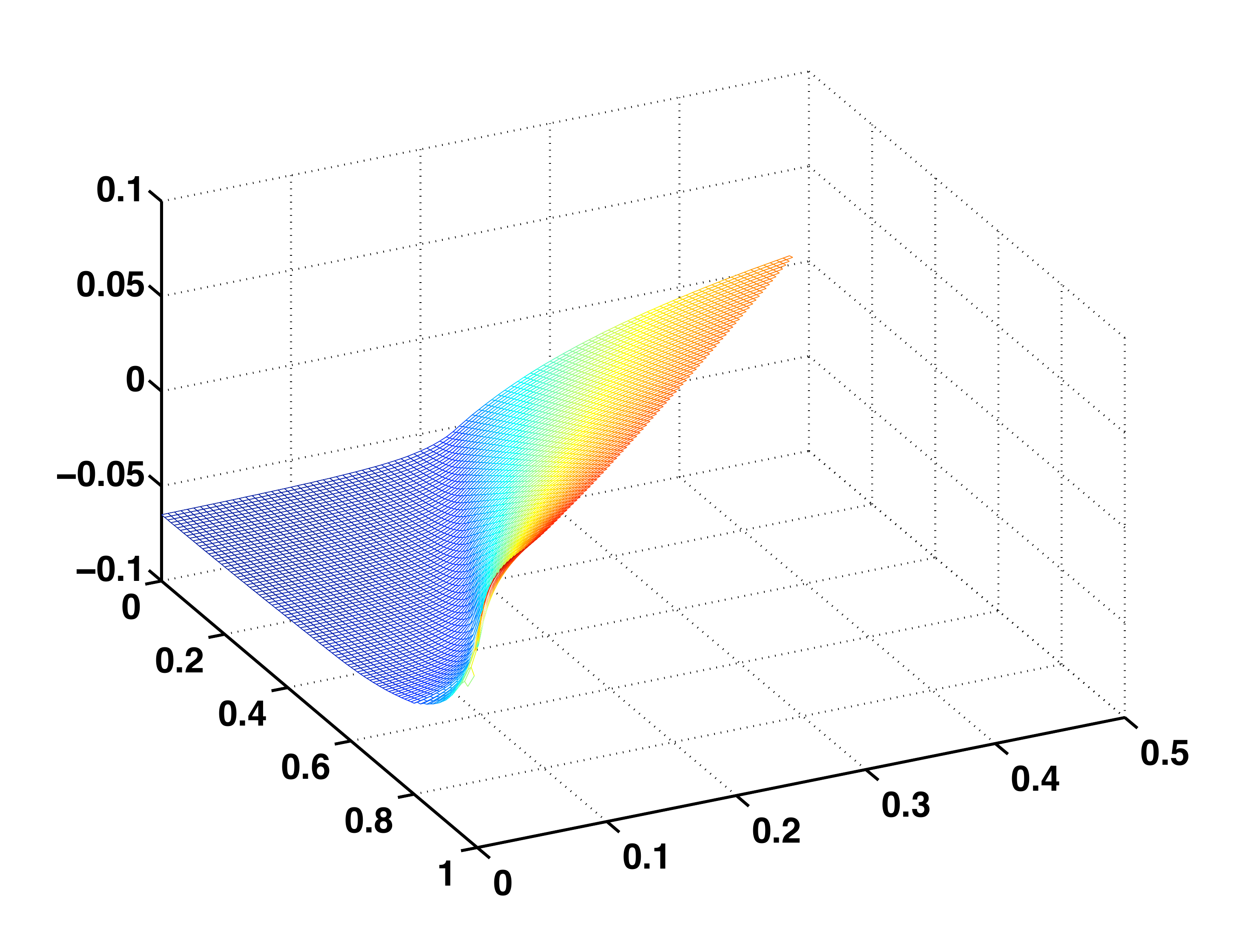}(c)\,\includegraphics[width = 0.46\linewidth]{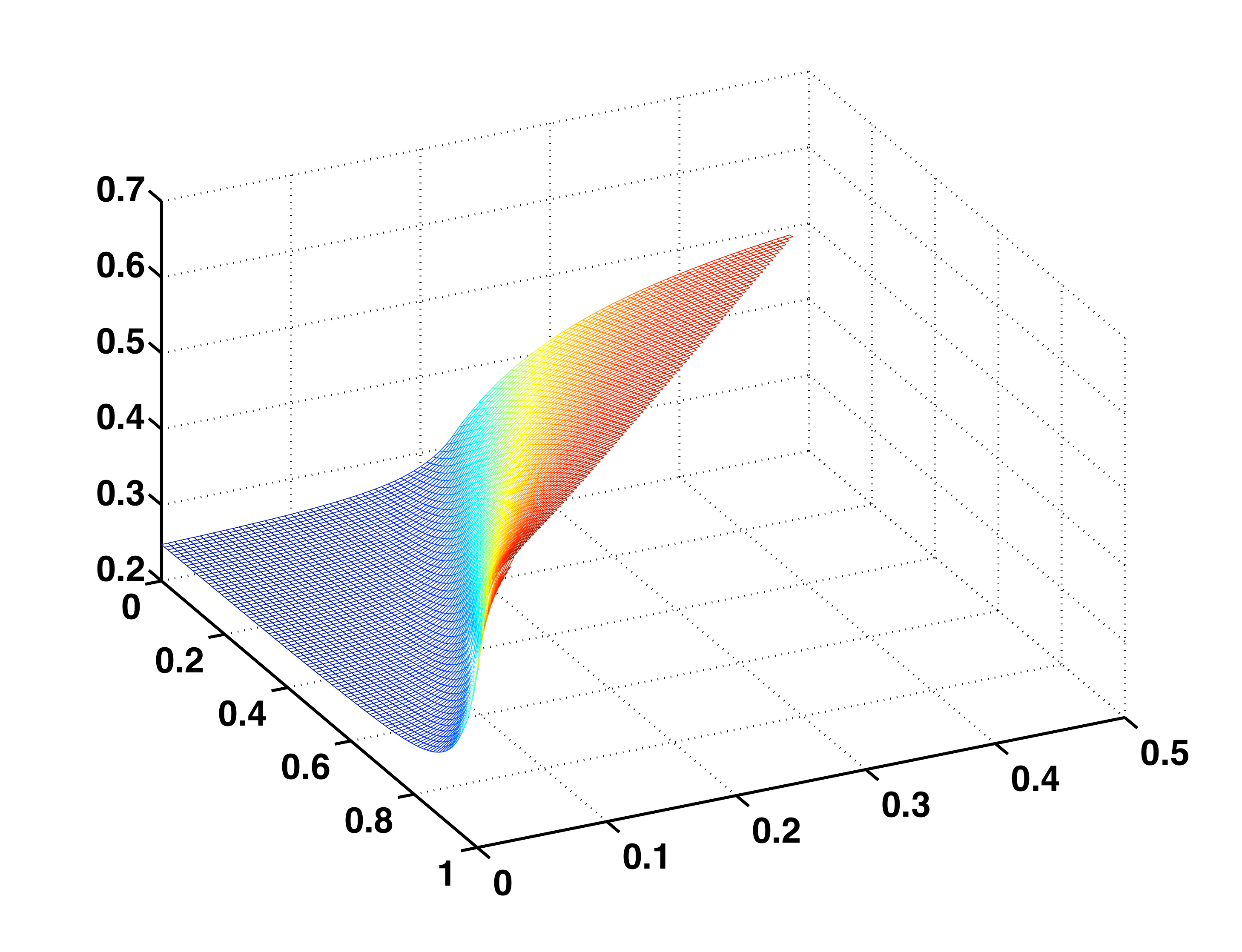}(d)
\caption{(a) The function $u(T,y)$ presumably converges to an eigenvector $\overline{u}$ as $T\to +\infty$, modulo a constant of order $\lambda_{HJ}T$.
(b) Level sets of the function $\overline{u}$ and the {\em separation line} where the optimal control switches from $a$ to $A$. (c) and (d) Derivatives of the function $\overline{u}$ with respect to the first and the second variable.} \label{fig:eigenvector-num}
\end{center}
\end{figure}

\begin{figure}
\begin{center}
\includegraphics[width = 0.46\linewidth]{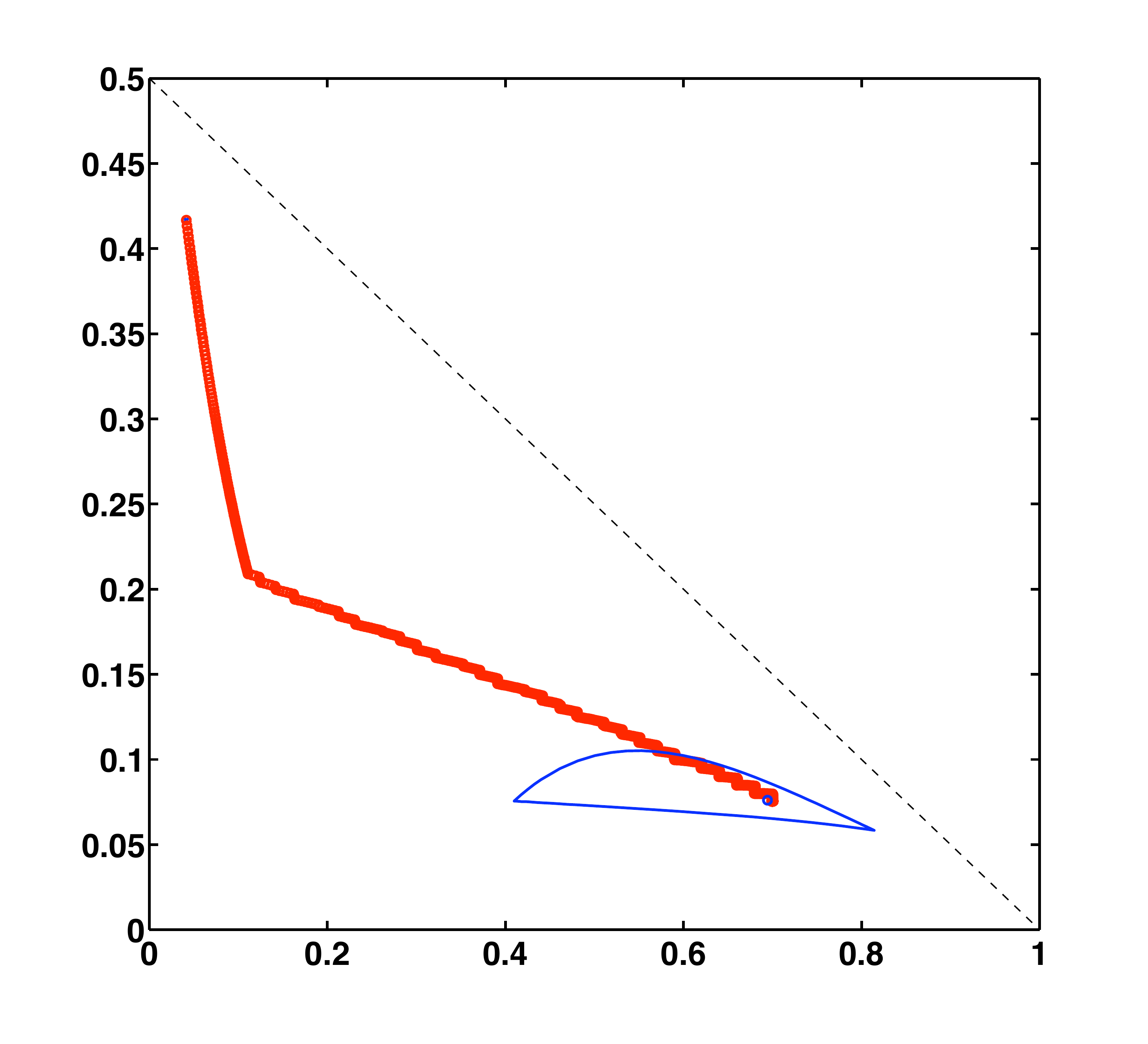}(a)\,\includegraphics[width = 0.46\linewidth]{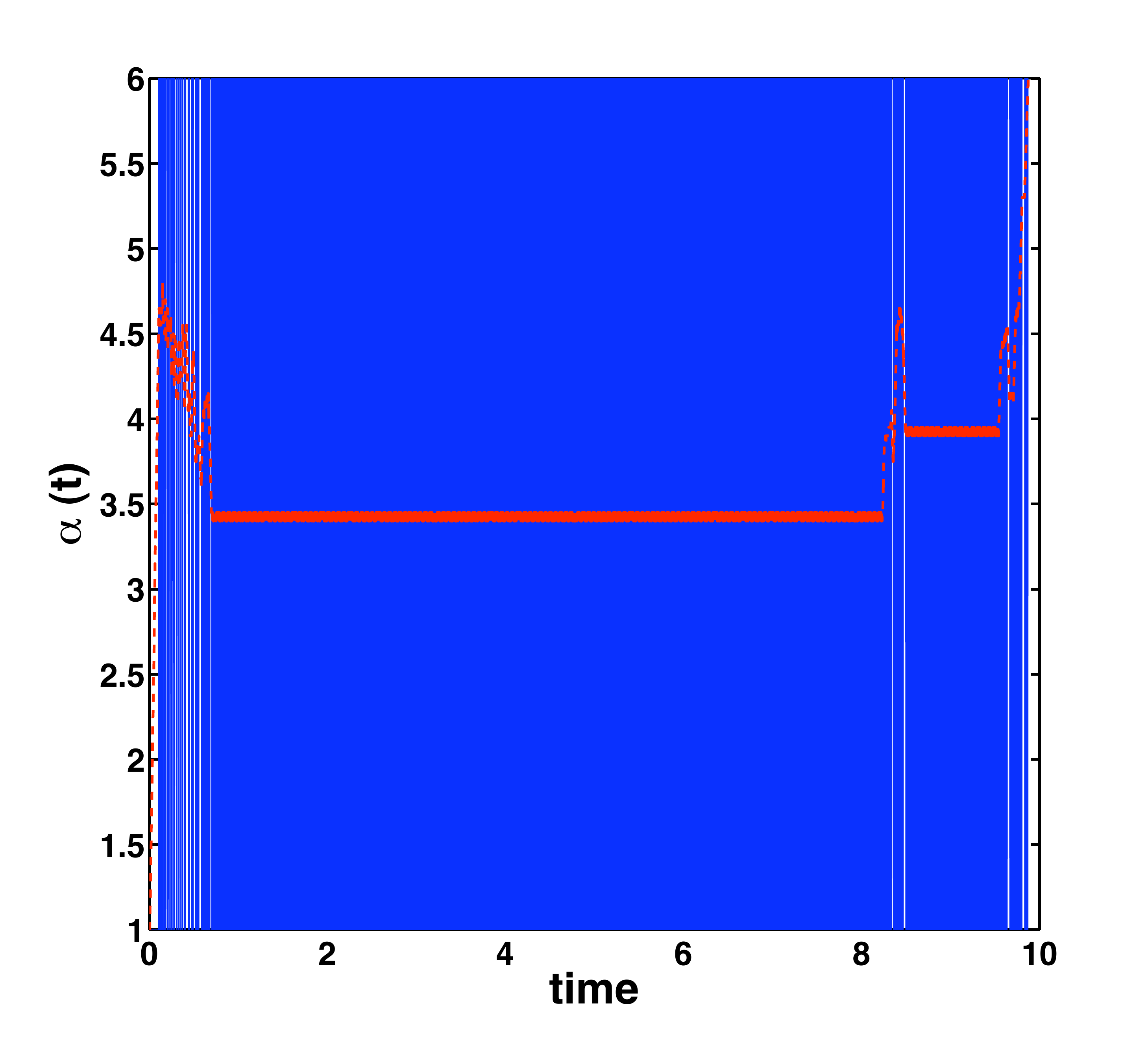}(b) 
\caption{(a) An optimal trajectory starting from the upper corner of the simplex. It presumably converges towards the optimal eigenvector $e_{\alpha^*}$. The optimal trajectory appears to be glued to the separation line.
(b) The corresponding optimal control $\alpha(t)$ is plotted in blue. It oscillates rapidly between the two extremal values $a = 1, A = 6$.
It converges weakly to $\alpha^* = 3.35$ as the time step $\Delta t$ goes to zero (here, $\Delta t = 10^{-3}$),
as can be observed when averaging the function $\alpha(t)$ (in red: local average of $\alpha(t)$ with a span of $100 \Delta t$). } \label{fig:optimal trajectory}
\end{center}
\end{figure}

\begin{table} 
\begin{center}
\begin{tabular}{|l|c|}
\hline
Growth in the first compartment & $\tau_1 = 0.5$  \\
\hline
Growth in the second compartment & $\tau_2 = 5$  \\
\hline
fragmentation in the second compartment & $\beta_2 = 1$  \\
\hline
fragmentation in the third compartment & $\beta_3 = 2$  \\
\hline
minimal rate of sonication & $a = 1$  \\
\hline
maximal rate of sonication & $A = 6$  \\
\hline
optimal rate of sonication (constant control) & $\alpha^* = 3.35$ \\
\hline
space step in the simplex & $\Delta y = 10^{-2}$  \\
\hline
time step & $\Delta t = 10^{-3}$   \\
\hline
final time of computation & $T = 10$  \\
\hline
\end{tabular}
\caption{Set of parameters for illustrations and numerical simulations.
}\label{tab:param}
\end{center}
\end{table}

All the numerical simulations shown in Section \ref{sec:HJB} have been performed for the running example \eqref{eq:example} with    parameters listed in Table \ref{tab:param}. 
We have solved numerically the Hamilton-Jacobi-Bellman equation \eqref{eq:HJB} using a classical upwind scheme for the discretization of the Hamiltonian \eqref{eq:Hamiltonian}. 
We observe that the quantity $\frac{u(T,y)}{T}$ converges to a constant value $\lambda_{HJ}$ which is close to the Perron eigenvalue $\lambda_P(\alpha^*)$. The discrepancy in the limit falls within the range of error due to the numerical scheme. Furthermore, we observe that the function $u(T,y) - \lambda_{HJ} T$ presumably converges to an eigenvector $\overline{u}$ (Figure~\ref{fig:eigenvector-num}).  

We notice that the numerical scheme selects either $a$ or $A$ at each step because of the very definition of the hamiltonian \eqref{eq:Hamiltonian}.  It is quite instructive to plot the line where the control switches from $a$ to $A$, namely where $\la Fy , D_y \overline{u}\ra = 0 $. We call it the {\em separation line} (Figure~\ref{fig:eigenvector-num}b). Apparently the optimal eigenvector $e_{\alpha^*}$ belongs to this line. We observe numerically that the optimal trajectories are glued to this separation line (Figure~\ref{fig:optimal trajectory}a). We observe fast oscillations between the extremal values $a$ and $A$ at the scale of the time step (see~Figure \ref{fig:optimal trajectory}b).  The values $a$ and $A$ are chosen in such a way that local averaging over several time steps yields a constant control $\alpha^*$. We conjecture that the control $\alpha(t)$ obtained through a bang-bang procedure (thus taking extremal values $a$ and $A$) converges weakly to the constant control $\alpha^*$ in infinite horizon. 
\medskip


We now list perspectives for future attention, by discussing the assumptions made in the present work. The first question concerns the higher-dimensional case. Here the two dimensional structure of the simplex plays a crucial role. For instance, it is not even clear how to define properly the ergodic set in higher dimension of space.  
A natural extension would be to study the case where the control $\alpha(t)$ has $N-1$ degree of freedom, where $N$ is the dimension of the simplex (the original  problem \eqref{eq:dynsyst} being of dimension $N+1$). Next we discuss ways to remove the hypotheses {\bf (H1-2-3-4-5)}, separately. 

\begin{figure}
\begin{center}
\includegraphics[width = 0.46\linewidth]{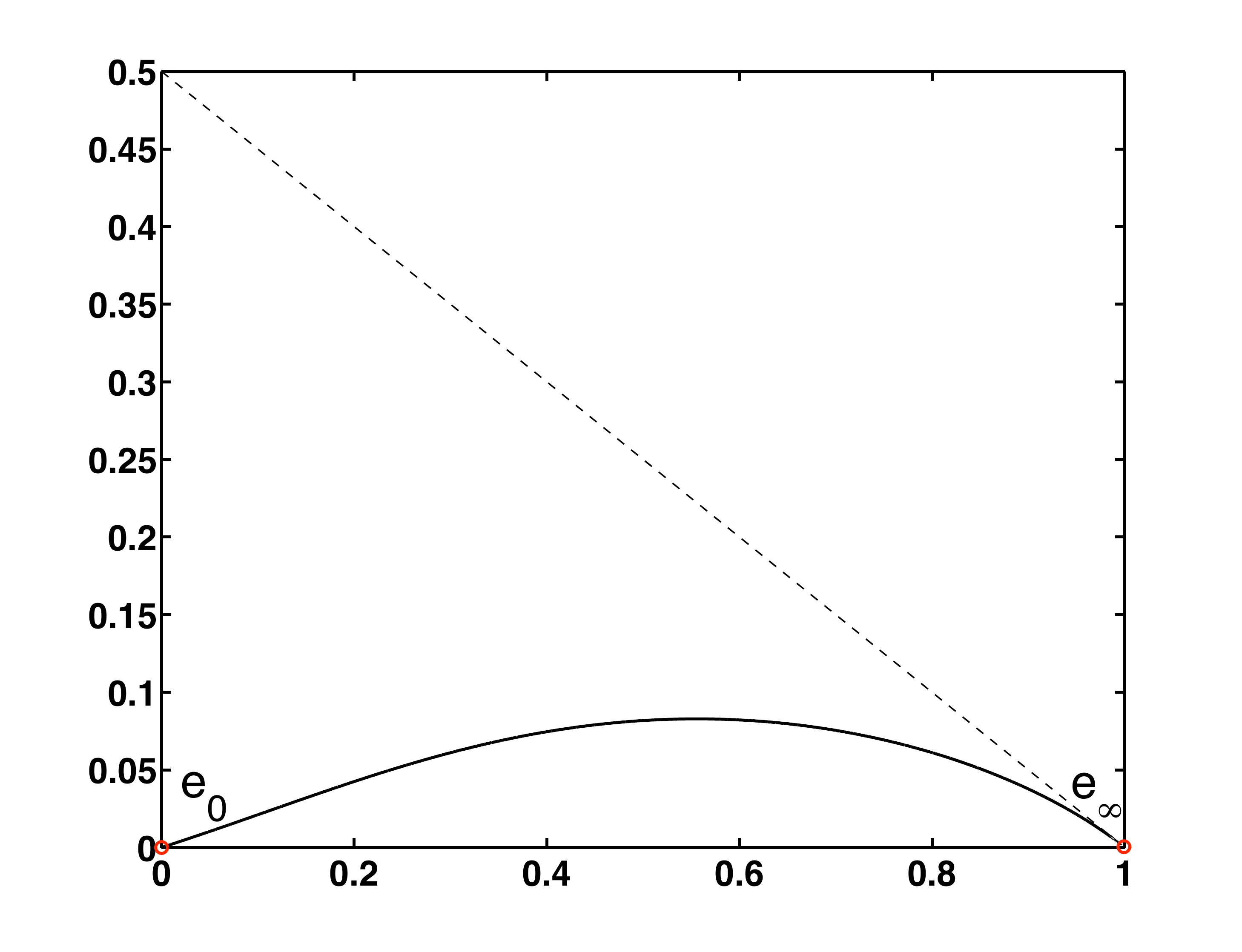}(a)\,
\includegraphics[width = 0.46\linewidth]{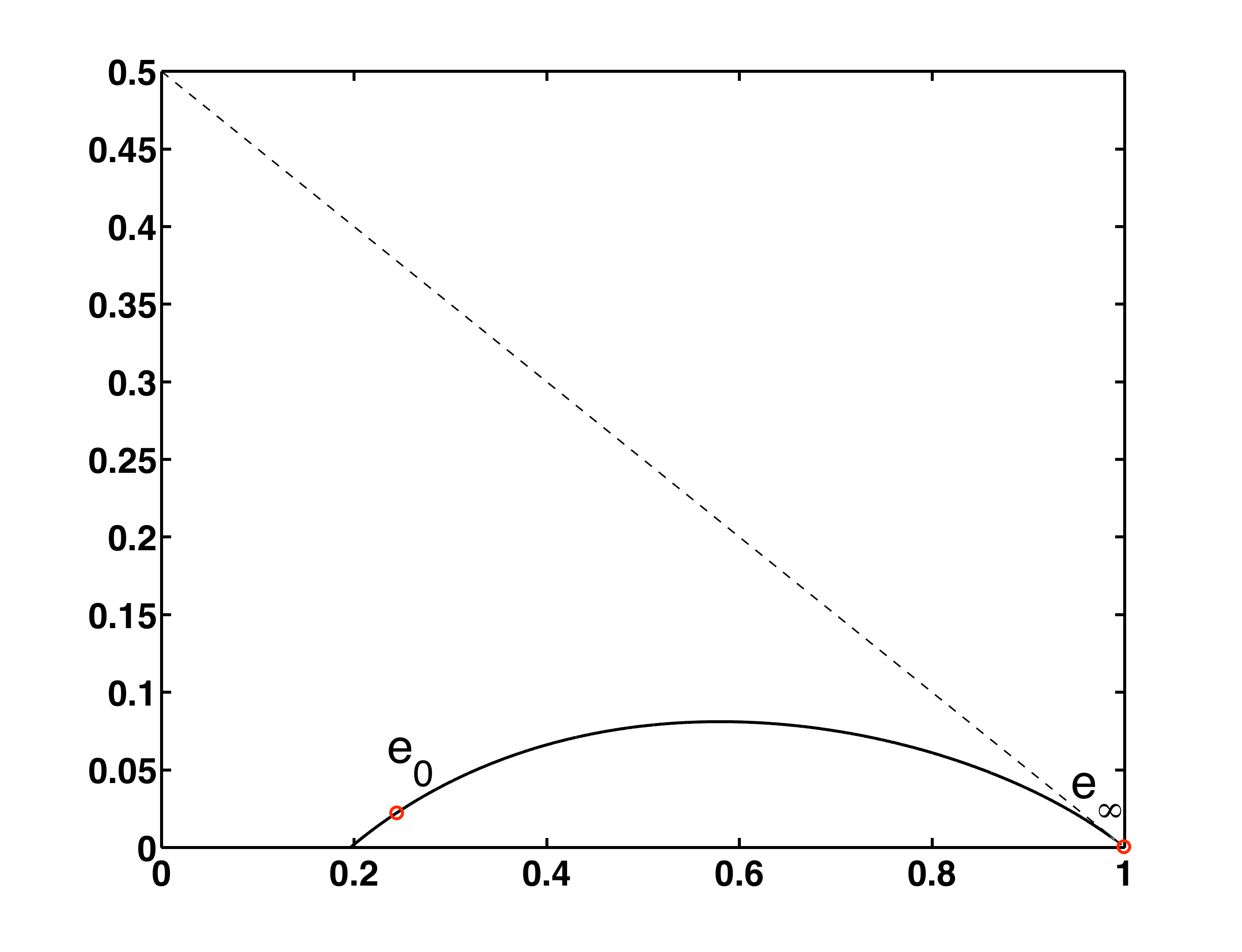}(b)
\\
\includegraphics[width = 0.46\linewidth]{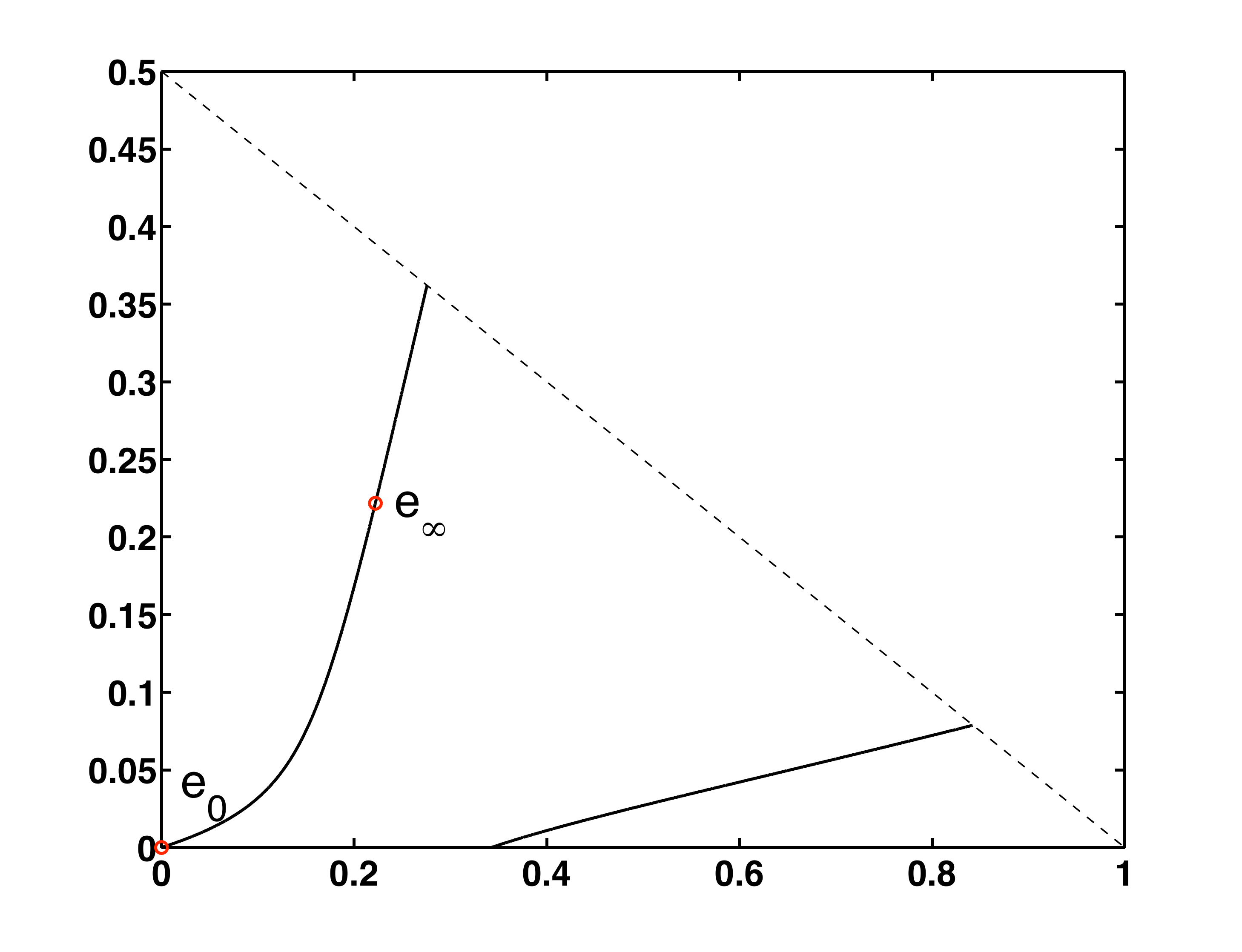}(c)
\,
\includegraphics[width = 0.46\linewidth]{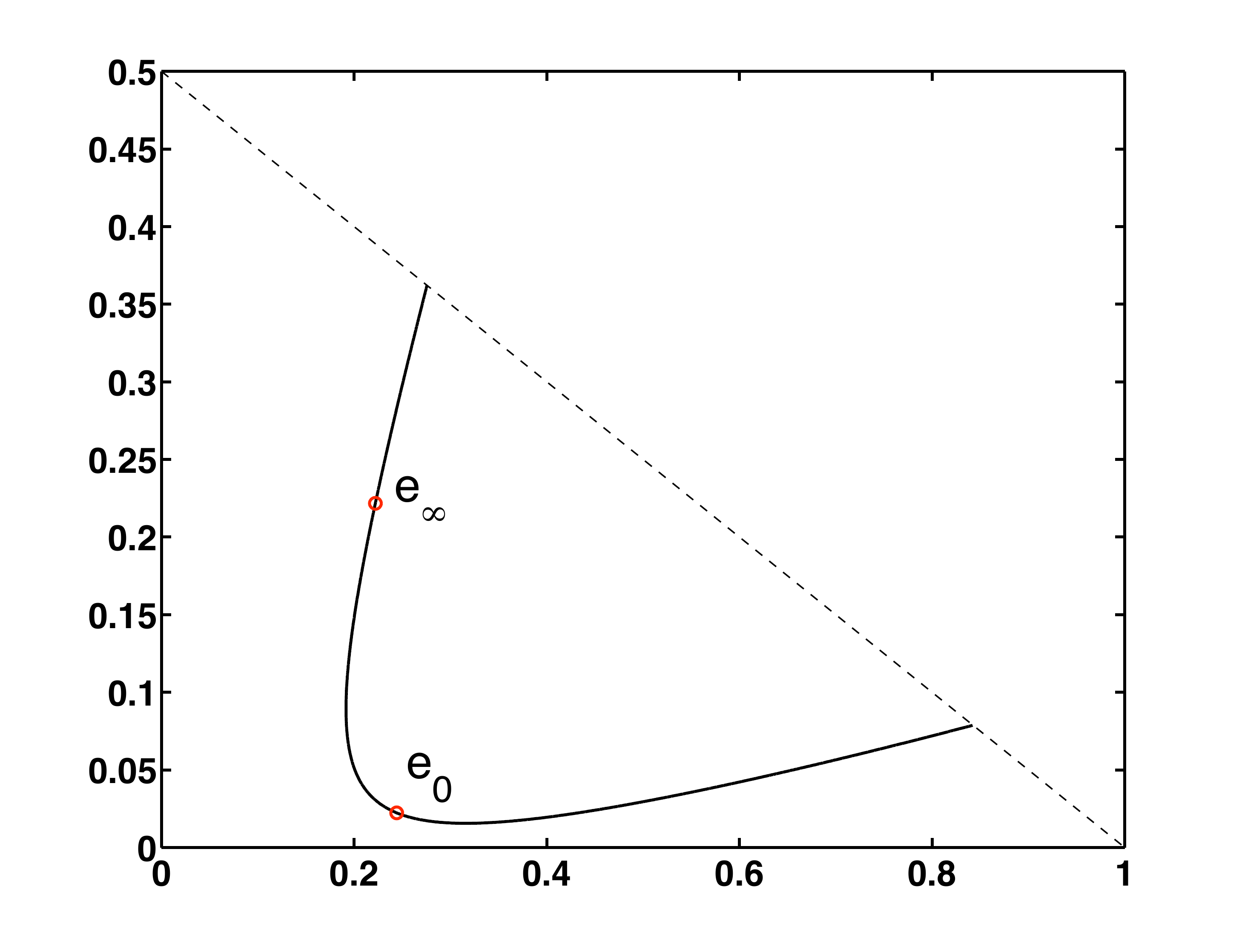}(d)
\caption{Discussion of Hypotheses {\bf (H1)-(H3)}. We have plotted the zero level set of the cubic function $\Phi_0 = \{ y\in \mS \;:\;\varphi(y) = 0\}$ together with the extremal eigenvectors $e_0$ and $e_\infty$ in four cases: (a) $G$ and $F$ are both reducible ({\em i.e.} the running example \eqref{eq:example}), (b) $G$ is irreducible and $F$ is reducible, (c) $G$ is reducible and $F$ is irreducible, (d) $G$ and $F$ are both irreducible. We have modified $F$ such that we still have $m^T F = 0$. }
\label{fig:phi0}
\end{center}
\end{figure}



\begin{figure}
\begin{center}
\includegraphics[width = 0.46\linewidth]{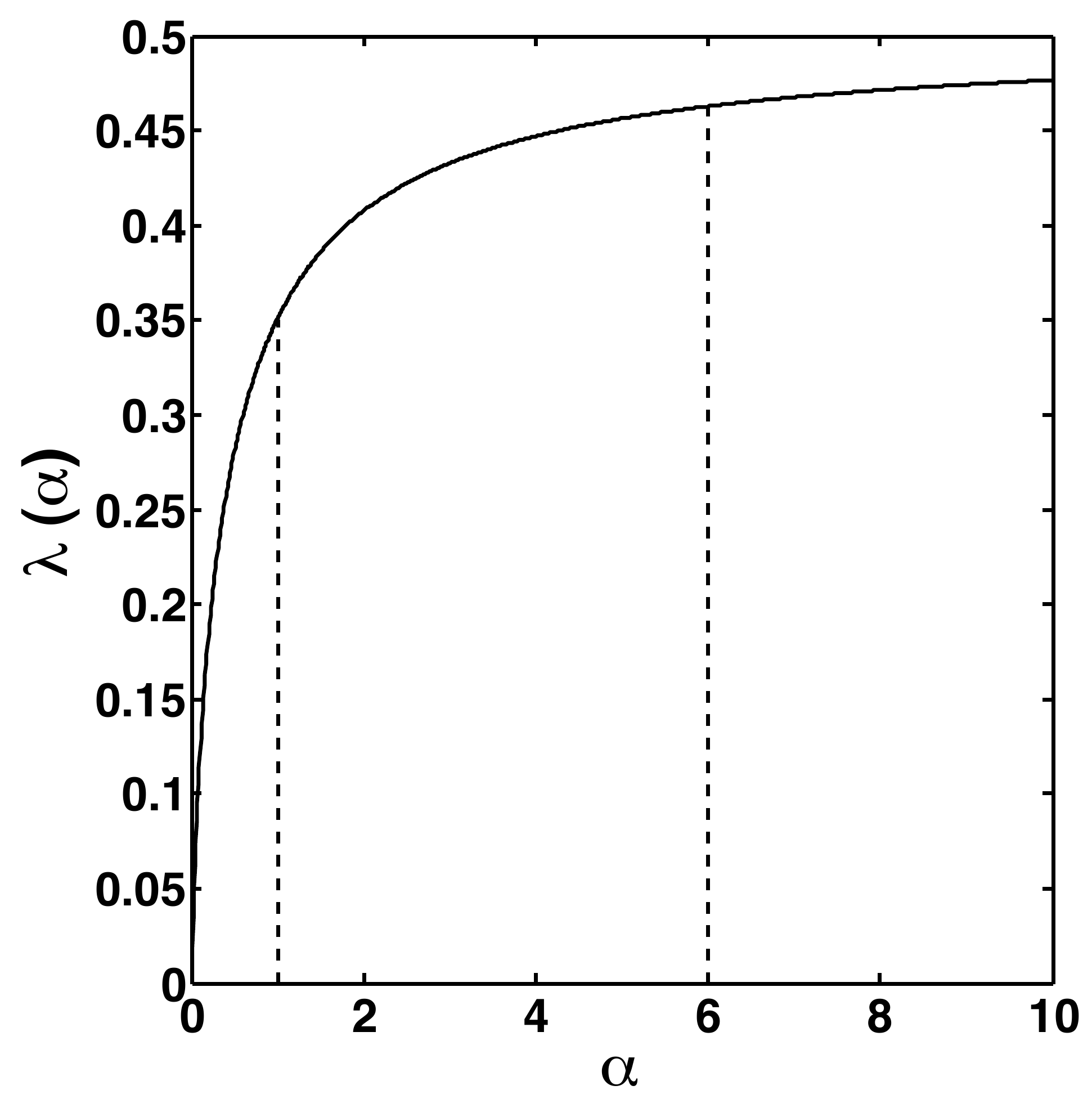}(a)\,
\includegraphics[width = 0.46\linewidth]{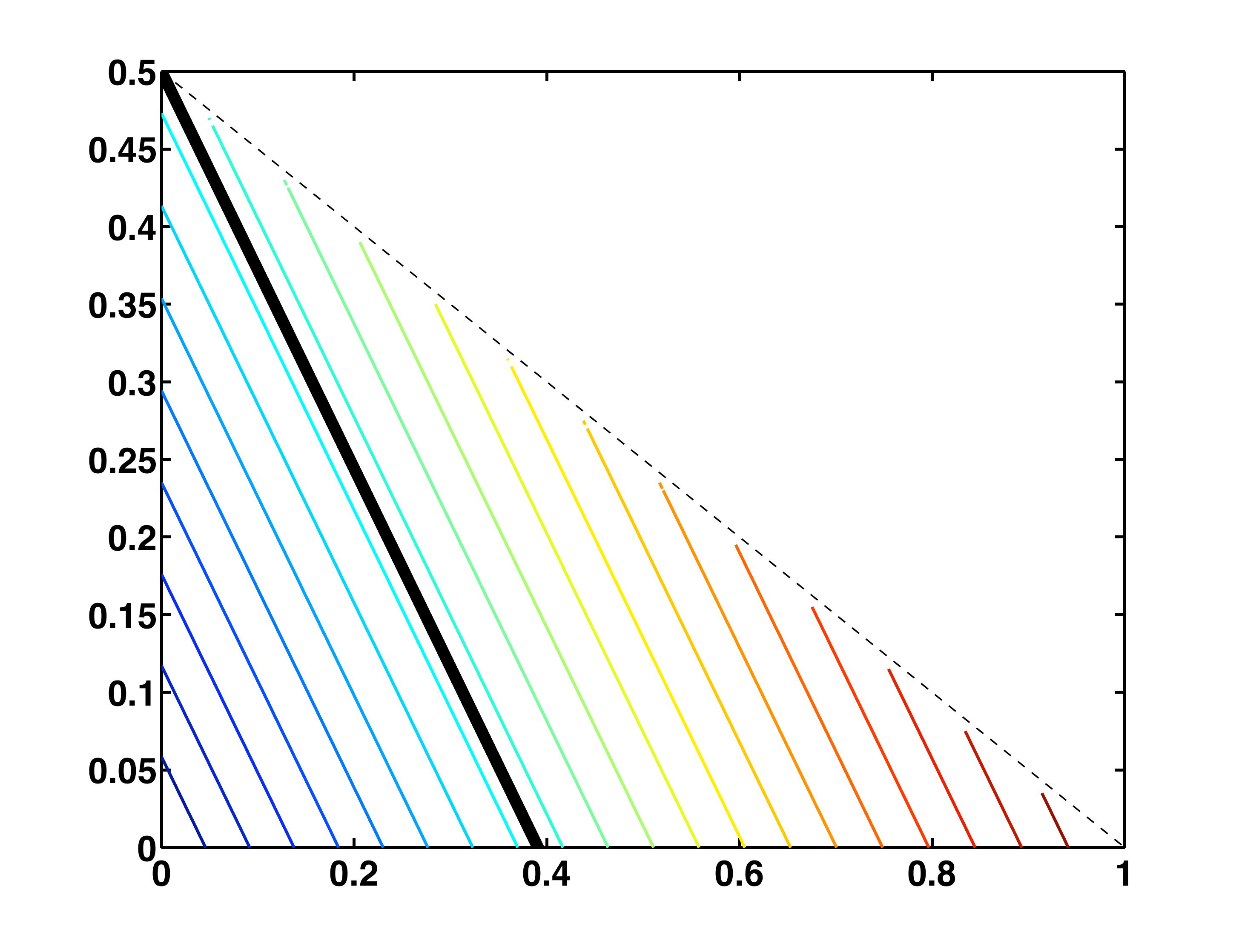}(b) \\
\includegraphics[width = 0.46\linewidth]{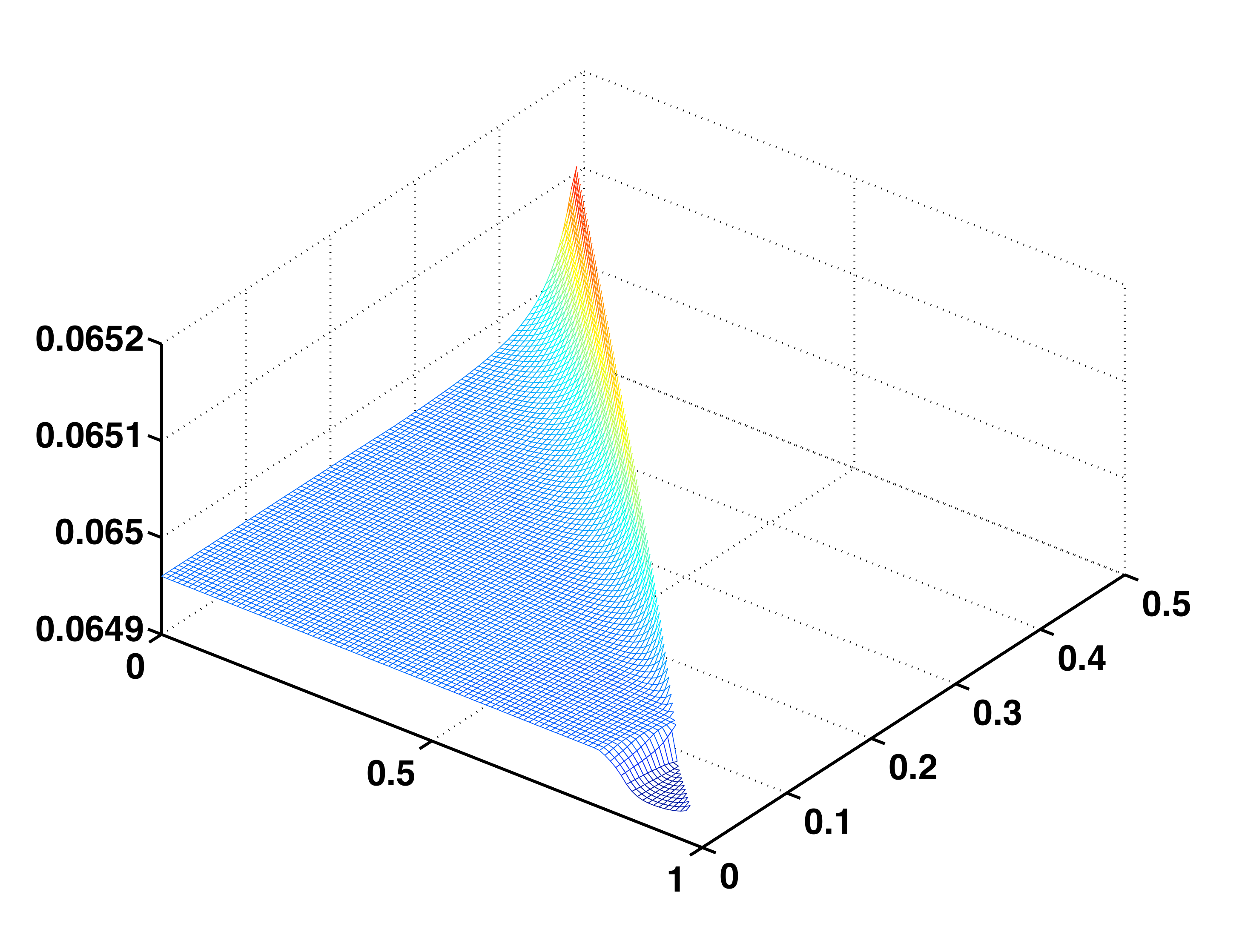}(c)\,\includegraphics[width = 0.46\linewidth]{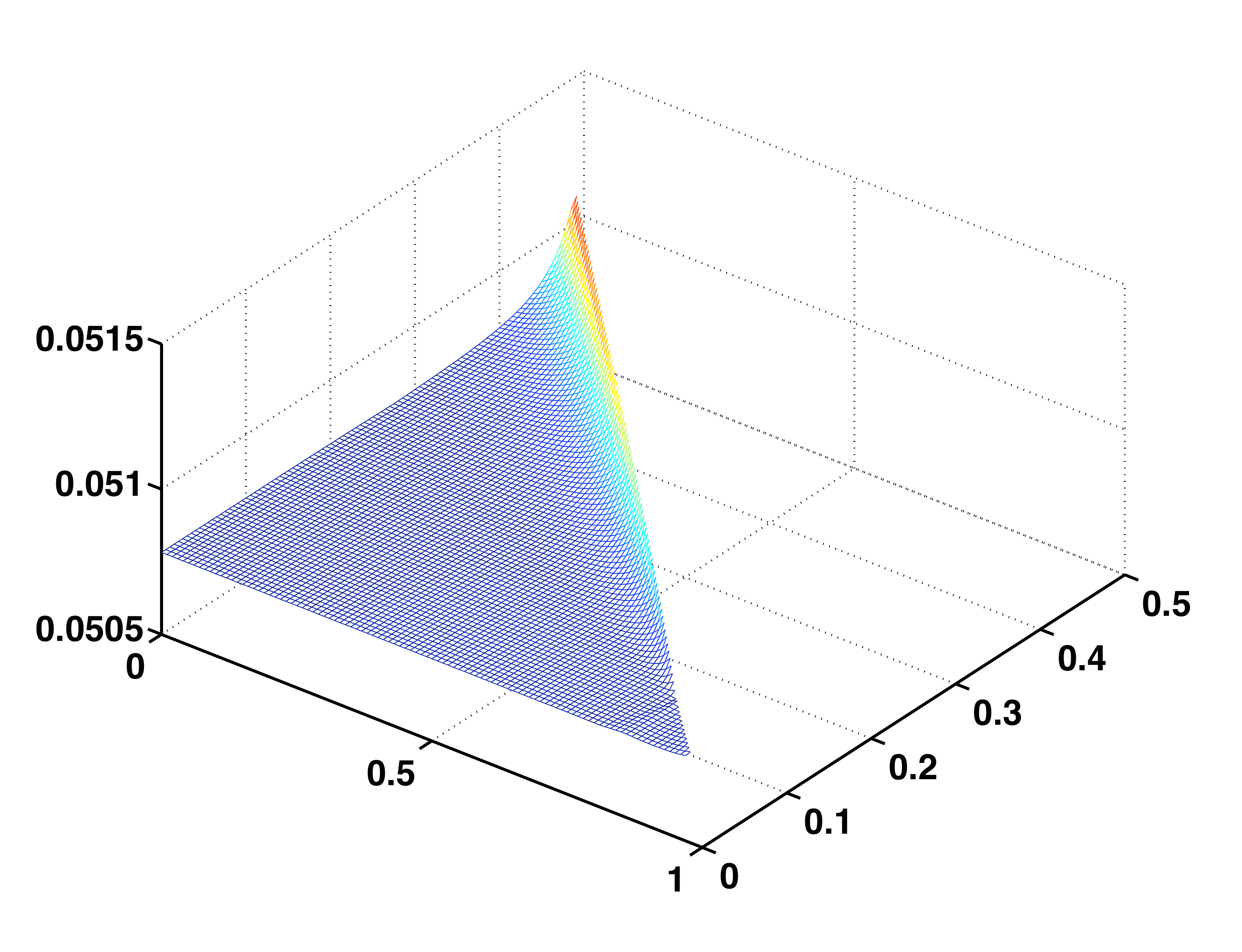}(d)
\caption{Discussion of Hypothesis {\bf (H2)}. Here, $\tau_1 = \tau_2 = 0.5$, $a = 1$ and $A = 6$. The function $\lambda(\alpha)$ is increasing (a). The optimal control is $\alpha(t) = A$. Consequently $\lambda^* = \lambda_P(A)$. The eigenvector is also known: $\overline u(y) = \log \la \phi_A,y\ra$. This claim is confirmed by numerical simulations: the level sets of $u(T,y)$ are straight lines (b). The slope of the lines is in accordance with $\phi_A$ (the thick line is directed by $m\wedge\phi_A\in T\mS$). The gradient of the function $\exp(u(T,y) - u(T,y_0))$ is almost constant: observe the small amplitude in (c) and (d).}
\label{fig:H2cas1}
\end{center}
\end{figure}



\begin{figure}
\begin{center}
\includegraphics[width = 0.6\linewidth]{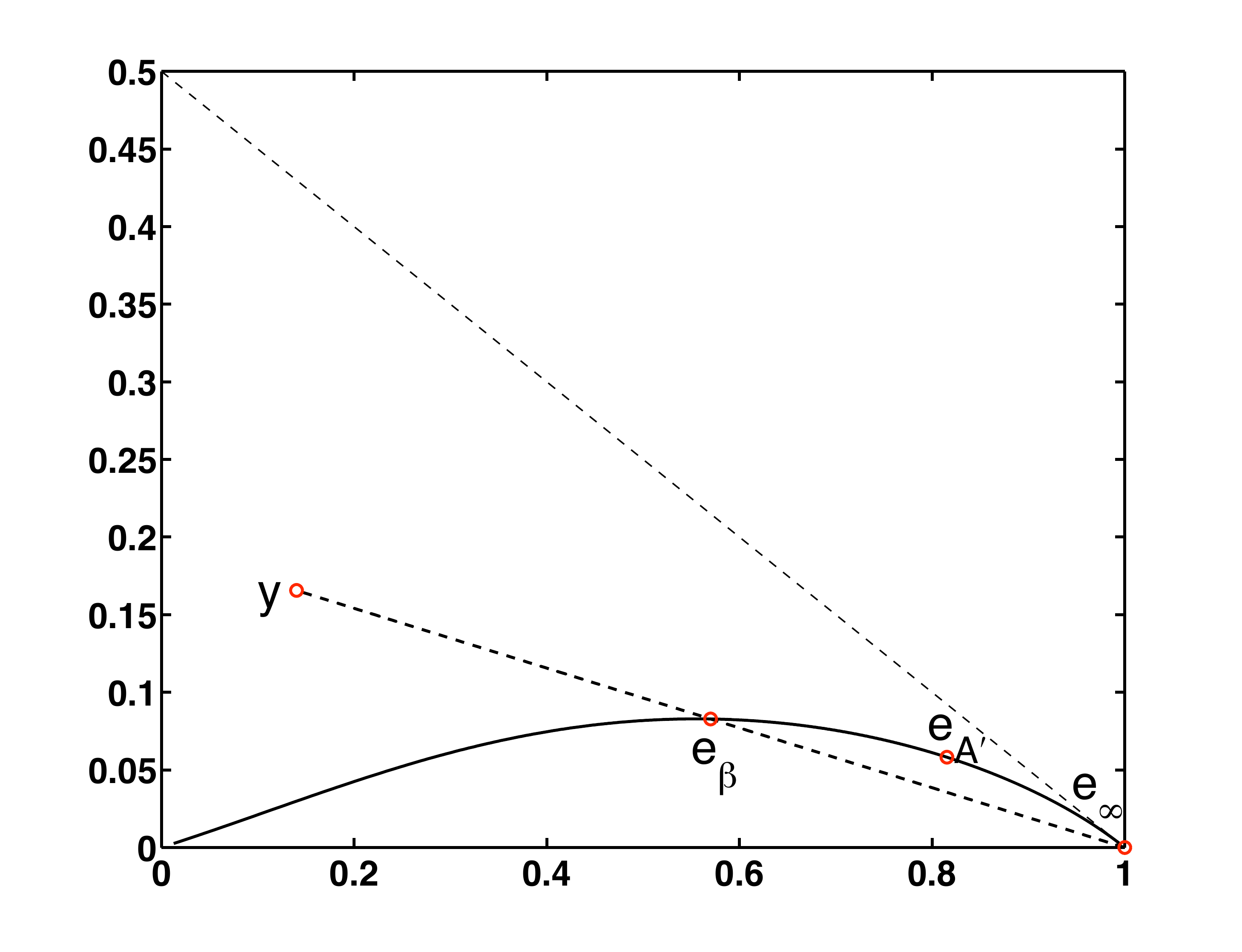}
\caption{Construction of the subset $\mS'$: the segment $[y,e_\infty]$ crosses the line $\Phi_0$ on $e_\beta$ with $\beta\leq A'$}
\label{fig:construction S'}
\end{center}
\end{figure}


\begin{figure}
\begin{center}
\includegraphics[width = 0.46\linewidth]{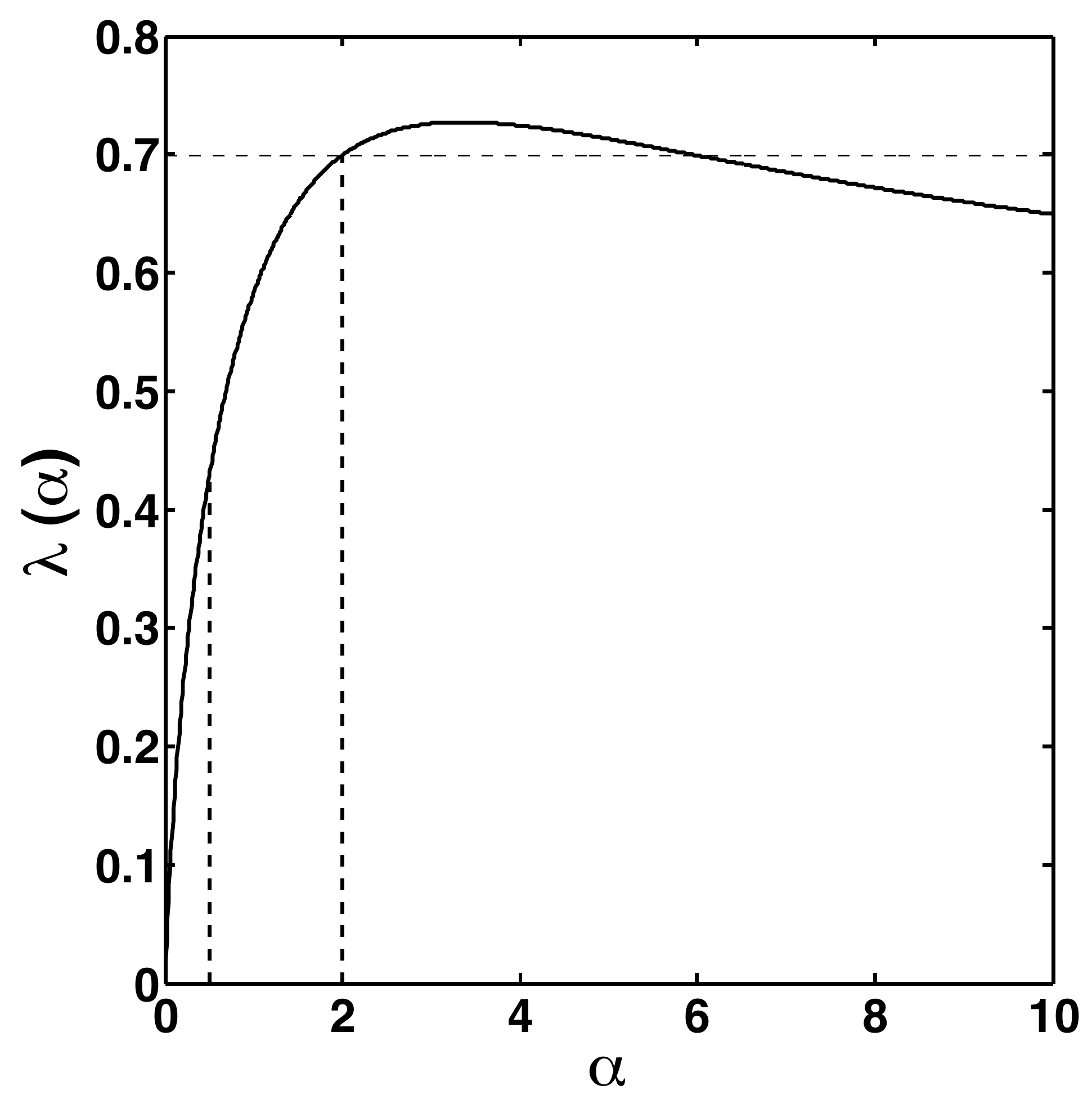}(a)\,
\includegraphics[width = 0.46\linewidth]{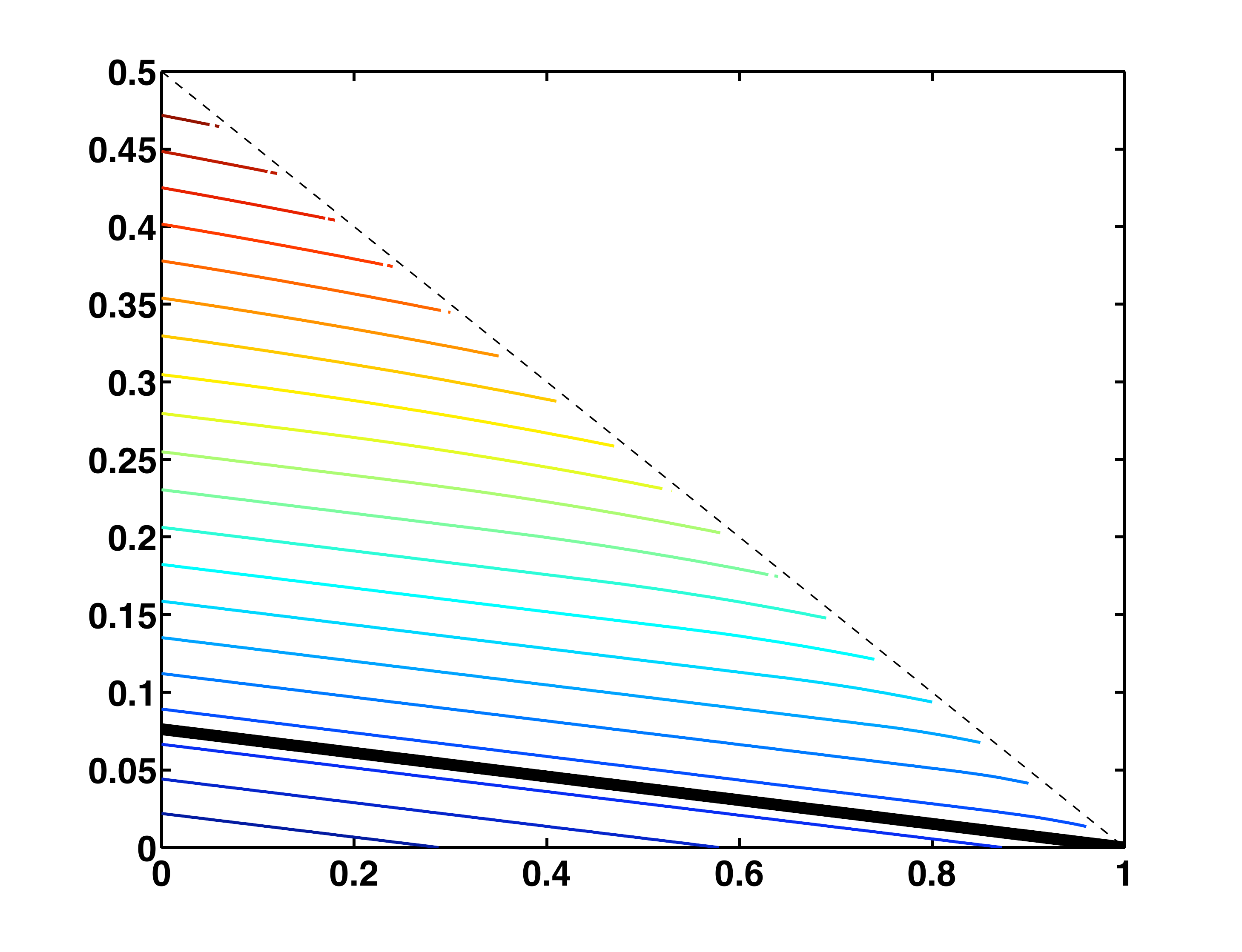}(b) \\
\includegraphics[width = 0.46\linewidth]{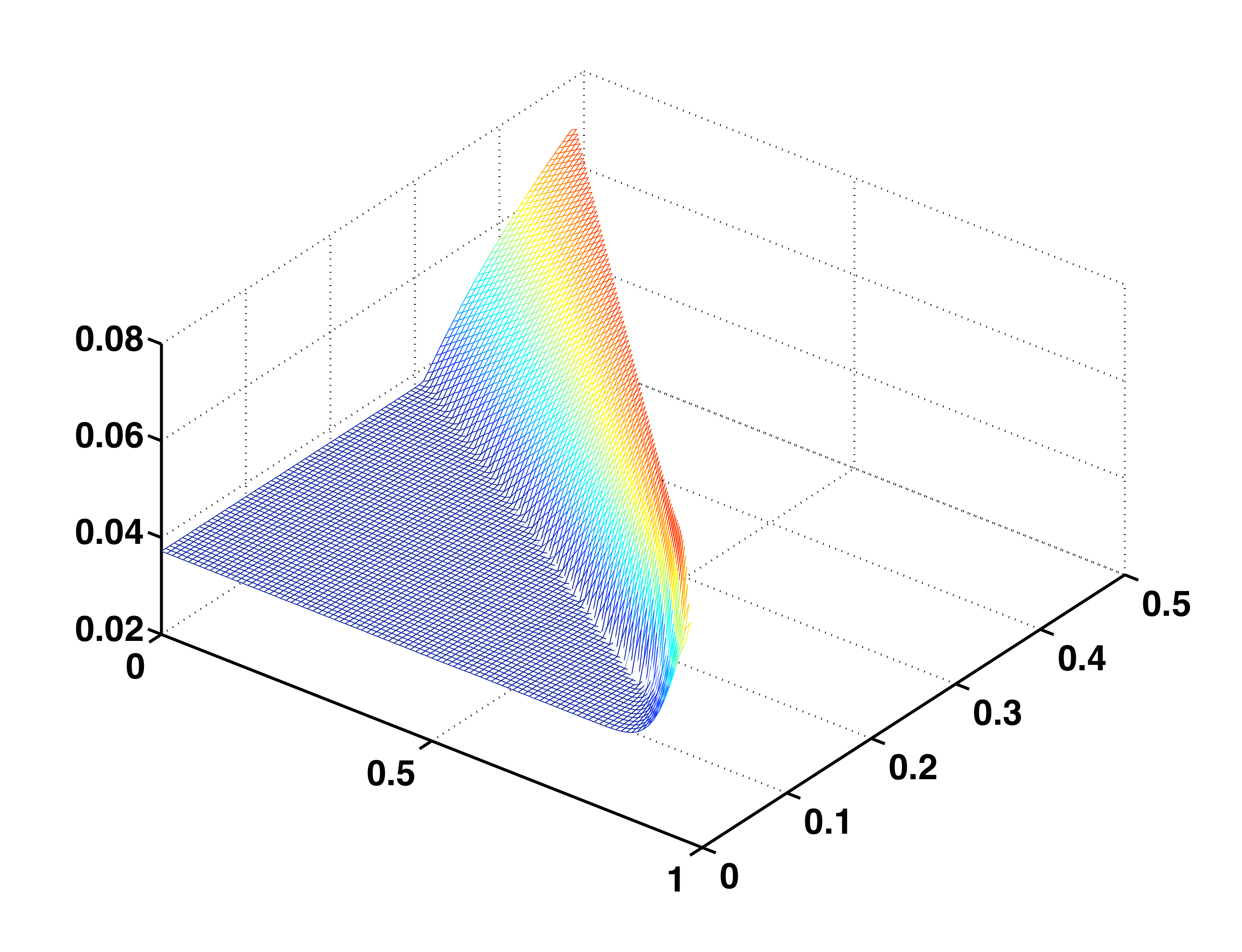}(c)\,
\includegraphics[width = 0.46\linewidth]{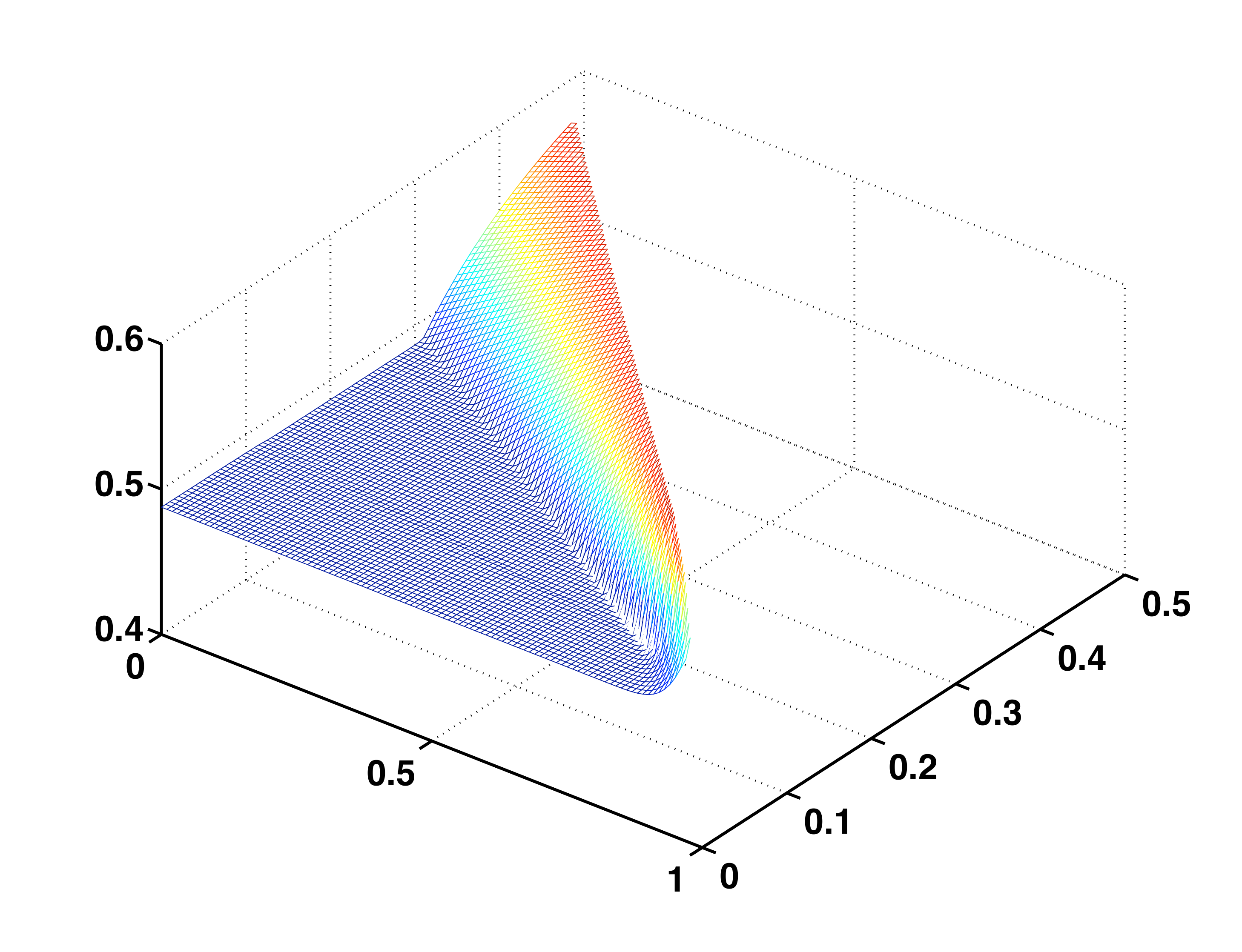}(d)
\caption{Discussion of Hypothesis {\bf (H2)}. Here, $\tau_1 = 0.5$, $ \tau_2 = 5$, $a = 0.5$ and $A = 2$. The function $\lambda(\alpha)$ is increasing up to $\alpha^* = 3.35$, then decreasing (a). The eigenvector is  known on a subpart $\mS'\subset \mS$, where $\mS'$ is defined in the text: $\forall y \in \mS'\; \overline u(y) = \log \la \phi_A,y\ra$. This claim is confirmed by numerical simulations: the level sets of $u(T,y)$ are straight lines on $\mS'$ (b). The slope of the lines is in accordance with $\phi_A$ (the thick line is directed by $m\wedge\phi_A\in T\mS$). The gradient of the function $\exp(u(T,y) - u(T,y_0))$ is plotted in (c) and (d).}
\label{fig:H2cas3}
\end{center}
\end{figure}

\begin{description}
\item[(H1)-(H3)] In the case where $G$ or $F$ is irreducible, the limiting eigenvectors do not lie  on the boundary of the simplex. Therefore the line of Perron eigenvectors $\{ e_\beta: \; 0\leq \beta \leq \infty\}$ does not divide the simplex into two parts. However the line $\Phi_0 = \{ y\in \mS \;:\;\varphi(y) = 0\}$ does so. It is the zero level set of a cubic function. We have plotted in Figure \ref{fig:phi0} several examples, for the cases where $G$ or $F$, or both, are irreducible.
\item[(H2)] When the Perron eigenvalue $\lambda_P(\alpha)$ is nondecreasing with respect to $\alpha$ and bounded (cf. Proposition~\ref{prop:n=3} for the running example), we have $\max_{\alpha\in [a,A]}\lambda_P(\alpha) = \lambda_P(A)$. Our strategy of proof does not contain this particular example (see Step 6).  Nevertheless, under an additional assumption (see below) we exhibit a particular solution to \eqref{eq:HJB} with a different initial condition, namely 
\begin{equation} \label{eq:sol part}
u(t,y) = \lambda_P(A) t + \log \la \phi_A,y\ra\, . 
\end{equation}
Interestingly this coincides with the expansion of  $r_\alpha(T,x)$ in the case of constant control $\alpha = A$ \eqref{eq:perron taylor}. 
Moreover the function $\overline u(y) = \log \la \phi_A,y\ra $ is a good candidate for being the Hamilton-Jacobi eigenvector. It is a viscosity solution of the ergodic stationary equation,
\[ - \lambda_P(A) + H(y,D_y \overline{u}(y)) = 0\, . \]
In Figure \ref{fig:H2cas1} we show   numerical simulations which confirm that this particular solution describes well  the asymptotic behaviour of solutions to \eqref{eq:HJB}. \\
To assert that $u(t,y)$ given by \eqref{eq:sol part} is a particular solution of \eqref{eq:HJB} we check that the optimality condition is verified for $\alpha = A$,
\begin{align*} 
\max_{\alpha\in [a,A]} \la b(y,\alpha), D_y u\ra 
&= \max_{\alpha\in [a,A]} \left\{ \dfrac{\la (G + \alpha F) y, \phi_A\ra }{\la y,\phi_A\ra} - \la m , G y\ra \right\} \\
& = \max_{\alpha\in [a,A]} \left\{ \lambda_P(A) + (\alpha - A) \frac{\la Fy , \phi_A\ra }{\la y , \phi_A\ra }\right\} \\
& = \lambda_P(A) \, .
\end{align*}
The optimality condition is satisfied if   $\la Fy,\phi_A\ra \geq 0 $ for all $y\in \mS$.
This condition is guaranteed under the two following conditions: (i) $\alpha\mapsto \lambda_P(\alpha)$ is nondecreasing,
and (ii) for all $y\in \mS$ the segment $[y,e_\infty]$ crosses the line of eigenvectors $\Phi_0$, where $e_\infty$ is a corner of the simplex (Figure~\ref{fig:construction S'}).
The second condition is satisfied for the running example since in this case $e_0=(0\ 0\ 1/3)^T$ and $e_\infty=(1\ 0\ 0)^T$ are corners of the simplex $\mS$
and the tangent to $\Phi_0$ at the point $e_\infty$, given by the vector $Ge_\infty-\la m,Ge_\infty\ra e_\infty=\tau_1(-2\ 1\ 0)^T,$ coincides with the edge of the simplex which does not contain $e_0$.
We introduce the intersection point $e_\beta\in \Phi_0$. We decompose $y = c_\beta e_\beta + c_\infty e_\infty$, with $c_\beta \geq 0$ (and  $c_\infty\in \RR$). We have on the one hand,
\begin{equation*} 
\lambda_P(A) \la \phi_A , e_\beta \ra  = \phi_A (G + AF) e_\beta  = \lambda_P(\beta) \la \phi_A , e_\beta \ra + (A - \beta ) \la \phi_A, F e_\beta \ra\, .  \end{equation*}
On the other hand we get since $Fe_\infty = 0$,
\begin{align*} 
\la Fy , \phi_A \ra & = c_\beta \la F e_\beta, \phi_A \ra + c_\infty \la F e_\infty, \phi_A \ra  \\
& = c_\beta  \dfrac{\lambda_P(A) - \lambda_P(\beta)}{A - \beta} \la \phi_A , e_\beta \ra  \geq 0 \, .
\end{align*}
The same result holds true in the case where $\lambda_P(\alpha) $ is increasing up to $A$, and satisfies $\lambda_P(\beta) \geq \lambda_P(A)$ for all $\beta >A$. \\
Last but not least, in the case where $\lambda_P(\alpha) $ is increasing up to $A$, and there exists $A'>A$ such that  $\lambda_P(A') = \lambda_P(A)$ the situation is quite different. We choose  the smallest possible $A'>A$. The function defined by \eqref{eq:sol part} is a particular solution only on the subset of the simplex $\mS'$ defined by the following rule: $y\in \mS'$ if the segment $[y,e_\infty]$ crosses the line $\Phi_0$ on $e_\beta$ with $\beta\leq A'$ (Figures \ref{fig:construction S'} and \ref{fig:H2cas3}). 
\item[(H4)] When the quantity \eqref{as:tech} does not have a constant sign, we cannot rule out the situation where the trajectories cross the line of eigenvectors $\Phi_0$ anywhere. This causes problems on the proper definition of the ergodic set $\mZ_0$ (see also discussion of how to remove {\bf (H5)} below). This also causes problems  on the monotonicity formulas~\eqref{eq:betaA}-\eqref{eq:betaa}. However we believe this is just a technical assumption. When the quantity $\la \frac{d e_\beta}{d\beta}  , \frac{\Theta F e_\beta}{|\Theta F e_\beta|} \ra$ changes sign then we may switch the parametrization from $A+\delta$ to $a-\delta$, or {\it vice-versa}, to preserve the monotonicity of $\beta(t)$. 
\item[(H5)] This Hypothesis is essential to build properly the ergodic set $\mZ_0$ in our way. In fact it guarantees that the trajectories which define the boundary $\partial\mZ_0$ lie on the two opposite sides of $\Phi_0$. This rules out a possible spiraling phenomenon. It could be possible to define the ergodic set $\mZ_0$ in another way but this would lead to increasing complexity of the preceding steps. \\
We propose the following construction: starting from $e_a$, the trajectory $\gamma_A^a$ crosses $\Phi_0$ on $e_{\beta_1}$ with $\beta_1 \geq A$ due to Assumption {\bf (H4)}. If $\beta_1>A$ we need to switch the control after crossing $\Phi_0$ to preserve the monotonicity formula (Lemma \ref{lem:dotbeta}). The trajectory starting from $e_{\beta_1}$ with constant control $a$, $\gamma_a^{\beta_1}$ reaches $\Phi_0$ on $e_{\beta_2}$ with $\beta_2\leq a$ due to Assumption {\bf (H4)}. If $\beta_2 = a$ we are done: the two portions of trajectories $\gamma_A^a$ and $\gamma_a^{\beta_1}$ enclose a stable set $\mZ_0$ which is a good candidate to being the ergodic set. \\
If $\beta_2 < a$ however we switch again: the trajectory starting from $e_{\beta_2}$ with constant control $A$, $\gamma_A^{\beta_2}$ reaches $\Phi_0$ on $e_{\beta_3}$ with $\beta_3 > \beta_1$ (because trajectories cannot cross). We define iteratively two monotonic sequences $\beta_{2i}$ and $\beta_{2i+1}$ (resp. decreasing and increasing) such that $\beta_{2i}< a$ and $\beta_{2(i+1)}>A$. The two limits $\underline{\beta} < \beta_{2i} < a < A < \beta_{2(i+1)} < \overline{\beta} $ define the boundary of a stable and controllable set $\mZ_0$ through the following periodic cycle: starting from $e_{\underline{\beta}}$ the trajectory with constant control $A$ reaches $e_{\overline{\beta}}$    in finite time, and the   trajectory starting from $e_{\overline{\beta}}$ with constant control $a$ reaches $e_{\underline{\beta}}$   in finite time.  \\
We cannot rule out the existence of several such periodic cycles (with control $a$ on $\Phi^+$ and control $A$ on $\Phi^-$). In this case the set $\mZ_0$ constructed above is certainly not attractive, because $\mZ_0$ does not contain any other cycle than its boundary. 
\end{description}

\bigskip

\noindent{\em Acknowledgements.} The authors have benefited from stimulating discussions with St\'ephane Gaubert, Thomas Lepoutre and Maxime Zavidovique. They warmly thank Marie Doumic for having coordinated the ANR Grant ANR-09-BLAN-0218 TOPPAZ, in which both authors were involved. The early motivation for this work comes from illuminating discussions with Natacha Lennuzza and Franck Mouthon, from whom the authors have learned everything they know about Prion proliferation and PMCA.

\bigskip


\section*{Appendix}

\appendix

\section{Diagonalisation of $G+\al F$}\label{ap:diag}

In the case of the running example, we give a condition for which the matrix $G+\alpha F$ is diagonalisable (in $\RR$).

\begin{proposition}
Consider the matrices $G$ and $F$ defined in~\eqref{eq:example}, with the condition $\tau_2>2\tau_1.$
Then the matrix $G+\al F$ is diagonalisable in $\R$ for any $\al>0.$
Furthermore, the three real eigenvalues satisfy
$$\lambda_1^*>0>\lb_2>\lb_3.$$
\end{proposition}

\begin{proof}
We compute the value of the characteristic polynomial of $G+\al F$ at $x=0$
$$P(0)=-\al\tau_1\tau_2\beta_3-\al^2\tau_1\beta_2\beta_3<0$$
as for two well chosen negative values
$$P(-\al\beta_3)=\al\tau_2\beta_3(\al\beta_3-2\tau_1)$$
and
$$P(-\tau_2)=\al\beta_2(\tau_2^2+\tau_1\tau_2)+\al\beta_3(\tau_2^2-2\tau_1\tau_2)-\al^2(\tau_1+\tau_2)\beta_2\beta_3.$$
We notice that
$$P(-\al\beta_3)>0\quad\Leftrightarrow\quad \al>\frac{2\tau_1}{\beta_3}$$
and
$$P(-\tau_2)>0\quad\Leftrightarrow\quad\al<\frac{\tau_2}{\beta_3}+\frac{\tau_2}{\tau_1\tau_2}\frac{\tau_2-2\tau1}{\beta_2}.$$
Because $\tau_2>2\tau_1,$ we have
$$\frac{\tau_2}{\beta_3}+\frac{\tau_2}{\tau_1\tau_2}\frac{\tau_2-2\tau1}{\beta_2}>\frac{\tau_2}{\beta_3}>\frac{2\tau_1}{\beta_3}$$
so if $P(-\al\beta_3)\leq0,$ then $P(-\tau_2)>0.$
Finally $P$ is a third order polynomial which satisfies the following properties:
(i) it tends to $-\infty$ at $-\infty$; (ii)
it takes positive values for negative $x$; (iii)
it is negative at $x=0$; (iv)
it tends to $+\infty$ at $+\infty$.
Thus $P$ has 3 real roots and it proves the proposition.
\end{proof}

\section{Criterion for Assumption~{\bf (H4)}}\label{ap:tech}

\begin{figure}
\begin{center}
\includegraphics[width = 0.6\linewidth]{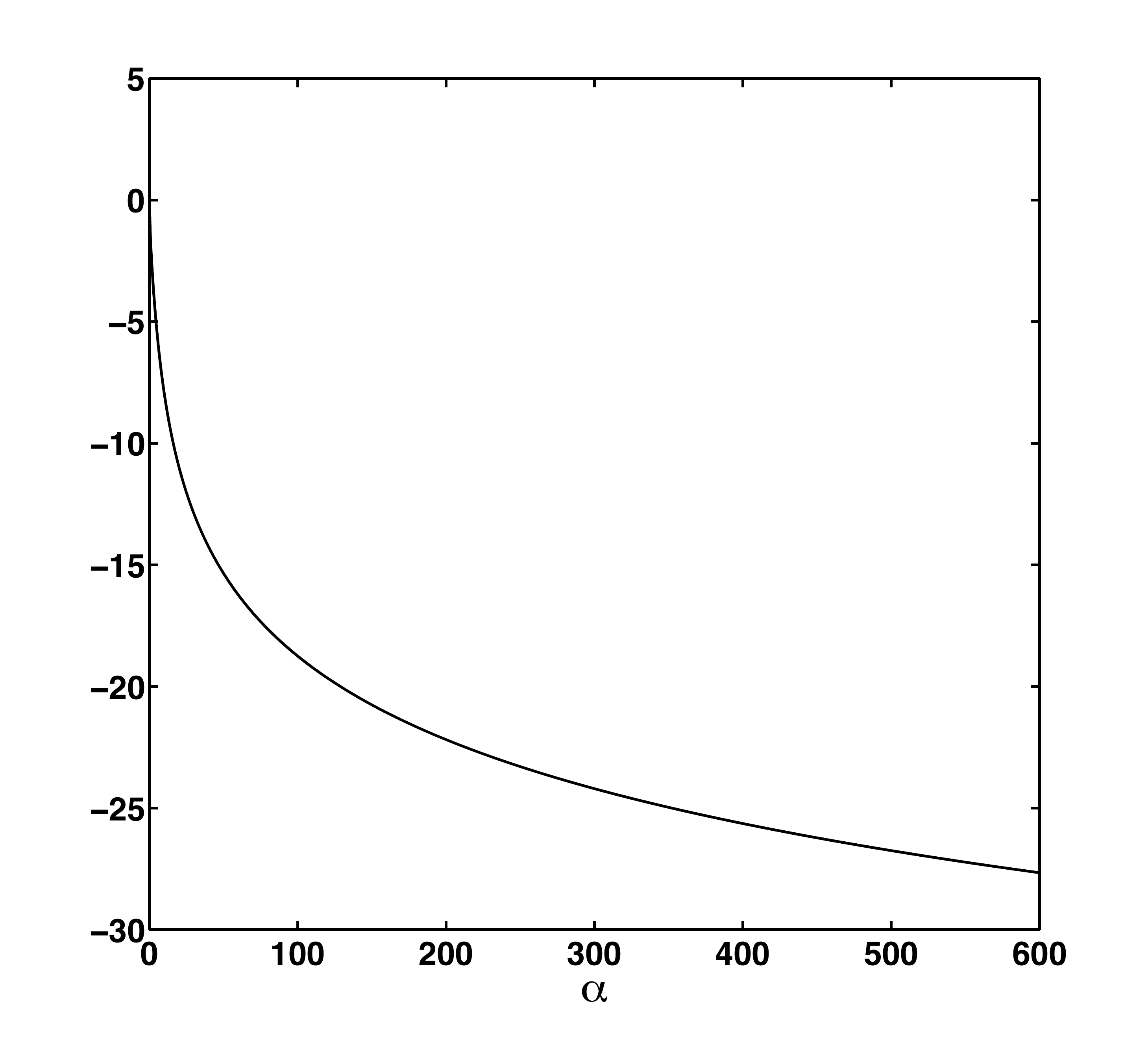}
\caption{Numerical check of the Assumption {\bf (H4)} for the running example, by using the formula in Proposition~\ref{prop:criterion}.
Plot of the quantity $- \la\frac{de_\al}{d\al},\Theta Fe_\al\ra$ as a function of $\alpha$ in log-scale.}
\end{center}
\end{figure}

\begin{proposition}\label{prop:criterion}
Denoting $(e_1,e_2,e_3)$ a basis of eigenvectors for $G+\al F,$ so that $e_1=e_\al,$ and $(\phi_1,\phi_2,\phi_3)$ a basis of dual eigenvectors, we have the formula
$$\la\frac{de_\al}{d\al},\Theta Fe_\al\ra=\frac{\lb_2-\lb_3}{(\lambda_1 -\lb_2)(\lambda_1 -\lb_3)}(\phi_2Fe_1)(\phi_3Fe_1)\la e_2-e_1,\Theta(e_3-e_1)\ra.$$
\end{proposition}

\begin{proof}
Decompose $\frac{de_\al}{d\al}$ and $Fe_\al$ along the basis $(e_1,e_2,e_3):$
$$\frac{de_\al}{d\al}=a_1e_1+a_2e_2+a_3e_3,\qquad Fe_\al=b_1e_1+b_2e_2+b_3e_3.$$
For all $\al$ we have $\la m,e_\al\ra=1$ because $e_\al\in\mS,$ so
$\la m,\frac{de_\al}{d\al}\ra=0.$
Because $m^TF=0,$ we also have
$\la m,Fe_\al\ra=0.$
It leads to the relations
$$a_1+a_2+a_3=b_1+b_2+b_3=0.$$
As a consequence we have
\begin{align*}
\la\frac{de_\al}{d\al},\Theta Fe_\al\ra & = \la a_2(e_2-e_1)+a_3(e_3-e_1),b_2\Theta(e_2-e_1)+b_3\Theta(e_3-e_1)\ra\\
& = a_2b_3\la(e_2-e_1),\Theta(e_3-e_1)\ra+a_3b_2\la(e_3-e_1),\Theta(e_2-e_1)\ra.
\end{align*}
Now we look for a relation between $a_i$ and $b_i$ and we start from
$$(G+\al F)e_\al=\lb(\al)e_\al$$
which gives after differentiation with respect to $\al$
$$(G+\al F)\frac{de_\al}{d\al}+Fe_\al=\lb_\al\frac{de_\al}{d\al}+\lb'(\al)e_\al.$$
Testing against $\phi_i$ with $i\in\{2,3\},$ we obtain
$$\lb_i\phi_i\frac{de_\al}{d\al}+\phi_iFe_1 = \lambda_1 \phi_i\frac{de_\al}{d\al}.$$
Finally we find
$$b_i=\phi_iFe_1=(\lambda_1 -\lb_i)\phi_i\frac{de_\al}{d\al}=(\lambda_1 -\lb_i)a_i$$
and the result of the proposition follows.
\end{proof}

\bibliographystyle{abbrv}
\bibliography{HJB}

%
%
%
%
%
%
%
%
%
%

\end{document}